\documentclass[reqno,english]{amsart}
\usepackage{amsfonts,amsmath,latexsym,verbatim,amscd,mathrsfs,color,array}
\usepackage[colorlinks=true]{hyperref}
\usepackage{amsmath,amssymb,amsthm,amsfonts,graphicx,color}
\usepackage{amssymb}
\usepackage{pdfsync}
\usepackage{epstopdf}
\usepackage{cite}
\usepackage{graphicx}

\newcommand{\bt}{\begin{theorem}}
\newcommand{\et}{\end{theorem}}
\newcommand{\bl}{\begin{lemma}}
\newcommand{\el}{\end{lemma}}
\newcommand{\bd}{\begin{definition}}
\newcommand{\ed}{\end{definition}}
\newcommand{\bc}{\begin{corollary}}
\newcommand{\ec}{\end{corollary}}
\newcommand{\bp}{\begin{proof}}
\newcommand{\ep}{\end{proof}}
\newcommand{\bx}{\begin{example}}
\newcommand{\ex}{\end{example}}
\newcommand{\bi}{\begin{exercise}}
\newcommand{\ei}{\end{exercise}}
\newcommand{\bo}{\begin{prop}}
\newcommand{\eo}{\end{prop}}
\newcommand{\br}{\begin{remark}}
\newcommand{\er}{\end{remark}}
\newcommand{\be}{\begin{equation}}
\newcommand{\ee}{\end{equation}}
\newcommand{\ba}{\begin{align}}
\newcommand{\ea}{\end{align}}
\newcommand{\bn}{\begin{enumerate}}
\newcommand{\en}{\end{enumerate}}
\newcommand{\bg}{\begin{align*}}
\newcommand{\bcs}{\begin{cases}}
\newcommand{\ecs}{\end{cases}}

\newcommand{\bean}{\begin{eqnarray*}}
\newcommand{\eean}{\end{eqnarray*}}

\newtheorem{definition}{Definition}[section]
\newtheorem{theorem}{Theorem}[section]
\newtheorem{lemma}{Lemma}[section]

\newtheorem{prop}{Proposition}[section]
\newtheorem{remark}{Remark}[section]

\numberwithin{equation}{section}

\begin{document}
\title[Fractional heat equation with critical exponent]{Infinite time blow-up for the fractional heat equation with critical exponent}

\author[M. Musso]{Monica Musso}
\address{\noindent
University of Bath,  North Rd., Bath BA2 7AY, UK}
\email{M.Musso@bath.ac.uk}

\author[Y. Sire]{Yannick Sire}
\address{\noindent
Department of Mathematics, Krieger Hall
Johns Hopkins University, Baltimore, MD, USA, 21218}
\email{sire@math.jhu.edu}

\author[J. Wei]{Juncheng Wei}
\address{\noindent
Department of Mathematics,
University of British Columbia, Vancouver, B.C., Canada, V6T 1Z2}
\email{jcwei@math.ubc.ca}

\author[Y. Zheng]{Youquan Zheng}
\address{\noindent School of Mathematics, Tianjin University, Tianjin 300072, P. R. China}
\email{zhengyq@tju.edu.cn}

\author[Y.Zhou]{Yifu Zhou}
\address{\noindent
Department of Mathematics,
University of British Columbia, Vancouver, B.C., Canada, V6T 1Z2}
\email{yfzhou@math.ubc.ca}

\begin{abstract}
We consider positive solutions for the fractional heat equation with critical exponent
\begin{equation*}
\begin{cases}
u_t = -(-\Delta)^{s}u + u^{\frac{n+2s}{n-2s}}\text{ in } \Omega\times (0, \infty),\\
 u = 0\text{ on } (\mathbb{R}^n\setminus \Omega)\times (0, \infty),\\
 u(\cdot, 0) = u_0\text{ in }\mathbb{R}^n,
  \end{cases}
\end{equation*}
where $\Omega$ is a smooth bounded domain in $\mathbb{R}^n$, $n > 4s$, $s\in (0, 1)$, $u:\mathbb{R}^n\times [0, \infty)\to \mathbb{R}$ and $u_0$ is a positive smooth initial datum with $u_0|_{\mathbb{R}^n\setminus \Omega} = 0$. We prove the existence of $u_0$ such that the solution blows up precisely at prescribed distinct points $q_1,\cdots, q_k$ in $\Omega$ as $t\to +\infty$.  The main ingredient of the proofs is a new inner-outer gluing scheme for the fractional parabolic  problems.
\end{abstract}
\maketitle
\section{Introduction}
Let $\Omega$ be a smooth bounded domain in $\mathbb{R}^n$, $n\geq 1$. We consider the fractional heat equation with critical exponent
\begin{equation}\label{e:main}
\begin{cases}
u_t = -(-\Delta)^{s}u + u^{\frac{n+2s}{n-2s}}\text{ in } \Omega\times (0, \infty),\\
 u = 0\text{ on } (\mathbb{R}^n\setminus \Omega)\times (0, \infty),\\
 u(\cdot, 0) = u_0 \text{ in }\mathbb{R}^n,
  \end{cases}
\end{equation}
for a function $u:\mathbb{R}^n\times [0, \infty)\to \mathbb{R}$ and a smooth, positive initial datum $u_0$ satisfying $u_0|_{\mathbb{R}^n\setminus \Omega} = 0$, $s\in (0, 1)$. Here, for any point $x \in \mathbb{R}^n$, the fractional Laplace operator $ (-\Delta)^su(x)$ is defined as
\begin{equation*}
  (-\Delta)^s u(x):= C(n,s) \mbox{P.V.}  \int_{\mathbb{R}^n} \frac{ u(x)- u(y)}{ |x-y|^{n+2s}} dy
  \end{equation*}
with a suitable positive normalizing constant $C(n,s)$. We refer to \cite{DPV} for an introduction to the fractional Laplace operator and to the appendix of \cite{DDDV} for a heuristic physical motivation in nonlocal quantum mechanics of the fractional operator considered here.

Parabolic problems like (\ref{e:main}) and related ones have attracted much attention in recent years, for example, \cite{Barriosvaldinoci2014armawidder}, \cite{BogdanTomaszRyznar2010}, \cite{caffarellichanvasseurJams2011}, \cite{caffrarellifigallijram2013}, \cite{caffarellisoriavazquezjems2013}, \cite{caffarellivazquesarma2011}, \cite{caffarellivasseur2010annals}, \cite{chenkimsongjems2010}, \cite{felsingerkassmancpde2013}, \cite{fernrosoton2016}, \cite{silvestreium2012differentiability}, \cite{silvestreium2012} and the references therein. As in the case of $s =1$, problem (\ref{e:main}) is the formal negative $L^2$-gradient flow of the functional $$E(u) = \frac{1}{2}\int_{\mathbb{R}^n}|(-\Delta)^{\frac{s}{2}}u|^2dx - \frac{n-2s}{2n}\int_{\Omega}|u|^{\frac{2n}{n-2s}}dx$$ in
\begin{equation*}
\begin{aligned}
H^s_0(\Omega) := \bigg\{u\in L^2(\mathbb{R}^n): \int_{\mathbb{R}^n}|(-\Delta)^{\frac{s}{2}}u|^2dx < +\infty &\text{ and } u = 0\\
&\text{ almost everywhere in } \mathbb{R}^n\setminus\Omega\bigg\},
\end{aligned}
\end{equation*}
i.e., $\frac{d}{dt}E(u(\cdot,t)) = -\int_{\mathbb{R}^n}|u_t|^2dx$.
If the function $u(x, t)$ is independent of $t$, (\ref{e:main}) is a semilinear elliptic problem with fractional Laplacian, which has been studied extensively, for instance, in \cite{choikimjfa2014asymtotic} and \cite{servadeivaldinoci2015brezis}.

When $s = 1$, problem (\ref{e:main}) is the classical parabolic equation with critical exponent
\begin{equation}\label{e:localcase}
\begin{cases}
u_t = \Delta u + u^{\frac{n+2}{n-2}}&\text{ in }\Omega\times (0, \infty),\\
u = 0&\text{ on }\partial \Omega\times (0, \infty),\\
u(\cdot, 0) = u_0&\text{ in }\Omega.
\end{cases}
\end{equation}
Many authors are interested in the blow-up phenomenon of (\ref{e:localcase}), for example, \cite{cortazar2016green}, \cite{delmussoweitype2}, \cite{delmussowei3d}, \cite{fujitajfsuts1966blowup}, \cite{matanomerlejfa2011threshold}, \cite{merlezaagduke1997stability}, \cite{quittnersouplet2007superlinear}, \cite{schweyerjfa2012typeii}.
In \cite{cortazar2016green}, Cortazar, del Pino and Musso obtained the following result.
Suppose $n>4$, let $\hat{G}(x, y)$ be the Green's function of $-\Delta$ in $\Omega$ with Dirichlet boundary value and $\hat{H}(x,y)$ be its regular part.
Given $k$ distinct points $q_1,\cdots, q_k$ in $\Omega$ such that the matrix
\begin{eqnarray*}
\hat{\mathcal{G}}(q) = \left[
\begin{matrix}
\hat{H}(q_1, q_1)&-\hat{G}(q_1, q_2)&\cdots & -\hat{G}(q_1, q_k)\\
-\hat{G}(q_2, q_1)&\hat{H}(q_2, q_2)&\cdots &-\hat{G}(q_2, q_k)\\
\vdots&\vdots&\ddots&\vdots\\
-\hat{G}(q_k,q_1)&-\hat{G}(q_k, q_2)&\cdots&\hat{H}(q_k, q_k)
\end{matrix}
\right]
\end{eqnarray*}
is positive definite, they proved the existence of an initial datum $u_0$ and smooth parameter functions $\xi_j(t)\to q_j$, $0<\mu_j(t)\to 0$, as $t\to +\infty$, $j = 1, \cdots, k$, such that there exists an infinite time blow-up solution $u_q$ of (\ref{e:localcase}) which has the approximate form
\begin{equation*}
u_q \approx\sum_{j= 1}^k\alpha_{n}\left(\frac{\mu_j(t)}{\mu_j^2(t) + |x-\xi_j(t)|^2}\right)^{\frac{n-2}{2}}
\end{equation*}
with $\mu_j(t) = \beta_jt^{-\frac{1}{n-4}}(1+o(1))$ for certain positive constants $\beta_j$.
The aim of this paper is to show that this phenomenon also occurs in problem (\ref{e:main}).
Our starting point is the positive entire solutions of the equation
\begin{equation*}
-(-\Delta)^sU + U^{\frac{n+2s}{n-2s}} = 0\text{ in }\mathbb{R}^n,
\end{equation*}
which are given by the bubbles
\begin{equation}\label{umuxi}
U_{\mu,\xi}(x) = \mu^{-\frac{n-2s}{2}}U_0\left(\frac{x-\xi}{\mu}\right) = \alpha_{n,s}\left(\frac{\mu}{\mu^2 + |x-\xi|^2}\right)^{\frac{n-2s}{2}},
\end{equation}
where
\begin{equation*}
U_0(y) = \alpha_{n,s}\left(\frac{1}{1+|y|^2}\right)^{\frac{n-2s}{2}}
\end{equation*}
and $\alpha_{n,s}$ is a constant depending only on $n$ and $s$, see, \cite{ChenLiOucpam2006classification} and \cite{Lijems2004remarkconformally}.
Let $G(x,y)$ be the Green's function for the following nonlocal problem
\begin{equation*}
\begin{cases}
(-\Delta)^sG(x,y) = c(n,s)\delta(x-y)&\text{ in }\Omega, \\
G(\cdot, y) = 0&\text{ in }\mathbb{R}^n\setminus\Omega,
\end{cases}
\end{equation*}
where $\delta(x)$ denotes the Dirac measure at the origin and $c(n,s)$ satisfies
\begin{equation*}
(-\Delta)^s\Gamma(x) = c(n,s)\delta(x),\quad \Gamma(x) = \frac{\alpha_{n,s}}{|x|^{n-2s}}.
\end{equation*}
The regular part of $G(x,y)$ is denoted by $H(x,y)$, namely $H(x,y)$ solves the following problem
\begin{equation*}
\begin{cases}
(-\Delta)^sH(x,y) = 0&\text{ in }\Omega,\\
H(\cdot, y) = \Gamma(\cdot - y)&\text{ in }\mathbb{R}^n\setminus\Omega.
\end{cases}
\end{equation*}
Let $q = (q_1,\cdots, q_k)$ be the collection of $k$ distinct points in $\Omega$ and define
\begin{eqnarray}\label{e:matrix}
\mathcal{G}(q) := \left[
\begin{matrix}
H(q_1, q_1)&-G(q_1, q_2)&\cdots & -G(q_1, q_k)\\
-G(q_2, q_1)&H(q_2, q_2)&\cdots &-G(q_2, q_k)\\
\vdots&\vdots&\ddots&\vdots\\
-G(q_k,q_1)&-G(q_k, q_2)&\cdots&H(q_k, q_k)
\end{matrix}
\right].
\end{eqnarray}
Our main result is stated as follows.
\begin{theorem}\label{t:main}
Assume $n>4s$, $s\in (0, 1)$ and $q_1,\cdots, q_k$ are distinct points in $\Omega$ such that the matrix (\ref{e:matrix}) is positive definite. Then there exist $u_0$ and smooth parameter functions $\xi_j(t)\to q_j$, $0<\mu_j(t)\to 0$, as $t\to +\infty$, $j = 1, \cdots, k$, such that there exists solution $u_q$ to problem (\ref{e:main}) with the form
\begin{equation*}
u_q =\sum_{j= 1}^k\alpha_{n,s}\left(\frac{\mu_j(t)}{\mu_j^2(t) + |x-\xi_j(t)|^2}\right)^{\frac{n-2s}{2}}-\mu_j^{\frac{n-2s}{2}}(t)H(x, q_j)+\mu_j^{\frac{n-2s}{2}}(t)\varphi(x, t),
\end{equation*}
where $\varphi(x, t)$ is bounded satisfying $\varphi(x, t)\to 0$ as $t\to +\infty$, uniformly away from $q_j$. Furthermore,
there exists a submanifold $\mathcal{M}$ with codimension $k$ in $X:=\{u\in C^1(\mathbb{R}^n):u|_{\mathbb{R}^n\setminus\Omega} = 0\}$ containing $u_q(x,0)$ such that, if $u_0$ is a small perturbation of $u_q(x,0)$ in $\mathcal{M}$, then the solution $u(x, t)$ of (\ref{e:main}) still has the form
\begin{equation*}
u(x,t) =\sum_{j= 1}^k\alpha_{n,s}\left(\frac{\tilde{\mu}_j(t)}{\tilde{\mu}_j^2(t) + |x-\tilde{\xi}_j(t)|^2}\right)^{\frac{n-2s}{2}}-\tilde{\mu}_j^{\frac{n-2s}{2}}(t)H(x, \tilde{q}_j)+\tilde{\mu}_j^{\frac{n-2s}{2}}(t)\tilde{\varphi}(x, t),
\end{equation*}
where the point $\tilde{q}_j$ is close to $q_j$ for $j = 1,\cdots,k$.
\end{theorem}

In order to prove this theorem, we shall develop a new   {\it inner-outer gluing scheme}  for fractional parabolic problems.  It is well-known that gluing methods have been proven to be  very useful in singular perturbation elliptic problems, for example, \cite{dkw2007concentration}, \cite{delwei2011Degiorgi}, \cite{delkowalczykweijdg2013entire}. This method has also been applied in various parabolic flows recently, such as the infinite time blow-up for critical nonlinear heat equation \cite{cortazar2016green}, \cite{delmussowei3d} and half-harmonic map flow \cite{sireweizhenghalf}, the singularity formation for two dimensional harmonic map flow \cite{davila2017singularity}, finite time blow-up for critical nonlinear heat equation \cite{delmussoweitype2}, type II ancient solution for Yamabe flow \cite{del2012type}. %and Allen-Cahn flow \cite{del2017ancientarxive}, \cite{del2017ancient}.

When dealing with parabolic problems, a crucial step in the scheme is to find a solution of the linearized parabolic equation around the bubble with sufficiently fast decay. However, it seems that the local argument in \cite{cortazar2016green} for the classical critical heat equation does not work in the fractional case. Inspired by Lemma 4.5 of \cite{davila2017singularity} and the linear theory of \cite{sireweizhenghalf}, we will use a blow-up argument based on the nondegeneracy of bubbles and a removable singularity property for the corresponding limit equations. (See Section 5 below.)

As mentioned, the case of fractional parabolic problems is much more intricate. For the semilinear equation, Sugitani \cite{sugitani} proved non-existence below the Fujita exponent $p_*=1+\frac{2s}{n}$. The case of global existence above the exponent remains open since all the known techniques fail  in this case. Related to a similar question, the paper \cite{ishige} provides an optimal initial trace theory (see also \cite{BSV} for the case of the homogeneous fractional heat equation). As far as blow-up is concerned, a theory in the spirit of the one developed for instance by Giga and Kohn \cite{GK} , is missing. A crucial step in these approaches is to exhibit a monotone quantity. For the operator $\partial_t +(-\Delta)^s$, such a quantity is missing. On the other hand, as far as nonlocal operators are concerned, the fractional power of the heat equation, i.e. $(\partial_t -\Delta)^s$, on the other hand exhibits monotonicity (see for instance \cite{BG}). Notice that the latter operator has the same stationary solutions as the former one.  

The proof of Theorem \ref{t:main} is rather long. We outline the proof and point out key arguments here. To explain the idea, we assume that $k = 1$ in the rest of this section.

\noindent
{\bf Step 1. Construction of approximation}.
Our aim is to find a solution $u(x, t)$ in the following approximate form
\begin{equation*}
u(x, t) \approx U_{\mu(t), \xi(t)}(x)
\end{equation*}
with $\xi(t)\to q$, $\mu(t)\to 0$ as $t\to\infty$ and $U_{\mu(t), \xi(t)}(x)$ is defined in (\ref{umuxi}). Denote the error operator as
\begin{equation*}
S(u):= -u_t-(-\Delta)^su + u^p,
\end{equation*}
where $p = \frac{n+2s}{n-2s}$ . Then the error of $U_{\mu(t), \xi(t)}(x)$ is
\begin{eqnarray*}
\begin{aligned}
S(U_{\mu(t), \xi(t)}) &=  \mu^{-\frac{n-2s}{2}-1}\dot{\mu} Z_{n+1}(y)  + \mu^{-\frac{n-2s}{2}-1}\dot{\xi}\cdot \nabla U(y).
\end{aligned}
\end{eqnarray*}
Here $y = \frac{x-\xi(t)}{\mu(t)}$. It turns out that the terms $\mu^{-\frac{n-2s}{2}-1}\dot{\mu} Z_{n+1}(y)$ and $\mu^{-\frac{n-2s}{2}-1}\dot{\xi}\cdot \nabla U(y)$ do not have enough decay to perform the gluing method we shall use. (For $s=1$, this is enough.)  So we add a nonlocal term $\Phi^*(x, t) = \Phi^0(x, t) + \Phi^1(x, t)$ to cancel them out at main order. Since $u = 0$ in $\mathbb{R}^n\setminus\Omega$, a better approximation than $U_{\mu(t), \xi(t)}(x)$ is
\begin{equation*}
u_{\mu,\xi}(x, t) = U_{\mu,\xi}(x) + \mu^{\frac{n-2s}{2}}\Phi^*(x, t) -  \mu^{\frac{n-2s}{2}} H(x, q).
\end{equation*}
The error of $u_{\mu,\xi}$ can be computed as
\begin{equation*}
\mu^{\frac{n+2s}{2}}S(u_{\mu, \xi}) \approx \mu E_{0} + \mu E_{1}
\end{equation*}
with
\begin{equation*}
\begin{aligned}
E_{0}&= pU(y)^{p-1}\left(-\mu^{n-2s-1}H(q, q)\right)+ pU(y)^{p-1}\mu^{n-2s-1}\Phi^0(q, t)\\
& \quad +\mu^{2s-2}\dot{\mu}\left(Z_{n+1}(y)
+ \frac{n-2s}{2}\alpha_{n, s}\frac{1}{\left(1+|y|^2\right)^{\frac{n-2s}{2}}}\right)
\end{aligned}
\end{equation*}
and
\begin{equation*}
\begin{aligned}
E_{1} = pU(y)^{p-1}\left(-\mu^{n-2s}\nabla H(q, q)\right)\cdot y &+ pU(y)^{p-1}\mu^{n-2s-1}\Phi^1(q, t)\\
&+ \alpha_{n, s}(n-2s)\mu^{2s-2}\frac{\dot{\xi}\cdot y}{\left(1+|y|^2\right)^{\frac{n-2s}{2}+1}}.
\end{aligned}
\end{equation*}
We shall look for solutions with the form
\begin{equation*}
u(x, t) = u_{\mu, \xi} + \tilde{\phi}(x, t),
\end{equation*}
where
\begin{equation*}
\tilde{\phi}(x, t) = \mu^{-\frac{n-2s}{2}}\phi\left(\frac{x-\xi}{\mu}, t\right).
\end{equation*}
By $S(u)=S(u_{\mu, \xi} + \tilde{\phi}(x, t))=0$, the equation for $\phi(y, t)$ is
\begin{equation*}
-(-\Delta)^s_y\phi + pU(y)^{p-1}\phi + \mu^{\frac{n+2s}{2}}S(u_{\mu, \xi}) + A[\phi]=0
\end{equation*}
for a small term $A[\phi]$. Note that in the expansion of $\mu^{\frac{n+2s}{2}}S(u_{\mu, \xi})$, the largest term is $\mu E_0$, so $\phi(y, t)$ should equal a solution $\phi_{0}(y, t)$ of the following elliptic type equation at main order
\begin{equation}\label{e0:elliptic}
-(-\Delta)^s_y\phi_{0} + pU(y)^{p-1}\phi_{0} = -\mu_{0}E_{0}\text{ in }\mathbb{R}^n,~~\phi_0(y, t)\to 0\text{ as }|y|\to \infty.
\end{equation}
Equation (\ref{e0:elliptic}) is a special case of
\begin{equation}\label{e2:32111}
L_0[\psi]:= -(-\Delta)^s_y\psi + pU(y)^{p-1}\psi = h(y)\text{ in }\mathbb{R}^n,~~\psi(y)\to 0\text{ as }|y|\to \infty.
\end{equation}
It is well known that every bounded solution of $L_0[\psi] = 0$ in $\mathbb{R}^n$ is the linear combination of the functions
$$Z_1,\cdots, Z_{n+1}$$
where
\begin{equation*}
Z_i(y):= \frac{\partial U}{\partial y_i}(y),\quad i = 1,\cdots, n, \quad Z_{n+1}(y):=\frac{n-2s}{2}U(y) + y\cdot\nabla U(y).
\end{equation*}
The above non-degeneracy result can be found in \cite{DelPinoSirePAMS}.
Furthermore, problem (\ref{e2:32111}) is solvable for $h(y) = O(|y|^{-m})$, $m > 2s$, if it holds that
\begin{equation*}
\int_{\mathbb{R}^n}h(y)Z_i(y)dy = 0\quad\text{for}\quad i = 1,\cdots, n+1
\end{equation*}
By choosing $\bar{\mu}_{0} = b\mu_0(t)$ for suitable positive constant $b$ and $\xi_0 = q$, the solvability conditions
\begin{equation}\label{e0:solvability}
\int_{\mathbb{R}^n}\mu_{0}E_{0}(y, t)Z_{i}(y)dy = 0,\quad i = 1,\cdots, n+1
\end{equation}
can be achieved at main order. Here $\mu_0(t) = c_{n, s}t^{-\frac{1}{n-4s}}$ for some constant $c_{n, s}$.
Under the solvability condition (\ref{e0:solvability}), (\ref{e0:elliptic}) has a solution $\Phi(y, t)$, which leads to the following corrected approximation
\begin{equation*}
u^*_{\mu, \xi}(x, t) = u_{\mu, \xi}(x, t) + \tilde{\Phi}(x, t),
\end{equation*}
where
\begin{equation*}
\tilde{\Phi}(x, t) = \mu^{-\frac{n-2s}{2}}\Phi\left(\frac{x-\xi}{\mu}, t\right)
\end{equation*}
and $\mu(t) = b\mu_0(t) + \lambda(t)$.
Finally, we use the ansatz
\begin{equation*}
u = u^*_{\mu,\xi} + \tilde{\phi}.
\end{equation*}
We shall show the details and the general case $k\geq 1$ in Section 2.

\medskip

\noindent
{\bf Step 2. The inner-outer gluing procedure}.
Denote
\begin{equation*}
\tilde{\phi}(x, t) = \psi(x, t) + \phi^{in}(x, t),\quad\text{where}\quad\phi^{in}(x, t): = \eta_{R}\tilde{\phi}(x, t)
\end{equation*}
with
\begin{equation*}
\tilde{\phi}(x, t): = (b\mu_{0})^{-\frac{n-2s}{2}}\phi\left(\frac{x-\xi}{b\mu_{0}}, t\right)
\end{equation*}
and
\begin{equation*}
\eta_{R}(x, t) = \eta\left(\frac{|x-\xi|}{R\mu_{0}}\right).
\end{equation*}
The cut-off function $\eta(\tau)$ satisfies $\eta(\tau) = 1$ for $0\leq \tau < 1$ and $\eta(\tau)= 0$ for $\tau > 2$. The number $R$ is independent of $t$ and fixed sufficiently large. In terms of $\tilde{\phi}$, problem (\ref{e:main}) can be expressed as
\begin{equation}\label{0417}
\begin{cases}
\partial_t\tilde{\phi} = -(-\Delta)^{s}\phi + p(u^*_{\mu, \xi})^{p-1}\tilde{\phi} + \tilde{N}(\tilde{\phi}) + S(\mu^*_{\mu, \xi}), &\quad \text{ in }\Omega\times (t_0, \infty),\\
\tilde{\phi} = -u^*_{\mu, \xi}, &\quad \text{ in }(\mathbb{R}^n\setminus \Omega)\times (t_0, \infty).
\end{cases}
\end{equation}
Let
\begin{equation*}
V_{\mu, \xi} = p\left((u^*_{\mu, \xi})^{p-1} - \left[\mu^{-\frac{n-2s}{2}}U\left(\frac{x-\xi}{\mu}\right)\right]^{p-1}\right)\eta_{R} + p(1-\eta_{R})(u^*_{\mu, \xi})^{p-1}.
\end{equation*}
Then $\tilde{\phi}$ solves problem (\ref{0417}) if $\psi$ and $\phi$ satisfy the following two coupled equations respectively,
\begin{equation}\label{e0:outerproblem}
\begin{cases}
\partial_t\psi =
-(-\Delta)^{s}\psi + V_{\mu, \xi}\psi+\tilde{\phi}(-(-\Delta)^{s})\eta_{R} + \cdots,&\quad\text{ in }\Omega\times (t_0, \infty),\\
\psi = -u^*_{\mu, \xi}, &\quad\text{ in }(\mathbb{R}^n\setminus\Omega)\times (t_0,\infty)
\end{cases}
\end{equation}
and
\begin{equation}\label{0417'}
\begin{aligned}
\mu_{0}^{2s}\partial_t\phi = &-(-\Delta)^s_y\phi + pU^{p-1}(y)\phi\\ &+ \Bigg\{p\mu_{0}^{\frac{n-2s}{2}}\frac{\mu_{0}^{2s}}{\mu^{2s}}U^{p-1}\left(\frac{\mu_{0}}{\mu}y\right)\psi(\xi + \mu_{0}y, t)+\cdots \Bigg\}\chi_{B_{2R}(0)}(y), y\in \mathbb{R}^n.
\end{aligned}
\end{equation}
\eqref{e0:outerproblem} is the so-called $outer~problem$ and \eqref{0417'} is the $inner~ problem$.  Note that the inner problem is solved in the {\em whole space} with error supported in $B_{2R} (0)$. See Section 3 for details.

\noindent
{\bf Step 3. The outer problem}. For a fixed $a > 2s$, we solve the outer problem (\ref{e0:outerproblem}) for $\psi$ under the initial condition $\psi(\cdot,t_0) = \psi_0$ in $\mathbb{R}^n$. Suppose
\begin{equation}\label{e0:decay}
(1+|y|)|\nabla_y \phi(y, t)|\chi_{B_{2R}(0)}(y) + |\phi(y, t)|\lesssim t_0^{-\varepsilon}\frac{\mu_0^{n-2s+\sigma}(t)}{1+|y|^a}
\end{equation}
holds for a small constant $\sigma >0$ and small $\varepsilon > 0$. Using the super-sub solution method, we solve (\ref{e0:outerproblem})
and obtain the existence of a unique solution $\psi = \Psi[\lambda, \xi, \dot{\lambda}, \dot{\xi}, \phi]$ satisfying
\begin{equation*}
|\psi(x, t)|\lesssim \frac{t_0^{-\varepsilon}}{R^{a-2s}}\frac{\mu_0^{\frac{n-2s}{2}+\sigma}(t)}{1+|y|^{a-2s}} + e^{-\delta(t-t_0)}\|\psi_0\|_{L^{\infty}(\mathbb{R}^n)}
\end{equation*}
and
\begin{equation*}
[\psi(x, t)]_{\eta, B_{\mu R}(\xi)}\lesssim \frac{t_0^{-\varepsilon}}{R^{a-2s}}\frac{\mu^{-\eta}\mu_0^{\frac{n-2s}{2}+\sigma}(t)}{1+|y|^{a-2s+\eta}}\text{ for } |y|\leq 2 R,
\end{equation*}
where  $y=\frac{x-\xi}{\mu_{0}}$.
This is the content of Section 4.

After substituting $\psi = \Psi[\lambda,\xi,\dot{\lambda},\dot{\xi},\phi]$ into the inner problem (\ref{0417'}) and using the change of variables $\frac{dt}{d\tau}=\mu_0^{2s}(t)$, the full problem is reduced to the solvability of the following nonlinear nonlocal equation
\begin{equation}\label{e05:6}
\begin{cases}
\partial_\tau\phi = -(-\Delta)^s_y\phi + pU^{p-1}(y)\phi + H[\lambda,\xi,\dot{\lambda},\dot{\xi},\phi](y,t(\tau)), &\quad y\in \mathbb{R}^n, \tau\geq \tau_0,\\
\phi(y,\tau_0) = e_{0}Z_0(y), &\quad  y\in \mathbb{R}^n,
\end{cases}
\end{equation}
for some constant $e_{0}$, and $Z_0$ is the bounded eigenfunction corresponding to the only negative eigenvalue $\lambda_0$ to the following eigenvalue problem
$$-(-\Delta)^s \phi+pU^{p-1}\phi+\lambda\phi=0,~\phi\in L^{\infty}(\mathbb{R}^n).$$

\noindent
{\bf Step 4. Linear theory for (\ref{e05:6})}. To solve the problem (\ref{e05:6}), we first consider the following linear parabolic problem
\begin{equation}\label{0417''}
\begin{cases}
\partial_\tau\phi = -(-\Delta)^s\phi + pU^{p-1}(y)\phi + h(y,\tau), &\quad y\in \mathbb{R}^n,~ \tau\geq \tau_0,\\
\phi(y,\tau_0) = e_{0}Z_0(y), &\quad y\in \mathbb{R}^n.
\end{cases}
\end{equation}
Assuming $h(\cdot, \tau)$ is supported in $B_{2R}(0)$ for any $\tau \geq \tau_0$, $\|h\|_{2s+a,\nu, \eta} < +\infty$ and
\begin{equation*}
\int_{B_{2R}(0)}h(y,\tau)Z_j(y)dy = 0\quad\text{for all}\quad\tau\in (\tau_0,\infty),\quad j = 1,\cdots, n+1,
\end{equation*}
we prove the existence of a fast-decaying solution $\phi = \phi[h](y, \tau)$ and $e_0 = e_0[h](\tau)$ ($\tau\in (\tau_0,+\infty),y\in \mathbb{R}^n$) solving problem (\ref{0417''}). In addition, the following estimates hold,
\begin{equation*}\label{e5:100}
\begin{aligned}
(1+|y|)|\nabla_y \phi(y, \tau)|\chi_{B_{2R}(0)}(y)&+ |\phi(y,\tau)|\\
&\lesssim \tau^{-\nu}(1+|y|)^{-a}\|h\|_{2s+a,\nu, \eta}, \tau\in (\tau_0,+\infty),y\in \mathbb{R}^n
\end{aligned}
\end{equation*}
and
\begin{equation*}\label{e5:101}
|e_0[h]|\lesssim \|h\|_{2s+a,\nu, \eta}.
\end{equation*}
It seems that the linear theory in \cite{cortazar2016green} does not work in the fractional case, instead, we will use the blow up argument similar to  \cite{davila2017singularity}.  Here we need the technical assumption $a > 2s$ to ensure the integrability. This is the reason why we add two nonlocal terms in {\bf Step 1} and there is a term $\frac{1}{R^{a-2s}}$ in the estimation of $\psi$, see Section 5.1.

\noindent
{\bf Step 5. The solvability condition for (\ref{e05:6})}.
From {\bf Step 4}, we see that problem (\ref{e05:6}) is solvable for functions $\phi$ satisfying (\ref{e0:decay}), provided $\xi$ and $\lambda$ are chosen such that
\begin{equation*}
\int_{B_{2R}}H[\lambda,\xi,\dot{\lambda},\dot{\xi},\phi](y,t(\tau))Z_l(y)dy = 0,\,\, \text{for all }\tau\geq \tau_0, l = 1,2,\cdots,n+1.
\end{equation*}
By the orthogonality conditions above, our original problem is reduced to a nonlinear nonlocal system of ODEs for $\lambda$ and $\xi$, which is achieved in Section 5.2.

\noindent
{\bf Step 6. The inner problem: gluing}. We finally solve the nonlinear nonlocal problem (\ref{e05:6}) based on the linear theory for (\ref{0417''}) and the Contraction Mapping Theorem. See Section 6 for details.

\section{Construction of the approximation}
\subsection{Setting up the problem.}
Let $t_0 > 0$. We consider the following evolution problem
\begin{equation}\label{e3:1}
\begin{cases}
u_t = -(-\Delta)^{s}u + u^{\frac{n+2s}{n-2s}}&\text{  in  }\Omega\times (t_0, \infty),\\
u = 0&\text{  in  }(\mathbb{R}^n\setminus \Omega)\times (t_0, \infty),
\end{cases}
\end{equation}
which provides a solution $u(x, t) = u(x, t-t_0)$ to (\ref{e:main}). Given $k$ points $q_1, \cdots, q_k\in\mathbb{R}^n$, our aim is to find a solution $u(x, t)$ of (\ref{e3:1}) in the following approximate form
\begin{equation*}
u(x, t) \approx \sum_{j=1}^kU_{\mu_j(t), \xi_j(t)}(x)
\end{equation*}
with $\xi_j(t)\to q_j$, $\mu_j(t)\to 0$ as $t\to\infty$ for all $j = 1,\cdots, k$ and $U_{\mu_j(t), \xi_j(t)}(x)$ is defined in (\ref{umuxi}). Denote the error operator as
\begin{equation*}
S(u):= -u_t-(-\Delta)^su + u^p,
\end{equation*}
where $p = \frac{n+2s}{n-2s}$ . Then the error of $U_{\mu_j(t), \xi_j(t)}(x)$ is
\begin{eqnarray*}
\begin{aligned}
S(U_{\mu_j(t), \xi_j(t)}) &= -\frac{\partial }{\partial t}U_{\mu_j,\xi_j}(x) = \mu^{-\frac{n-2s}{2}}_j\left(\frac{\dot{\mu_j}}{\mu_j} Z_{n+1}(y_j)  + \frac{\dot{\xi_j}}{\mu_j}\cdot \nabla U(y_j)\right)\\
& = \mu^{-\frac{n-2s}{2}-1}_j\dot{\mu}_j Z_{n+1}(y_j)  + \mu^{-\frac{n-2s}{2}-1}_j\dot{\xi}_j\cdot \nabla U(y_j).
\end{aligned}
\end{eqnarray*}
Here $y_j = \frac{x-\xi_j(t)}{\mu_j(t)}$. It turns out that the terms $\mu^{-\frac{n-2s}{2}-1}_j\dot{\mu}_j Z_{n+1}(y_j)$ and $\mu^{-\frac{n-2s}{2}-1}_j\dot{\xi}_j\cdot \nabla U(y_j)$ do not have enough decay to perform the gluing method, so we add nonlocal terms to cancel them out at main order. Note that the main order of
$$
Z_{n+1}(y) = \frac{n-2s}{2}\alpha_{n, s}\frac{1-|y|^2}{\left(1+|y|^2\right)^{\frac{n-2s}{2}+1}}
$$
is
$$
-\frac{n-2s}{2}\alpha_{n, s}\frac{1}{\left(1+|y|^2\right)^{\frac{n-2s}{2}}}.
$$
Therefore, we consider the equation
\begin{equation}\label{2018320}
-\varphi_t -(-\Delta)^s\varphi - \frac{n-2s}{2}\alpha_{n, s}\frac{\dot{\mu}_j}{\mu_j}\frac{\mu^{-(n-2s)}_j}{\left(1+\left|\frac{x-\xi_j}{\mu_j}\right|^2\right)^{\frac{n-2s}{2}}} = 0 \text{ in }\mathbb{R}^n\times (t_0, +\infty).
\end{equation}
Then
\begin{equation*}
\Phi^0_j(x,t) = -\int_{t_0}^{t}\int_{\mathbb{R}^n}p(t-\tilde{s}, x-y)\frac{\dot{\mu}_j(\tilde{s})}{\mu_j(\tilde{s})}\frac{\mu^{-(n-2s)}_j(\tilde{s})}{\left(1+\left|\frac{y-\xi_j(\tilde{s})}
{\mu_j(\tilde{s})}\right|^2\right)^{\frac{n-2s}{2}}}dyd\tilde{s}
\end{equation*}
is a bounded solution for (\ref{2018320}). Here the function $p(t, x)$ is the heat kernel for the fractional heat operator $-\frac{\partial }{\partial t} - (-\Delta)^s$, see \cite{CabreRoquejoffreCMP} for its definition and properties.
Using the super-sub solution argument (see Lemma \ref{l4:lemma4.1}), it is easy to see that $\Phi^0_j(x, t)$ satisfies the estimate $\Phi^0_j(x, t) \sim  \frac{\dot{\mu}_j}{\mu_j}\frac{\mu^{-n+4s}_j}{1+|y_j|^{n-4s}}$.

Similarly, for $y_j = \frac{x-\xi_j}{\mu_j}$, we consider the equation
\begin{equation}\label{20183201}
-\varphi_t -(-\Delta)^s\varphi + \alpha_{n, s}(n-2s)\mu^{-(n-2s)-1}_j\frac{|y_j|^2}{\left(1+|y_j|^2\right)^{\frac{n-2s}{2}+2}}\dot{\xi}_j\cdot y_j = 0 \text{ in }\mathbb{R}^n\times (t_0, +\infty).
\end{equation}
Its solution defined by
\begin{equation*}
\begin{aligned}
\Phi^1_j(x,t) = -\int_{t_0}^{t}\int_{\mathbb{R}^n}p(t-\tilde{s}, x-y)\mu^{-(n-2s)}_j(\tilde{s})&\frac{\dot{\xi}_j(\tilde{s})\cdot \frac{y-\xi_j(\tilde{s})}{\mu_j(\tilde{s})}}{\mu_j(\tilde{s})}\times\\
&\quad\quad\quad\frac{\left|\frac{y-\xi_j(\tilde{s})}{\mu_j(\tilde{s})}\right|^2}
{\left(1+\left|\frac{y-\xi_j(\tilde{s})}{\mu_j(\tilde{s})}\right|^2\right)^{\frac{n-2s}{2}+2}}dyd\tilde{s}
\end{aligned}
\end{equation*} satisfies the estimate
$\Phi^1_j(x, t) \sim  \frac{|\dot{\xi}_j|}{\mu_j}\frac{\mu^{-n+4s}_j}{1+|y_j|^{n-4s+1}}$. Define $\Phi_j^*(x, t) = \Phi^0_{j}(x, t) + \Phi^1_j(x, t)$.
Since $u = 0$ in $\mathbb{R}^n\setminus\Omega$, a better approximation than $\sum_{j=1}^kU_{\mu_j(t), \xi_j(t)}(x)$ is
\begin{equation}\label{e2:3}
u_{\mu,\xi}(x, t) = \sum_{j=1}^ku_j(x, t)~\mbox{with}~ u_j(x, t):=U_{\mu_j,\xi_j}(x) + \mu_j^{\frac{n-2s}{2}}\Phi_j^*(x, t) -  \mu_j^{\frac{n-2s}{2}} H(x, q_j).
\end{equation}
The error of $u_{\mu,\xi}$ can be computed as
\begin{equation}\label{e2:3333}
S(u_{\mu, \xi}) = -\sum_{i=1}^k\partial_tu_i + \left(\sum_{i=1}^ku_i\right)^p - \sum_{i=1}^kU^p_{\mu_i,\xi_i} - \sum_{i=1}^k\mu_i^{\frac{n-2s}{2}}(-\Delta)^s\Phi_i^*(x, t).
\end{equation}
\subsection{The error $S(u_{\mu, \xi})$.}
Near a given point $q_j$, we have the following estimate.
\begin{lemma}\label{l2.1}
Consider the region $|x-q_j|\leq \frac{1}{2}\min_{i\neq l}|q_i - q_l|$ for a fixed index $j$, denote $x = \xi_j + \mu_j y_j$, then we have
\begin{equation*}
S(u_{\mu, \xi}) = \mu_j^{-\frac{n+2s}{2}}(\mu_j E_{0j} + \mu_jE_{1j} + \mathcal{R}_j)
\end{equation*}
with
\begin{equation*}
\begin{aligned}
E_{0j}&= pU(y_j)^{p-1}\left[-\mu_j^{n-2s-1}H(q_j, q_j) + \sum_{i\neq j}\mu_j^{\frac{n-2s}{2}-1}\mu_i^{\frac{n-2s}{2}}G(q_j, q_i)\right]\\
&\quad + pU(y_j)^{p-1}\mu_j^{n-2s-1}\Phi^0_j(q_j, t)\\
&\quad +\mu_j^{2s-2}\dot{\mu}_j\left(Z_{n+1}(y_j) + \frac{n-2s}{2}\alpha_{n, s}\frac{1}{\left(1+|y_j|^2\right)^{\frac{n-2s}{2}}}\right),
\end{aligned}
\end{equation*}
\begin{equation*}
\begin{aligned}
E_{1j} &= pU(y_j)^{p-1}\left[-\mu_j^{n-2s}\nabla H(q_j, q_j)+ \sum_{i\neq j}\mu_j^{\frac{n-2s}{2}}\mu_i^{\frac{n-2s}{2}}\nabla G(q_j, q_i)\right]\cdot y_j\\
&\quad + pU(y_j)^{p-1}\mu_j^{n-2s-1}\Phi^1_j(q_j, t) + \alpha_{n, s}(n-2s)\mu^{2s-2}_j\frac{\dot{\xi}_j\cdot y_j}{\left(1+|y_j|^2\right)^{\frac{n-2s}{2}+1}}
\end{aligned}
\end{equation*}
and
\begin{equation*}
\mathcal{R}_j = \frac{\mu_0^{n-2s+2}g}{1+|y_j|^{4s-2}}+\frac{\mu_0^{n-2s}\vec{g}}{1+|y_j|^{4s}}\cdot(\xi_j-q_j) + \mu_0^{n+2s}f + \mu_0^{n-1}\sum_{i=1}^k\dot{\mu}_if_i + \mu_0^{n}\sum_{i=1}^k\dot{\xi}_i\cdot\vec{f}_i,
\end{equation*}
where $f$, $f_i$, $\vec{f}_i$ are smooth, bounded functions depending on $(y, \mu_0^{-1}\mu, \xi, \mu_jy_j)$, and $g$, $\vec{g}$ depend on $(y, \mu_0^{-1}\mu, \xi)$.
\end{lemma}
\begin{proof}
We write
\begin{equation*}
u_{\mu,\xi}(x, t) = \sum_{i=1}^k\mu_i^{-\frac{n-2s}{2}}U(y_i) + \mu_i^{\frac{n-2s}{2}}\Phi_i^*(x, t) -  \mu_i^{\frac{n-2s}{2}} H(x, q_i), \quad y_i = \frac{x-\xi_i}{\mu_i}
\end{equation*}
and
\begin{equation*}
S(u_{\mu, \xi}) = S_1 + S_2,
\end{equation*}
where
\begin{equation*}
\begin{aligned}
S_1 :=&\sum_{i=1}^k\Big(\mu_i^{-\frac{n-2s}{2}-1}\dot{\xi}_i\cdot \nabla U(y_i)+\mu_i^{-\frac{n-2s}{2}-1}\dot{\mu}_iZ_{n+1}(y_i)\\
& + \frac{n-2s}{2}\mu_i^{\frac{n-2s}{2}-1}\dot{\mu}_iH(x, q_i)\Big) -\sum_{i=1}^k\partial_t\left(\mu_i^{\frac{n-2s}{2}}\Phi_i^*(x, t)\right),
\end{aligned}
\end{equation*}
\begin{equation*}
\begin{aligned}
S_2:=&\left(\sum_{i=1}^k\mu_i^{-\frac{n-2s}{2}}U(y_i) + \mu_i^{\frac{n-2s}{2}}\Phi_i^*(x, t) -  \mu_i^{\frac{n-2s}{2}} H(x, q_i)\right)^p\\ &-\sum_{i=1}^k\mu_i^{-\frac{n+2s}{2}}U(y_i)^p- \sum_{i=1}^k\mu_i^{\frac{n-2s}{2}}(-\Delta)^s\Phi_i^*(x, t).
\end{aligned}
\end{equation*}
Let
\begin{equation*}
S_2 = S_{21} + S_{22}
\end{equation*}
with
\begin{equation*}
\begin{aligned}
S_{21} = &\mu_j^{-\frac{n+2s}{2}}\left[\left(U(y_j) + \Theta_j\right)^p - U(y_j)^p\right],
\end{aligned}
\end{equation*}
\begin{equation*}
S_{22} = -\sum_{i\neq j}\mu_i^{-\frac{n+2s}{2}}U(y_i)^p- \sum_{i=1}^k\mu_i^{\frac{n-2s}{2}}(-\Delta)^s\Phi_i^*(x, t)
\end{equation*}
and
\begin{equation}\label{e:2018317}
\begin{aligned}
\Theta_j =&  -\mu_j^{n-2s}\left(H(x, q_j) - \Phi_j^*(x, t)\right)\\
& + \sum_{i\neq j}\left[(\mu_j\mu_i^{-1})^{\frac{n-2s}{2}}U(y_i)-(\mu_j\mu_i)^{\frac{n-2s}{2}}\left(H(x, q_i) - \Phi_i^*(x, t)\right)\right].
\end{aligned}
\end{equation}
Observe that $|\Theta_j|\lesssim \mu_0^{n-2s}$ uniformly in small $\delta$, we assume $U(y_j)^{-1}|\Theta_j|< \frac{1}{2}$ in the considered region. By Taylor expansion, we have
\begin{equation*}
\begin{aligned}
S_{21} = &\mu_j^{-\frac{n+2s}{2}}\left[pU(y_j)^{p-1}\Theta_j + p(p-1)\int_{0}^1(1-s)\left(U(y_j) + s\Theta_j\right)^{p-2}ds\Theta_j^2\right].
\end{aligned}
\end{equation*}
For $i\neq j$,
\begin{equation*}
\begin{aligned}
U(y_i) &= U\left(\frac{\mu_jy_j+\xi_j-\xi_i}{\mu_i}\right) = \frac{\alpha_{n,s}\mu_i^{n-2s}}{(\mu_i^2 + |\mu_jy_j+\xi_j - \xi_i|^2)^{\frac{n-2s}{2}}}\\
& = \frac{\alpha_{n,s}\mu_i^{n-2s}}{|\mu_jy_j+\xi_j - \xi_i|^{n-2s}} + \mu_i^{n-2s+2}f(\xi, \mu, \mu_jy_j),
\end{aligned}
\end{equation*}
where $f$ is smooth in its arguments and $f(q, 0, 0) = 0$. Then
\begin{equation*}
\begin{aligned}
\Theta_j =& -\mu_j^{n-2s}\left(H(\mu_jy_j+\xi_j, q_j) - \Phi_j^*(\mu_jy_j+\xi_j, t)\right)\\
&+ \sum_{i\neq j}\left[(\mu_j\mu_i)^{\frac{n-2s}{2}}G(\mu_jy_j+\xi_j, q_i) + \mu_i^{n-2s+2}f(\xi, \mu, \mu_jy_j)\right.\\
&+\left.(\mu_j\mu_i)^{\frac{n-2s}{2}}\Phi_i^*(\mu_j y_j+\xi_j,t)\right].
\end{aligned}
\end{equation*}
By further expansion, we get
\begin{equation*}
\begin{aligned}
\Theta_j =& -\mu_j^{n-2s}\left(H(q_j, q_j) - \Phi_j^*(q_j, t)\right)+ \sum_{i\neq j}(\mu_j\mu_i)^{\frac{n-2s}{2}}G(q_j, q_i) \\
& + \mu_i^{n-2s+2}f(\xi, \mu, \mu_jy_j)+(\mu_j\mu_i)^{\frac{n-2s}{2}}\Phi_i^*(\mu_j y_j+\xi_j,t)+(\mu_jy_j+\xi_j - q_j)\cdot\\
&\quad\quad\quad\quad\quad\quad \left[-\mu_j^{n-2s}\nabla \left(H(q_j, q_j) - \Phi_j^*(q_j, t)\right) + \sum_{i\neq j}(\mu_j\mu_i)^{\frac{n-2s}{2}}\nabla G(q_j, q_i)\right]\\
& +\int_{0}^1\Big\{-\mu_j^{n-2s}D_x^2 \left(H-\Phi_j^*\right)(q_j + s(\mu_jy_j+\xi_j - q_j), q_j)\\ &  + \sum_{i\neq j}(\mu_j\mu_i)^{\frac{n-2s}{2}}D_x^2 G(q_j + s(\mu_jy_j+\xi_j - q_j), q_i)\Big\}[\mu_jy_j+\xi_j - q_j]^2(1-s)ds.
\end{aligned}
\end{equation*}
We conclude that
\begin{equation*}
\begin{aligned}
\Theta_j &= -\mu_j^{n-2s}\left(H(q_j, q_j) - \Phi_j^*(q_j, t)\right) + \sum_{i\neq j}(\mu_j\mu_i)^{\frac{n-2s}{2}}G(q_j, q_i) \\
&\quad +\left[-\mu_j^{n-2s+1}\nabla H(q_j, q_j) + \sum_{i\neq j}\mu_j^{\frac{n-2s}{2}+1}\mu_i^{\frac{n-2s}{2}}\nabla G(q_j, q_i)\right]\cdot y_j\\
&\quad +\mu_0^{n-2s}(\xi_j-q_j)\cdot f(\xi, \mu_jy_j, \mu_0^{-1}\mu) + \mu_0^{n-2s+2}F(\xi, \mu_jy_j, \mu_0^{-1}\mu)[y_j]^2\\
&\quad +\mu_i^{n-2s+2}f(\xi, \mu, \mu_jy_j),
\end{aligned}
\end{equation*}
where $f$ and $F$ are smooth and bounded in its arguments.
On the other hand, we have
\begin{equation*}
\begin{aligned}
S_{22} &= -\sum_{i\neq j}\mu_i^{-\frac{n+2s}{2}}U(y_i)^p - \sum_{i=1}^k\mu_i^{\frac{n-2s}{2}}(-\Delta)^s\Phi_i^*(x, t)\\
& = -\sum_{i\neq j}\frac{\alpha_{n,s}^p\mu_i^{\frac{n+2s}{2}}}{|q_j - q_i|^{n+2s}} + \mu_i^{\frac{n+2s}{2}}f(\xi, \mu, \mu_iy_i) - \sum_{i=1}^k\mu_i^{\frac{n-2s}{2}}(-\Delta)^s\Phi_i^*(x, t),
\end{aligned}
\end{equation*}
so
\begin{equation*}
S_{22} = \mu_0^{\frac{n+2s}{2}}f(\xi, \mu_0^{-1}\mu, \mu_jy_j) - \sum_{i=1}^k\mu_i^{\frac{n-2s}{2}}(-\Delta)^s\Phi_i^*(x, t),
\end{equation*}
where $f$ is smooth in its arguments and $f(q, 0, 0) = 0$.

Decompose $S_1 = S_{11} + S_{12}$, where
\begin{equation*}
S_{11} :=\mu_j^{-\frac{n-2s}{2}-1}\dot{\xi}_j\cdot \nabla U(y_j)+\mu_j^{-\frac{n-2s}{2}-1}\dot{\mu}_jZ_{n+1}(y_j) - \mu_j^{\frac{n-2s}{2}}\partial_t\Phi_j^*(x, t),
\end{equation*}
\begin{equation*}
\begin{aligned}
S_{12} :=&\sum_{i\neq j}\mu_i^{-\frac{n-2s}{2}-1}\dot{\xi}_i\cdot \nabla U(y_i) - \left(\partial_t\mu_j^{\frac{n-2s}{2}}\right)\Phi_j^*(x, t)\\
&+\mu_i^{-\frac{n-2s}{2}-1}\dot{\mu}_iZ_{n+1}(y_i) + \sum_{i=1}^k\frac{n-2s}{2}\mu_i^{\frac{n-2s}{2}-1}\dot{\mu}_iH(x, q_i)\\
& - \sum_{i\neq j}^k\partial_t\left(\mu_i^{\frac{n-2s}{2}}\Phi_i^*(x, t)\right).
\end{aligned}
\end{equation*}
We write
\begin{eqnarray*}
\begin{aligned}
S_{12} =&\sum_{i\neq j}-\alpha_{n,s}(n-2s)\mu_i^{\frac{n-2s}{2}}\dot{\xi}_i\cdot \left[\frac{q_i-q_j}{|q_i-q_j|^{n-2s+2}}+\vec{f}_i(\xi,\mu, \mu_jy)\right]\\
& +\sum_{i\neq j}\mu_i^{\frac{n-2s}{2}-1}\dot{\mu}_i\left[\frac{c_{n, s}}{|q_i-q_j|^{n-2s}} + f_i(\xi,\mu, \mu_jy)\right]\\
& +\sum_{i=1}^k\frac{n-2s}{2}\mu_i^{\frac{n-2s}{2}-1}\dot{\mu}_i\left[\left(H(q_j, q_i)-\Phi_i^*(q_j, t)\right)+f_i(\mu_jy, \xi)\right],
\end{aligned}
\end{eqnarray*}
where $\vec{f}_i$ are smooth in their arguments vanishing in the limit. In total, we can write
\begin{equation*}
S_{12} = \mu_0^{\frac{n-2s}{2}-1}\sum_{i=1}^k\dot{\mu}_if_{i0}(\mu_0^{-1}\mu, \xi, \mu_jy) + \mu_0^{\frac{n-2s}{2}}\sum_{i=1}^k\dot{\xi}_i\cdot \vec{f}_{i1}(\mu_0^{-1}\mu, \xi, \mu_jy)
\end{equation*}
for functions $f_{i0}$, $\vec{f}_{i1}$ smooth in their arguments. This concludes the proof of the lemma.
\end{proof}

Next, we try to find a solution with the following form
\begin{equation*}
u(x, t) = u_{\mu, \xi}(x,t) + \tilde{\phi}(x, t),
\end{equation*}
where $\tilde{\phi}$ is a small term. By $S(u_{\mu, \xi}+\tilde{\phi})=0$, our main equation can be expressed with respect to $\tilde{\phi}$ as
\begin{equation}\label{e2:27}
 -\partial_t\tilde{\phi}-(-\Delta)^s\tilde{\phi} + pu_{\mu, \xi}^{p-1}\tilde{\phi} + S(u_{\mu, \xi}) + \tilde{N}_{\mu, \xi}(\tilde{\phi}),
\end{equation}
where
\begin{equation}\label{e2:28}
\tilde{N}_{\mu, \xi}(\tilde{\phi}) = (u_{\mu, \xi} + \tilde{\phi})^p-u^p_{\mu, \xi} - pu_{\mu, \xi}^{p-1}\tilde{\phi}.
\end{equation}
Write $\tilde{\phi}(x, t)$ in self-similar form around $q_j$
\begin{equation}\label{e2:29}
\tilde{\phi}(x, t) = \mu_j^{-\frac{n-2s}{2}}\phi\left(\frac{x-\xi_j}{\mu_j}, t\right).
\end{equation}
By a direct computation, we obtain from \eqref{e2:27}-\eqref{e2:29} that
\begin{equation}\label{e2:30}
0=\mu_j^{\frac{n+2s}{2}}S(u_{\mu, \xi}+\tilde{\phi}) = -(-\Delta)^s_y\phi + pU(y)^{p-1}\phi + \mu_j^{\frac{n+2s}{2}}S(u_{\mu, \xi}) + A[\phi]
\end{equation}
with
\begin{equation*}\label{Aphi}
\begin{aligned}
A[\phi] =& -\mu_j^{2s}\partial_t\phi + \mu_j^{2s-1}\dot{\mu}_j\left[\frac{n-2s}{2}\phi+y\cdot\nabla_y\phi\right] + \nabla_y\phi\cdot \mu_j^{2s-1}\dot{\xi}_j\\
&+ p\left[\left(U(y) + \Theta_j\right)^{p-1} - U(y)^{p-1}\right]\phi + \left(U(y) + \Theta_j + \phi\right)^{p}\\
&- \left(U(y) + \Theta_j\right)^{p} - p\left(U(y) + \Theta_j\right)^{p-1}\phi,\\
\end{aligned}
\end{equation*}
where $\Theta_j$ is defined in (\ref{e:2018317}). We assume that $\phi$ decays in the $y$ variable and $A[\phi]$ is small when $t$ is large.

\subsection{Improvement of the approximation.}
To improve the approximation, $\phi(y, t)$ should be equal to the solution $\phi_{0j}(y, t)$ of the following elliptic type equation of order $2s$
\begin{equation}\label{e2:31}
-(-\Delta)^s_y\phi_{0j} + pU(y)^{p-1}\phi_{0j} = -\mu_j^{\frac{n+2s}{2}}S(u_{\mu, \xi})\text{ in }\mathbb{R}^n,\,\,\phi_{0j}(y, t)\to 0\text{ as }|y|\to \infty
\end{equation}
at main order.

Equation (\ref{e2:31}) is a special case of
\begin{equation}\label{e2:32}
L_0[\psi]:= -(-\Delta)^s_y\psi + pU(y)^{p-1}\psi = h(y)\text{ in }\mathbb{R}^n,\,\,\psi(y)\to 0\text{ as }|y|\to \infty.
\end{equation}
It is well known (see \cite{DelPinoSirePAMS}) that every bounded solution of $L_0[\psi] = 0$ in $\mathbb{R}^n$ is the linear combination of the functions
$$Z_1,\cdots, Z_{n+1},$$
where
\begin{equation*}
Z_i(y):= \frac{\partial U}{\partial y_i}(y),\quad i = 1,\cdots, n, \quad Z_{n+1}(y):=\frac{n-2s}{2}U(y) + y\cdot\nabla U(y).
\end{equation*}
Furthermore, problem (\ref{e2:32}) is solvable for $h(y) = O(|y|^{-m})$, $m > 2s$, if it holds that
\begin{equation*}
\int_{\mathbb{R}^n}h(y)Z_i(y)dy = 0\quad\text{for all}\quad i = 1,\cdots, n+1.
\end{equation*}

First, we consider the solvability condition for (\ref{e2:31}),
\begin{equation}\label{e2:35}
\int_{\mathbb{R}^n}\mu_j^{\frac{n+2s}{2}}S(u_{\mu, \xi})(y, t)Z_{n+1}(y)dy = 0.
\end{equation}
We claim that by choosing $\mu_{0j} = b_j\mu_0(t)$ for suitable positive constant $b_j$, $j = 1, \cdots, k$, $\mu_0(t) = c_{n, s} t^{-\frac{1}{n-4s}}$ with $c_{n, s}$ be a positive constant depending only on $n$ and $s$, this identity can be achieved at main order. Note that, we have $\dot{\mu}_0(t) = -\frac{1}{(n-4s)c_{n,s}^{n-4s}}\mu_0^{n-4s+1}(t)$. The main contribution to the integral comes from the term
\begin{equation*}
\begin{aligned}
E_{0j}&= pU(y_j)^{p-1}\left[-\mu_j^{n-2s-1}H(q_j, q_j) + \sum_{i\neq j}\mu_j^{\frac{n-2s}{2}-1}\mu_i^{\frac{n-2s}{2}}G(q_j, q_i)\right]\\
&\quad + pU(y_j)^{p-1}\mu_j^{n-2s-1}\Phi^0_j(q_j, t)\\
&\quad +\mu_j^{2s-2}\dot{\mu}_j\left(Z_{n+1}(y_j) + \frac{n-2s}{2}\alpha_{n, s}\frac{1}{\left(1+|y_j|^2\right)^{\frac{n-2s}{2}}}\right).
\end{aligned}
\end{equation*}
Now, let us compute the nonlocal term $\Phi^0_j(q_j, t)$. Since the heat kernel function $p(t, x)$ satisfies
$$p(t-\tilde{s},q_j-y)\asymp \frac{t-\tilde{s}}{[(t-\tilde{s})^{\frac{1}{s}}+|q_j-y|^2]^{\frac{n+2s}{2}}},$$
 we have
$$p(t-\tilde{s},q_j-y)=(t-\tilde{s})^{-\frac{n}{2s}}p\left(1,\frac{|q_j-y|}{(t-\tilde{s})^{\frac{1}{2s}}}\right)$$
and
\begin{equation*}
\begin{aligned}
&\Phi^0_j(q_j, t)= -\int_{t_0}^{t}\int_{\mathbb{R}^n}p(t-\tilde{s}, q_j-y)\frac{\dot{\mu}_j(\tilde{s})}{\mu_j(\tilde{s})}\frac{\mu^{-(n-2s)}_j(\tilde{s})}
{\left(1+\left|\frac{y-\xi_j(\tilde{s})}{\mu_j(\tilde{s})}\right|^2\right)^{\frac{n-2s}{2}}}dyd\tilde{s}\\
&= -(1+o(1))\int_{t_0}^{t}\int_{\mathbb{R}^n}p(t-\tilde{s}, q_j-y)\frac{\dot{\mu}_j(\tilde{s})}{\mu_j(\tilde{s})}\frac{\mu^{-(n-2s)}_j(\tilde{s})}
{\left(1+\left|\frac{y-q_j}{\mu_j}\right|^2\right)^{\frac{n-2s}{2}}}dyd\tilde{s}\\
&= -(1+o(1))\int_{t_0}^{t}\frac{1}{(t-\tilde{s})^{\frac{n}{2s}}}d\tilde{s}\int_{\mathbb{R}^n}p\left(1, \frac{q_j-y}{(t-\tilde{s})^{\frac{1}{2s}}}\right)\frac{\dot{\mu}_j(\tilde{s})}{\mu_j(\tilde{s})}\\
&\quad\quad\quad\quad\quad\quad\quad\quad\quad\quad\quad\quad\quad\quad\quad\quad\quad \times\frac{\mu^{-(n-2s)}_j(\tilde{s})\left(t-\tilde{s}\right)^{\frac{n}{2s}}}
{\left(1+\left|\frac{(t-\tilde{s})^{\frac{1}{2s}}}{\mu_j(\tilde{s})}\frac{q_j-y}{(t-\tilde{s})^{\frac{1}{2s}}}\right|^2\right)^{\frac{n-2s}{2}}
}d\frac{y-q_j}{(t-\tilde{s})^{\frac{1}{2s}}}\\
&= -(1+o(1))\int_{t_0}^{t}\frac{\dot{\mu}_j(\tilde{s})}{\mu_j(\tilde{s})}\mu_j^{-(n-2s)}(\tilde{s})d\tilde{s}\int_{\mathbb{R}^n}p\left(1, \frac{q_j-y}{(t-\tilde{s})^{\frac{1}{2s}}}\right)\\
&\quad\quad\quad\quad\quad\quad\quad\quad\quad\quad\quad\quad\quad\quad\quad\quad\quad\times\frac{1}{\left(1+\left|\frac{(t-\tilde{s})^{\frac{1}{2s}}}{\mu_j(\tilde{s})}\frac{q_j-y}{(t-\tilde{s})^{\frac{1}{2s}}}\right|^2\right)^{\frac{n-2s}{2}}}
d\frac{y-q_j}{(t-\tilde{s})^{\frac{1}{2s}}}\\
&= -(1+o(1))\int_{t_0}^{t}\frac{\dot{\mu}_j(\tilde{s})}{\mu_j(\tilde{s})}\mu_j^{-(n-2s)}(\tilde{s})F\left(\frac{(t-\tilde{s})^{\frac{1}{2s}}}{\mu_j(\tilde{s})}\right)d\tilde{s},
\end{aligned}
\end{equation*}
where
\begin{equation*}
F(a) = \int_{\mathbb{R}^n}p\left(1, x\right)\frac{1}
{\left(1+a^2|x|^2\right)^{\frac{n-2s}{2}}}dx.
\end{equation*}
We claim that
\begin{equation}\label{e:201803312}
\Phi^0_j(q_j, t) = -(1+o(1))\int_{t_0}^{t}\frac{\dot{\mu}_j(\tilde{s})}{\mu_j(\tilde{s})}\mu_j^{-(n-2s)}(\tilde{s})F\left(\frac{(t-\tilde{s})^{\frac{1}{2s}}}{\mu_j(\tilde{s})}\right)d\tilde{s} = c(1+o(1))
\end{equation}
for a suitable positive constant $c$ depending on $n$, $s$ and $b_j$, $j = 1,\cdots, k$.

Indeed, for a small constant $\delta > 0$, we decompose the integral $$\int_{t_0}^{t}\frac{\dot{\mu}_j(\tilde{s})}{\mu_j(\tilde{s})}\mu_j^{-(n-2s)}(\tilde{s})F\left(\frac{(t-\tilde{s})^{\frac{1}{2s}}}{\mu_j(\tilde{s})}\right)d\tilde{s}$$ as
\begin{equation*}
\begin{aligned}
&\int_{t_0}^{t}\frac{\dot{\mu}_j(\tilde{s})}{\mu_j(\tilde{s})}\mu_j^{-(n-2s)}(\tilde{s})F\left(\frac{(t-\tilde{s})^{\frac{1}{2s}}}{\mu_j(\tilde{s})}\right)d\tilde{s}\\
&=  \int_{t_0}^{t-\delta}\frac{\dot{\mu}_j(\tilde{s})}{\mu_j(\tilde{s})}\mu_j^{-(n-2s)}(\tilde{s})F\left(\frac{(t-\tilde{s})^{\frac{1}{2s}}}{\mu_j(\tilde{s})}\right)d\tilde{s}\\
&\quad + \int_{t-\delta}^{t}\frac{\dot{\mu}_j(\tilde{s})}{\mu_j(\tilde{s})}\mu_j^{-(n-2s)}(\tilde{s})F\left(\frac{(t-\tilde{s})^{\frac{1}{2s}}}{\mu_j(\tilde{s})}\right)d\tilde{s}\\
& := I_1 + I_2.
\end{aligned}
\end{equation*}
For the first integral we have $t-\tilde{s} >\delta$, therefore
\begin{equation*}
\begin{aligned}
0\leq -I_1 &\leq \int_{t_0}^{t-\delta}\mu_j^{-2s}(\tilde{s})F\left(\frac{(t-\tilde{s})^{\frac{1}{2s}}}{\mu_j(\tilde{s})}\right)d\tilde{s}\leq C\int_{t_0}^{t-\delta}\mu_j^{-2s}(\tilde{s})\left|\frac{(t-\tilde{s})^{\frac{1}{2s}}}{\mu_j(\tilde{s})}\right|^{-(n-2s)}d\tilde{s}\\
&= C b_j^{n-4s}c_{n,s}^{n-4s}\int_{t_0}^{t-\delta}\frac{1}{\tilde{s}}\frac{1}{(t-\tilde{s})^{\frac{n-2s}{2s}}}d\tilde{s}\leq \frac{C b_j^{n-4s}c_{n,s}^{n-4s}}{t_0}\frac{2s}{n-4s}\frac{1}{\delta^{\frac{n-4s}{2s}}}.
\end{aligned}
\end{equation*}
Here we have used the ansatz $\mu_{0j} = b_j c_{n,s} t^{-\frac{1}{n-4s}}$ and the fact that $|a|^{n-2s}F(a)\leq C$.
For the second integral $$I_2 =\int_{t-\delta}^{t}\frac{\dot{\mu}_j(\tilde{s})}{\mu_j(\tilde{s})}\mu_j^{-(n-2s)}(\tilde{s})F\left(\frac{(t-\tilde{s})^{\frac{1}{2s}}}{\mu_j(\tilde{s})}\right)d\tilde{s},$$ using change of variables
$\frac{(t-\tilde{s})^{\frac{1}{2s}}}{\mu_j(\tilde{s})} = \hat{s}$, we obtain
$$
d\tilde{s} = -\frac{\mu_j(\tilde{s})}{\frac{1}{2s}(t-\tilde{s})^{\frac{1}{2s}-1}+\dot{\mu}_j(\tilde{s})\hat{s}}d\hat{s}
$$
and the integral becomes
\begin{equation*}
\begin{aligned}
I_2&=\int_{t-\delta}^{t}\frac{\dot{\mu}_j(\tilde{s})}{\mu_j(\tilde{s})}\mu_j^{-(n-2s)}(\tilde{s})F\left(\frac{(t-\tilde{s})^{\frac{1}{2s}}}{\mu_j(\tilde{s})}\right)
d\tilde{s}\\ &=\int^{\frac{\delta^{\frac{1}{2s}}}{\mu(t-\delta)}}_{0}\frac{\dot{\mu}_j(\tilde{s})}{\mu_j(\tilde{s})}\mu_j^{-(n-2s)}(\tilde{s})F\left(\hat{s}\right)
\frac{\mu_j(\tilde{s})}{\frac{1}{2s}(t-\tilde{s})^{\frac{1}{2s}-1}+\dot{\mu}_j(\tilde{s})\hat{s}}d\hat{s}.
\end{aligned}
\end{equation*}
Note that $\frac{1}{2s}(t-\tilde{s})^{\frac{1}{2s}-1}+\dot{\mu}_j(\tilde{s})\hat{s} =\frac{1}{2s}(t-\tilde{s})^{\frac{1}{2s}-1}(1-\frac{2s}{(n-4s)\tilde{s}}(t-\tilde{s}))
> \frac{1}{2s}(t-\tilde{s})^{\frac{1}{2s}-1}(1-\frac{2s}{(n-4s)\tilde{s}}\delta)$, $d\tilde{s} = \frac{\mu_j(\tilde{s})}{\frac{1}{2s}(t-\tilde{s})^{\frac{1}{2s}-1}}(1+O(\delta))d\hat{s}$ for $\delta$ small and
\begin{equation*}
I_2= -\frac{2s b_j^{4s-n}}{(n-4s)c_{n,s}^{n-4s}}\left(\int^{\frac{\delta^{\frac{1}{2s}}}{\mu_j(t-\delta)}}_{0}\hat{s}^{2s-1}F\left(\hat{s}\right)
d\hat{s} + o(1)\right)
= -\frac{2s b_j^{4s-n}}{(n-4s)c_{n,s}^{n-4s}}A + o(1)
\end{equation*}
as long as $\frac{\delta^{\frac{1}{2s}}}{\mu_j(t-\delta)}$ is large. Here $A = \int_{0}^\infty \tilde{s}^{2s-1} F(\tilde{s})d\tilde{s} < +\infty$ since $n > 4s$.
Thus we have
\begin{equation}\label{e:2018330}
\begin{aligned}
\Phi^0_j(q_j, t) &= -(1+o(1))\int_{t_0}^{t}\frac{\dot{\mu}_j(\tilde{s})}{\mu_j(\tilde{s})}\mu_j^{-(n-2s)}(\tilde{s})F\left(\frac{(t-\tilde{s})^{\frac{1}{2s}}}{\mu_j(\tilde{s})}\right)d\tilde{s} \\
&= \frac{2s b_j^{4s-n}}{(n-4s)c_{n,s}^{n-4s}}A + o(1):= B b_j^{4s-n}+ o(1)
\end{aligned}
\end{equation}
when $t_0$ is chosen sufficiently large. Here $ B = B_{n, s}: = \frac{2s}{(n-4s)c_{n,s}^{n-4s}}A$. This proves (\ref{e:201803312}).

By direct computations, we have
\begin{equation}\label{e2:36}
\begin{aligned}
&\mu_0^{-(n-2s-1)}(t)\int_{\mathbb{R}^n}E_{0j}(y, t)Z_{n+1}(y)dy\\
& \approx  c_1\left[b_j^{n-2s-1}H(q_j, q_j) - \sum_{i\neq j}b_j^{\frac{n-2s}{2}-1}b_i^{\frac{n-2s}{2}}G(q_j, q_i)\right]\\
&\quad -\frac{2sc_1 A +c_2}{(n-4s)c_{n,s}^{n-4s}}b_j^{2s-1}
\end{aligned}
\end{equation}
with
\begin{equation*}
c_1 = -p\int_{\mathbb{R}^n}U(y)^{p-1}Z_{n+1}(y)dy,
\end{equation*}
\begin{equation*}
c_2 = \int_{\mathbb{R}^n}\left(Z_{n+1}(y) + \frac{n-2s}{2}\alpha_{n, s}\frac{1}{\left(1+|y|^2\right)^{\frac{n-2s}{2}}}\right)Z_{n+1}(y)dy.
\end{equation*}
Denote
\begin{equation*}
\mu_j(t) = b_j\mu_0(t) = b_j c_{n, s}t^{-\frac{1}{n-4s}}.
\end{equation*}
Then the solvability conditions (\ref{e2:35}) can be achieved at main order if
\begin{equation}\label{e2:39}
b_j^{n-2s}H(q_j, q_j) - \sum_{i\neq j}(b_ib_j)^{\frac{n-2s}{2}}G(q_j, q_i)-\frac{2sc_1 A +c_2}{(n-4s)c_{n,s}^{n-4s}c_1}b_j^{2s} = 0\text{ for all } j = 1,\cdots, k.
\end{equation}
By imposing $-\frac{2sc_1 A +c_2}{(n-4s)c_{n,s}^{n-4s}c_1}=-\frac{2s}{n-2s}$, namely
$$
c_{n,s}=\left[\frac{(2sc_1 A +c_2)(n-2s)}{2s(n-4s)c_1}\right]^{\frac{1}{n-4s}},
$$
we have
\begin{equation}\label{e2:40}
\dot{\mu}_0(t) = -\frac{2sc_1}{(2sc_1A + c_2)(n-2s)}\mu_0^{n-4s+1}(t).
\end{equation}
From \eqref{e2:39} and \eqref{e2:40}, the constants $b_j$ satisfy the system
\begin{equation}\label{e2:42}
b_j^{n-2s-1}H(q_j, q_j) - \sum_{i\neq j}b_j^{\frac{n-2s}{2}-1}b_i^{\frac{n-2s}{2}}G(q_j, q_i) = \frac{2sb_j^{2s-1}}{n-2s}\text{ for all } j = 1,\cdots, k.
\end{equation}

System (\ref{e2:42}) is equivalent to the variational problem $\nabla_bI(b) = 0$ with
\begin{equation}\label{e2:43}
I(b): = \frac{1}{n-2s}\left[\sum_{j=1}^kb_j^{n-2s}H(q_j, q_j) - \sum_{i\neq j}b_j^{\frac{n-2s}{2}}b_i^{\frac{n-2s}{2}}G(q_j, q_i)-\sum_{j=1}^kb_j^{2s}\right].
\end{equation}
Let $\Lambda_j = b_j^{\frac{n-2s}{2}}$. Then
\begin{equation*}
(n-2s)I(b)=\tilde{I}(\Lambda)=\left[\sum_{j=1}^kH(q_j, q_j)\Lambda_j^2 - \sum_{i\neq j}G(q_j, q_i)\Lambda_i\Lambda_j-\sum_{j=1}^k\Lambda_j^{\frac{4s}{n-2s}}\right].
\end{equation*}
Standard argument shows that system (\ref{e2:42}) has a unique solution if and only if the following matrix
\begin{eqnarray*}
\mathcal{G}(q) = \left[
\begin{matrix}
H(q_1, q_1)&-G(q_1, q_2)&\cdots & -G(q_1, q_k)\\
-G(q_2, q_1)&H(q_2, q_2)&\cdots &-G(q_2, q_k)\\
\vdots &\vdots & \ddots &\vdots\\
-G(q_k,q_1)&-G(q_k, q_2)&\cdots&H(q_k, q_k)
\end{matrix}
\right]
\end{eqnarray*}
is positive definite.

Next, we consider the solvability conditions for (\ref{e2:31}),
\begin{equation*}
\int_{\mathbb{R}^n}\mu_j^{\frac{n+2s}{2}}S(u_{\mu, \xi})(y, t)Z_{i}(y)dy = 0,\quad i = 1, \cdots, n.
\end{equation*}
It is clear that these conditions are fulfilled at main order by simply setting $\xi_{0j} = q_j$.

Now we fix the function $\mu_0(t)$ and the positive constants $b_j$ satisfying (\ref{e2:42}) and denote
\begin{equation*}
\bar{\mu}_0 = (\mu_{01},\cdots, \mu_{0k}) = (b_1\mu_0,\cdots, b_k\mu_0).
\end{equation*}
Let $\Phi_j$ be the unique solution to (\ref{e2:31}) for $\mu = \bar{\mu}_0$. Then we have
\begin{equation*}
-(-\Delta)^s_y\Phi_j + pU(y)^{p-1}\Phi_j = -\mu_{0j}E_{0j}[\bar{\mu}_0,\dot{\mu}_{0j}]\text{ in }\mathbb{R}^n,\,\,\Phi_j(y, t)\to 0\text{ as }|y|\to \infty.
\end{equation*}
From the definition of $\mu_0$ and $b_j$, one has
\begin{equation*}
\mu_{0j}E_{0j}=-\gamma_j\mu_0^{n-2s}q_0(y),
\end{equation*}
where the constant $\gamma_j$ is positive and
\begin{equation}\label{e2:50}
q_0(y): = \frac{pU(y)^{p-1}c_2b_j^{2s}}{(n-4s)c_{n,s}^{n-4s}c_1} + \frac{b_j^{2s}}{(n-4s)c_{n,s}^{n-4s}}\left(Z_{n+1}(y) + \frac{n-2s}{2}\alpha_{n, s}\frac{1}{\left(1+|y|^2\right)^{\frac{n-2s}{2}}}\right)
\end{equation}
with $\int_{\mathbb{R}^n}q_0(y)Z_{n+1}dy = 0$.

Let $p_0 = p_0(|y|)$ be the radial solution of $L_0(p_0) = q_0$. Then $p_0(|y|) = O(|y|^{-2s})$ as $|y|\to\infty$ since we have (\ref{e2:50}). Therefore,
\begin{equation}\label{e2:51}
\Phi_j(y, t) = \gamma_j\mu_0^{n-2s}p_0(y).
\end{equation}
Thus we can define the corrected approximation as
\begin{equation}\label{e2:52}
u^*_{\mu, \xi}(x, t) = u_{\mu, \xi}(x, t) + \tilde{\Phi}(x, t)
\end{equation}
with
\begin{equation*}
\tilde{\Phi}(x, t) = \sum_{j=1}^k\mu_j^{-\frac{n-2s}{2}}\Phi_j\left(\frac{x-\xi_j}{\mu_j}, t\right).
\end{equation*}

\subsection{Estimating the error $S(u^*_{\mu, \xi})$.}
In the region $|x-q_i| > \delta$ for all $i = 1,\cdots, k$, $S(u^*_{\mu, \xi})$ can be described as
\begin{equation}\label{awayq}
S(u^*_{\mu, \xi})(x, t) = \mu_0^{\frac{n-2s}{2}-1}\sum_{j=1}^k\dot{\mu}_jf_{j} + \mu_0^{\frac{n-2s}{2}}\sum_{j=1}^k\dot{\xi}_j\cdot \vec{f}_{j} + \mu_0^{\frac{n+2s}{2}}f,
\end{equation}
where the functions $f_j$, $\vec{f}_j$, $f$ are smooth bounded and depend on $(x, \mu_0^{-1}\mu, \xi)$.
Now we consider the region near each of the points $q_j$. By direct computations, $S(u^*_{\mu, \xi})$ is given by
\begin{equation}\label{sstar}
\begin{aligned}
S(u^*_{\mu, \xi})  = & S(u_{\mu, \xi}) - \sum_{j=1}^k\mu_j^{-\frac{n+2s}{2}}\mu_{0j}E_{0j}[\bar{\mu}_0, \dot{\mu}_{0j}]+\sum_{j=1}^k\mu_j^{-\frac{n+2s}{2}}\Bigg\{-\mu_j^{2s}\partial_t\Phi_j(y_j, t) \\
&+ \mu_j^{2s-1}\dot{\mu}_j\left[\frac{n-2s}{2}\Phi_{j}(y_j, t)+y_j\cdot\nabla_y\Phi_j\right] + \mu_j^{2s-1}\nabla_y\Phi_j(y_j, t)\cdot \dot{\xi}_j\Bigg\}\\
&+ (u_{\mu, \xi} + \tilde{\Phi})^{p} - u_{\mu, \xi}^{p} - p\sum_{j=1}^k\mu_j^{-\frac{n+2s}{2}}U(y_j)^{p-1}\Phi_j(y_j, t),
\end{aligned}
\end{equation}
where $y_j=\frac{x-\xi_j}{\mu_j}$.

From (\ref{e2:30}), for a given fixed $j$ and $|x-q_j|\leq \delta$, we have
\begin{equation}\label{e2:55}
\mu_j^{\frac{n+2s}{2}}S(u^*_{\mu, \xi}) = \mu_j^{\frac{n+2s}{2}}S(u_{\mu, \xi}) - \mu_{0j}E_{0j}[\bar{\mu}_0, \dot{\mu}_{0j}]+A_j(y).
\end{equation}
After direct computations,
\begin{equation}\label{e2:56}
A_j = \mu_0^{n+4s}f(\mu_0^{-1}\mu, \xi, \mu_jy) + \frac{\mu_0^{2n-4s}}{1+|y_j|^{2s}}g(\mu_0^{-1}\mu, \xi,\mu_jy),\quad y_j = \frac{x-\xi_j}{\mu_j},
\end{equation}
where $f$ and $g$ are smooth and bounded.

Finally we set
\begin{equation*}
\mu(t) = \bar{\mu}_0 + \lambda(t)\text{ with }\lambda(t) = (\lambda_1(t),\cdots, \lambda_k(t)).
\end{equation*}
From (\ref{e2:55}), we have
\begin{equation*}
\begin{aligned}
S(u^*_{\mu, \xi}) = \mu_j^{-\frac{n+2s}{2}}\Big\{\mu_{0j}\big(E_{0j}[\mu, \dot{\mu}_j]- E_{0j}[\bar{\mu}_0, \dot{\mu}_{0j}]\big)+\lambda_j&E_{0j}[\mu, \dot{\mu}_j]\\
&+\mu_jE_{1j}[\mu,\dot{\xi}_j]+R_j+A_j\Big\}.
\end{aligned}
\end{equation*}

Recall Lemma \ref{l2.1} that
\begin{equation*}
\begin{aligned}
&E_{0j}[\mu, \dot{\mu}_j]=pU(y)^{p-1}\left[-\mu_j^{n-2s-1}H(q_j, q_j) + \sum_{i\neq j}\mu_j^{\frac{n-2s}{2}-1}\mu_i^{\frac{n-2s}{2}}G(q_j, q_i)\right]\\
&\quad + pU(y)^{p-1}\mu_j^{n-2s-1}\Phi^0_j(q_j, t)+\mu_j^{2s-2}\dot{\mu}_j\left(Z_{n+1}(y) + \frac{n-2s}{2}\alpha_{n, s}\frac{1}{\left(1+|y|^2\right)^{\frac{n-2s}{2}}}\right).
\end{aligned}
\end{equation*}
Note that $\Phi_j^0$ depends on $\mu,$ $\dot{\mu}$ and
\begin{equation*}
\begin{aligned}
\mu_0^{n-2s-1}\Phi^0_j[\bar{\mu}_0+\lambda, b_j\dot{\mu}_0&+\dot{\lambda}_j](q_j, t) - \mu_0^{n-2s-1}\Phi^0_j[\bar{\mu}_0, b_j\dot{\mu}_0](q_j, t)\\
& =  -(b_j\mu_0)^{2s-2}2sA \dot{\lambda}_j - \mu_0^{n-2s-2}(n-4s+1)b_j^{2s-2}B\lambda_j
\end{aligned}
\end{equation*}
which can be deduced by similar arguments as (\ref{e:2018330}), we have
\begin{equation*}
\begin{aligned}
&E_{0j}[\bar{\mu}_0+\lambda, b_j\dot{\mu}_0+\dot{\lambda}_j]- E_{0j}[\bar{\mu}_0, b_j\dot{\mu}_0]\\
&= (b_j\mu_0)^{2s-2}\dot{\lambda}_j\left(Z_{n+1}(y) + \frac{n-2s}{2}\alpha_{n, s}\frac{1}{\left(1+|y|^2\right)^{\frac{n-2s}{2}}}\right)\\
&\quad + (2s-2)(b_j\mu_0)^{2s-3}\lambda_j(b_j\dot{\mu}_0+\dot{\lambda}_j)\left(Z_{n+1}(y) + \frac{n-2s}{2}\alpha_{n, s}\frac{1}{\left(1+|y|^2\right)^{\frac{n-2s}{2}}}\right)\\
&\quad -\mu_0^{n-2s-2}pU(y)^{p-1}\left[\sum_{i=1}^kM_{ij}\lambda_i +\sum_{i,l=1}^kf_{ijl}(\mu_0^{-1}\lambda)\lambda_i\lambda_l\right]\\
&\quad + \mu_0^{n-2s-2} pU(y)^{p-1}(n-2s-1)b_j^{2s-2}B\lambda_j\\
&\quad - pU(y)^{p-1}(b_j\mu_0)^{2s-2}2sA \dot{\lambda}_j- \mu_0^{n-2s-2}pU(y)^{p-1}(n-4s+1)b_j^{2s-2}B\lambda_j,
\end{aligned}
\end{equation*}
where $f_{ijl}$ are smooth functions and for $i = j$,
$$
M_{ij} = (n-2s-1)b_j^{n-2s-2}H(q_j, q_j) - (\frac{n-2s}{2}-1)\sum_{i\neq j}b_j^{\frac{n-2s}{2}-2}b_i^{\frac{n-2s}{2}}G(q_j, q_i),
$$
for $i\neq j$,
$$
M_{ij} = - \frac{n-2s}{2}\sum_{i\neq j}b_j^{\frac{n-2s}{2}-1}b_i^{\frac{n-2s}{2}-1}G(q_j, q_i).
$$
$M = D^2I_0(b)$ with
\begin{equation*}
I_0(b): = \frac{1}{n-2s}\left[\sum_{j=1}^kb_j^{n-2s}H(q_j, q_j) - \sum_{i\neq j}b_j^{\frac{n-2s}{2}}b_i^{\frac{n-2s}{2}}G(q_j, q_i)\right].
\end{equation*}
Since $D^2I(b)$ is positive definite, we denote its positive eigenvalues corresponding to the eigenvectors $\bar{w}_j$ by $\bar{\sigma}_j$  for $j=1,\cdots,k$. Thus
\begin{equation}\label{e2:46}
D^2I(b) = P^T\text{diag}(\bar{\sigma}_1,\cdots, \bar{\sigma}_k)P
\end{equation}
and $P$ is the $k\times k$ matrix given by $P := (\bar{w}_1,\cdots ,\bar{w}_k)$.
From the definition of $b_j$ in (\ref{e2:42}), one has
\begin{equation*}
\begin{aligned}
M=& D^2I_0(b) = D^2I(b) + \frac{2s(2s-1)}{n-2s}\text{diag}(b_1^{2s-2},\cdots, b_k^{2s-2})\\
=& P^T\text{diag}(\bar{\sigma}_1 + \frac{2s(2s-1)}{n-2s}b_1^{2s-2},\cdots, \bar{\sigma}_k+\frac{2s(2s-1)}{n-2s}b_k^{2s-2})P.
\end{aligned}
\end{equation*}

Now we estimate $\lambda_jE_{0j}[\mu, \dot{\mu}_j]$. Indeed, we have
\begin{equation*}
\begin{aligned}
&\lambda_j E_{0j}[\mu, \dot{\mu}_j]\\
&=(b_j\mu_0)^{2s-2}\lambda_j\dot{\lambda}_j\left(Z_{n+1}(y) + \frac{n-2s}{2}\alpha_{n, s}\frac{1}{\left(1+|y|^2\right)^{\frac{n-2s}{2}}}\right)\\
&\quad+ (2s-2)(b_j\mu_0)^{2s-3}\lambda_j^2(b_j\dot{\mu}_0+\dot{\lambda}_j)\left(Z_{n+1}(y) + \frac{n-2s}{2}\alpha_{n, s}\frac{1}{\left(1+|y|^2\right)^{\frac{n-2s}{2}}}\right) \\
&\quad +\lambda_jb_j^{2s-1}\left[\mu_0^{2s-2}\dot{\mu}_0\left(Z_{n+1}(y) + \frac{n-2s}{2}\alpha_{n, s}\frac{1}{\left(1+|y|^2\right)^{\frac{n-2s}{2}}}\right)\right.\\
&\quad + \left. pU(y)^{p-1}\mu_0^{n-2s-1}\left(-b_j^{n-4s}H(q_j, q_j) + \sum_{i\neq j}b_j^{\frac{n-6s}{2}}b_i^{\frac{n-2s}{2}}G(q_j, q_i)\right)\right]\\
&\quad + pU(y)^{p-1}b_j^{2s-1}\mu_0^{n-2s-1}B\lambda_j\\
&\quad -\mu_0^{n-2s-2}pU(y)^{p-1}\sum_{i,l=1}^kf_{ijl}(\mu_0^{-1}\lambda)\lambda_i\lambda_l,
\end{aligned}
\end{equation*}
where functions $f_{ijl}$ are smooth in its arguments.

Collecting all the above estimates, we have the full expansion for $S(u^*_{\mu, \xi})$.
\begin{lemma}\label{l2.2}
We consider the region $|x-q_j|\leq \frac{1}{2}\min_{i\neq l}|q_i-q_l|$ for a fixed index $j$. Let $\mu = \bar{\mu}_0 + \lambda$ and $|\lambda(t)|\leq \mu_0(t)^{1+\sigma}$ for some $0 < \sigma < \bar{\sigma}$ with $\bar{\sigma}$ be a constant satisfying $0 < \bar{\sigma}\leq \frac{n-2s}{2s}\bar{\sigma}_jb_j^{2-2s}$, $j = 1,\cdots,k$. Then for large $t$, $S(u^*_{\mu,\xi})$ can be expressed as
\begin{equation*}
\begin{aligned}
& S(u^*_{\mu, \xi}) =\sum_{j=1}^k\mu_j^{-\frac{n+2s}{2}}\\
&\quad\quad\Bigg\{\mu_{0j}(b_j\mu_0)^{2s-2}\dot{\lambda}_j\left(Z_{n+1}(y_j) + \frac{n-2s}{2}\alpha_{n, s}\frac{1}{\left(1+|y_j|^2\right)^{\frac{n-2s}{2}}}-2sApU(y_j)^{p-1}\right)\\
&\quad\quad\quad -\mu_{0j}\mu_0^{n-2s-2}pU(y_j)^{p-1}\sum_{i=1}^kM_{ij}\lambda_i+\mu_j^{2s-1}\alpha_{n, s}(n-2s)\frac{\dot{\xi}_j\cdot y_j}{\left(1+|y_j|^2\right)^{\frac{n-2s}{2}+1}}\\
&\quad\quad\quad +\mu_jpU(y_j)^{p-1}\Big[-\mu_j^{n-2s}\nabla H(q_j, q_j) + \sum_{i\neq j}\mu_j^{\frac{n-2s}{2}}\mu_i^{\frac{n-2s}{2}}\nabla G(q_j, q_i)\Big]\cdot y_j\Bigg\}\\
&\quad\quad + \sum_{j=1}^k\mu_j^{-\frac{n+2s}{2}}\lambda_jb_j^{2s-1}\\
&\quad\quad\quad\Bigg[(2s-1)\mu_0^{2s-2}\dot{\mu}_0\left(Z_{n+1}(y_j) + \frac{n-2s}{2}\alpha_{n, s}\frac{1}{\left(1+|y_j|^2\right)^{\frac{n-2s}{2}}}\right) \\
&\quad\quad\quad\quad +pU(y_j)^{p-1}\mu_0^{n-2s-1}\bigg(-b_j^{n-4s}H(q_j, q_j) + \sum_{i\neq j}b_j^{\frac{n-6s}{2}}b_i^{\frac{n-2s}{2}}G(q_j, q_i)\bigg)\\
\end{aligned}
\end{equation*}
\begin{equation*}
\begin{aligned}
&\quad\quad\quad +pU(y_j)^{p-1}\mu_0^{n-2s-1}(2s-1)B\Bigg]\\
&\quad\quad +\mu_0^{-\frac{n+2s}{2}}\Bigg[\sum_{j=1}^k\frac{\mu_0^{n-2s+2}g_j}{1+|y_j|^{4s-2}}+
\sum_{j=1}^k\frac{\mu_0^{2n-4s}g_j}{1+|y_j|^{2s}} + \sum_{j=1}^k\frac{\mu_0^{n-2s}g_j}{1+|y_j|^{4s}}\lambda_j\Bigg]\\
&\quad\quad +\mu_0^{-\frac{n+2s}{2}}\Bigg[\sum_{j=1}^k\frac{\mu_0^{n-2s}\vec{g}_j}{1+|y_j|^{4s}}\cdot(\xi_j-q_j)\Bigg]\\
&\quad\quad +\mu_0^{-\frac{n+2s}{2}}\left[\mu_0^{n-2s-2}\sum_{i,j,l=1}^kpU(y_j)^{p-1}f_{ijl}\lambda_i\lambda_l +\sum_{i,j,l=1}^k\frac{f_{ijl}}{1+|y_j|^{n-2s}}\lambda_i\dot{\lambda}_l\right]\\
&\quad\quad + \mu_0^{-\frac{n+2s}{2}}\left[\mu_0^{n+2s}f + \mu_0^{n-1}\sum_{i=1}^k\dot{\mu}_if_i + \mu_0^{n}\sum_{i=1}^k\dot{\xi}_i\vec{f}_i\right],
\end{aligned}
\end{equation*}
where $x = \xi_j + \mu_jy_j$, $\vec{f}_i$, $f_i$, $f$, $f_{ijl}$ are smooth and bounded functions depending on $(\mu_0^{-1}\mu, \xi, x)$ and $g_j$, $\vec{g}_j$ depend on $(\mu^{-1}_0\mu, \xi, y_j)$.
\end{lemma}

\section{The inner-outer gluing procedure}
We are looking for a solution to (\ref{e3:1}) with the form
\begin{equation*}
u = u^*_{\mu,\xi} + \tilde{\phi}
\end{equation*}
when $t_0$ is sufficiently large. The function $\tilde{\phi}(x, t)$ is a smaller term and we will find it by means of the \textbf{inner-outer gluing procedure}.

Let us write
\begin{equation*}
\tilde{\phi}(x, t) = \psi(x, t) + \phi^{in}(x, t)\quad\text{where}\quad\phi^{in}(x, t): = \sum_{j=1}^k\eta_{j, R}(x,t)\tilde{\phi}_j(x, t)
\end{equation*}
with
\begin{equation*}
\tilde{\phi}_j(x, t): = \mu_{0j}^{-\frac{n-2s}{2}}\phi_j\left(\frac{x-\xi_j}{\mu_{0j}}, t\right),\quad \mu_{0j}(t) = b_j\mu_0(t)
\end{equation*}
and
\begin{equation*}
\eta_{j, R}(x, t) = \eta\left(\frac{|x-\xi_j|}{R\mu_{0j}}\right).
\end{equation*}
Here $\eta(\tau)$ is a smooth cut-off function defined on $[0, +\infty)$ with $\eta(\tau) = 1$ for $0\leq \tau < 1$ and $\eta(\tau)= 0$ for $\tau > 2$. The number $R$ is defined as
\begin{equation}\label{e3:20}
R = t_0^{\rho}\quad \text{with}\quad 0 < \rho \ll 1.
\end{equation}
Problem (\ref{e3:1}) can be written with respect to $\tilde{\phi}$ as
\begin{equation}\label{e3:5}
\begin{cases}
\partial_t\tilde{\phi} = -(-\Delta)^{s}\tilde{\phi} + p(u^*_{\mu, \xi})^{p-1}\tilde{\phi} + \tilde{N}(\tilde{\phi}) + S(u^*_{\mu, \xi})&\text{in}\quad \Omega\times (t_0, \infty),\\
\tilde{\phi} = -u^*_{\mu, \xi}&\text{in}\quad (\mathbb{R}^n\setminus \Omega)\times (t_0, \infty),
\end{cases}
\end{equation}
where $\tilde{N}(\tilde{\phi})=(u_{\mu.\xi}^*+\tilde{\phi})^p-p(u_{\mu.\xi}^*)^{p-1}\tilde{\phi}-(u_{\mu.\xi}^*)^p~\mbox{and}~ S(u^*_{\mu, \xi})=-\partial_t \mu^*_{\mu, \xi}-(-\Delta)^s u^*_{\mu, \xi}+(u^*_{\mu, \xi})^p.$

According to Lemma \ref{l2.2}, we let $y_j = \frac{x-\xi_j}{\mu_j}$ and denote
\begin{equation*}
S(u^*_{\mu, \xi}) = \sum_{j=1}^kS_{\mu, \xi, j} + S^{(2)}_{\mu, \xi}
\end{equation*}
where
\begin{equation*}
\begin{aligned}
S_{\mu, \xi, j}=&\mu_j^{-\frac{n+2s}{2}}\Bigg\{\mu_{0j}(b_j\mu_0)^{2s-2}\dot{\lambda}_j\\
&\quad\quad\quad\quad\quad\times\left(Z_{n+1}(y_j) + \frac{n-2s}{2}\alpha_{n, s}\frac{1}{\left(1+|y_j|^2\right)^{\frac{n-2s}{2}}}-2sApU(y_j)^{p-1}\right)\\
&\quad\left.-\mu_{0j}\mu_0^{n-2s-2}pU(y_j)^{p-1}\sum_{i=1}^kM_{ij}\lambda_i\right.\\
&\quad +\lambda_jb_j^{2s-1}\Bigg[(2s-1)\mu_0^{2s-2}\dot{\mu}_0\left(Z_{n+1}(y_j) + \frac{n-2s}{2}\alpha_{n, s}\frac{1}{\left(1+|y_j|^2\right)^{\frac{n-2s}{2}}}\right)\\
&\quad +pU(y_j)^{p-1}\mu_0^{n-2s-1}\bigg(-b_j^{n-4s}H(q_j, q_j) + \sum_{i\neq j}b_j^{\frac{n-6s}{2}}b_i^{\frac{n-2s}{2}}G(q_j, q_i)\bigg)\\
&\quad +pU(y_j)^{p-1}\mu_0^{n-2s-1}(2s-1)B\Bigg]\\
&\quad +\mu_j\left[\mu_j^{2s-2}\alpha_{n, s}(n-2s)\frac{\dot{\xi}_j\cdot y_j}{\left(1+|y_j|^2\right)^{\frac{n-2s}{2}+1}}\right.\\
&\quad +\left.pU(y_j)^{p-1}\left(-\mu_j^{n-2s}\nabla H(q_j, q_j) + \sum_{i\neq j}\mu_j^{\frac{n-2s}{2}}\mu_i^{\frac{n-2s}{2}}\nabla G(q_j, q_i)\right)\cdot y_j\right]\Bigg\}.
\end{aligned}
\end{equation*}
Set
\begin{equation}\label{e3:7}
V_{\mu, \xi} = p\sum_{j=1}^k\left((u^*_{\mu, \xi})^{p-1} - \left[\mu_j^{-\frac{n-2s}{2}}U\left(\frac{x-\xi_j}{\mu_j}\right)\right]^{p-1}\right)\eta_{j,R} + p(1-\sum_{j=1}^k\eta_{j, R})(u^*_{\mu, \xi})^{p-1}.
\end{equation}
Then $\tilde{\phi}$ solves problem (\ref{e3:5}) if

\noindent (1) $\psi$ solves the \textbf{outer problem}
\begin{equation}\label{outerproblem}
\left\{
\begin{aligned}
&\partial_t\psi =-(-\Delta)^{s}\psi + V_{\mu, \xi}\psi \\
&\quad\quad\quad+\sum_{j=1}^k\left\{\big[-(-\Delta)^{\frac{s}{2}}\eta_{j, R},-(-\Delta)^{\frac{s}{2}}\tilde{\phi}_j\big]+ \tilde{\phi}_j\big(-(-\Delta)^{s} -\partial_t\big)\eta_{j, R}\right\}\\
&\quad\quad\quad + \tilde{N}_{\mu, \xi}(\tilde{\phi}) + S_{out},\qquad \text{ in }\Omega\times(t_0,\infty),\\
&\psi = -u^*_{\mu, \xi}\quad\text{in}\quad (\mathbb{R}^n\setminus\Omega)\times (t_0,\infty),
\end{aligned}
\right.
\end{equation}
where
\begin{equation}\label{e3:8}
S_{out} = S^{(2)}_{\mu, \xi} + \sum_{j = 1}^k(1 - \eta_{j, R})S_{\mu, \xi, j}
\end{equation}
and
\begin{equation}\label{liebracket}
\big[-(-\Delta)^{\frac{s}{2}}\eta_{j, R},-(-\Delta)^{\frac{s}{2}}\tilde{\phi}_j\big]:=C_{n,s}\int_{\mathbb{R}^n} \frac{[\eta_{j,R}(y)-\eta_{j,R}(x)][\tilde{\phi}_j(x)-\tilde{\phi}_j(y)]}{|x-y|^{n+2s}}dy.
\end{equation}

\noindent (2) $\tilde{\phi}_j$ solves
\begin{equation}\label{e3:9}
\eta_{j, R}\partial_t\tilde{\phi}_j = \eta_{j, R}\left[-(-\Delta)^s\tilde{\phi}_j + pU_j^{p-1}\tilde{\phi}_j + pU_j^{p-1}\psi + S_{\mu, \xi, j}\right]\text{ in }\mathbb{R}^n\times (t_0, \infty),
\end{equation}
for all $j = 1, \cdots, k$, $U_j := \mu_j^{-\frac{n-2s}{2}}U\left(\frac{x-\xi_j}{\mu_j}\right)$.
In terms of $\phi_j(y, t)$, (\ref{e3:9}) becomes the \textbf{inner problem}
\begin{equation}\label{e3:10}
\left\{
\begin{aligned}
&\mu_{0j}^{2s}\partial_t\phi_j = -(-\Delta)^s_y\phi_j + pU^{p-1}(y)\phi_j\\
&\quad\quad\quad\quad + \Bigg\{\mu_{0j}^{\frac{n+2s}{2}}S_{\mu, \xi, j}(\xi_j + \mu_{0j}y, t)+ p\mu_{0j}^{\frac{n-2s}{2}}\frac{\mu_{0j}^{2s}}{\mu_j^{2s}}U^{p-1}(\frac{\mu_{0j}}{\mu_j}y)\psi(\xi_j + \mu_{0j}y, t)\\
&\quad\quad\quad\quad\quad\quad + B_j[\phi_j] + B_j^0[\phi_j]\Bigg\}\chi_{B_{2R}(0)}(y)\quad\text{in}\quad \mathbb{R}^n\times (t_0, \infty),
\end{aligned}
\right.
\end{equation}
for $j = 1, \cdots, k$, where
\begin{equation}\label{e3:11}
B_j[\phi_j]:=\mu_{0j}^{2s-1}\dot{\mu}_{0j}\left(\frac{n-2s}{2}\phi_j + y\cdot\nabla_y\phi_j\right) + \mu_{0j}^{2s-1}\nabla\phi_j\cdot\dot{\xi}_j
\end{equation}
and
\begin{equation}\label{e3:12}
B_j^0[\phi_j]:= p\left[U^{p-1}\left(\frac{\mu_{0j}}{\mu_j}y\right)-U^{p-1}(y)\right]\phi_j + p\left[\mu_{0j}^{2s}(u^{*}_{\mu,\xi})^{p-1} -U^{p-1}\left(\frac{\mu_{0j}}{\mu_j}y\right)\right]\phi_j.
\end{equation}
Here $\chi_{B_{2R}(0)}(y)$ is the characteristic function of $B_{2R}(0)$.

The rest of the paper is organized as follows. In Section 4, we solve the outer problem (\ref{outerproblem}). In Section 5.1, a linear theory for the inner problem (\ref{e3:10}) is developed. We study the solvability conditions for (\ref{e3:10}) in Section 5.2 and the full problem is finally solved in Section 6.

\section{The outer problem}
In this section, we shall solve the outer problem for a given smooth function $\phi$ which is sufficiently small.
We consider the initial boundary value problem
\begin{equation}\label{e4:main}
\left\{
\begin{aligned}
&\partial_t\psi=-(-\Delta)^{s}\psi + V_{\mu, \xi}\psi\\
&\quad\quad\quad+ \sum_{j=1}^k \left\{\big[-(-\Delta)^{\frac{s}{2}}\eta_{j, R},-(-\Delta)^{\frac{s}{2}}\tilde{\phi}_j\big] + \tilde{\phi}_j\big(-(-\Delta)^{s} -\partial_t\big)\eta_{j, R}\right\}\\
&\quad\quad\quad+ \tilde{N}_{\mu, \xi}(\tilde{\phi})+ S_{out},\qquad\text{ in }\Omega\times (t_0,\infty),\\
&\psi= -u^*_{\mu, \xi}\quad\text{in}\quad(\mathbb{R}^n\setminus\Omega)\times (t_0,\infty),\\
&\psi(\cdot,t_0)=\psi_0\quad\text{in}\quad\mathbb{R}^n
\end{aligned}
\right.
\end{equation}
with a smooth and sufficiently small initial condition $\psi_0$.

\subsection{The model problem}
To solve problem (\ref{e4:main}), we first consider the linear problem
\begin{equation}\label{e4:3}
\begin{cases}
\partial_t\psi =
-(-\Delta)^{s}\psi + V_{\mu, \xi}\psi + f(x, t)\quad&\text{in}\quad\Omega\times (t_0, \infty),\\
\psi = g\quad&\text{in}\quad (\mathbb{R}^n\setminus\Omega)\times (t_0,\infty),\\
\psi(\cdot,t_0) = h\quad&\text{in}\quad\mathbb{R}^n,
\end{cases}
\end{equation}
where $f(x, t)$, $g(x, t)$ and $h(x)$ are given smooth functions and $V_{\mu, \xi}$ is defined in (\ref{e3:7}). Furthermore, we assume $f$ satisfies
\begin{equation}\label{e4:2}
|f(x, t)|\leq M\sum_{j=1}^k\frac{\mu_j^{-2s}t^{-\beta}}{1+|y_j|^{2s+\alpha}},\quad y_j = \frac{|x-\xi_j|}{\mu_j}
\end{equation}
for $\alpha$, $\beta > 0$. Denote the least $M > 0$ in (\ref{e4:2}) by $\|f\|_{*, \beta, 2s+\alpha}$.

In what follows, we use the symbol $a\lesssim b$ to denote $a\leq C b$ for a positive constant $C$ independent of $t$ and $t_0$. Then we have the following {\itshape a priori} estimate for the model problem \eqref{e4:3}.

\begin{lemma}\label{l4:lemma4.1}
Suppose that $\|f\|_{*, \beta, 2s+\alpha} < +\infty$ for some $\alpha$, $\beta > 0$, $0 < \alpha \ll 1$, $\|h\|_{L^\infty(\mathbb{R}^n)} < +\infty$ and
$\|\tau^\beta g(x, \tau)\|_{L^\infty((\mathbb{R}^n\setminus\Omega)\times (t_0,\infty))} < +\infty$. Let $\phi = \psi[f, g, h]$ be the solution of problem (\ref{e4:3}). Then there exists $\delta = \delta(\Omega) > 0$ small such that for all $(x,t)$ we have
\begin{equation}\label{e4:4}
\begin{aligned}
|\psi(x, t)|&\lesssim \|f\|_{*, \beta, 2s+\alpha}\left(\sum_{j=1}^k\frac{t^{-\beta}}{1+|y_j|^{\alpha}}\right)\\
&\quad + e^{-\delta(t-t_0)}\|h\|_{L^\infty(\mathbb{R}^n)} + t^{-\beta}\|\tau^\beta g(x, \tau)\|_{L^\infty((\mathbb{R}^n\setminus\Omega)\times (t_0,\infty))},
\end{aligned}
\end{equation}
where $y_j=\frac{x-\xi_j}{\mu_j}$. Moreover, the following H\"{o}lder estimate
\begin{equation}\label{e4:5}
[\psi(\cdot, t)]_{\eta, B_{\mu_j}(\xi_j)}\lesssim \|f\|_{*, \beta, 2s+\alpha}\left(\sum_{j=1}^k\frac{\mu_j^{-\eta}t^{-\beta}}{1+|y_j|^{\alpha+\eta}}\right)
\end{equation}
holds for some $\eta\in (0, 1)$ and $|y_j|\leq 2 R$. Here $$[\psi(\cdot, t)]_{\eta, B_{\mu_j}(\xi_j)}:= \sup_{x, y\in B_{\mu_j}(\xi_j)}\frac{|\psi(x, t) - \psi(y, t)|}{|x-y|^{\eta}}$$ is the H\"{o}lder seminorm.
\end{lemma}
\begin{proof}
Let $\psi_0[g, h]$ be the solution of the fractional heat equation
\begin{equation}\label{e4:6}
\left\{
\begin{array}{lll}
\partial_t\psi_0 =
-(-\Delta)^{s}\psi_0&\text{ in }\Omega\times (t_0, \infty),\\
\psi_0 = g&\text{ in }(\mathbb{R}^n\setminus\Omega)\times (t_0,\infty),\\
\psi_0(\cdot,t_0) = h&\text{ in }\mathbb{R}^n.
\end{array}
\right.
\end{equation}
Let $v(x)$ be the bounded solution of $-(-\Delta)^{s}v + 1 = 0$ in $\Omega$ with $v = 1$ on $\mathbb{R}^n\setminus\Omega$. Then $v\geq 1$ in $\Omega$ and by direct computations, the function
\begin{equation*}
\bar{\psi}(x, t) = \left(e^{-\delta(t-t_0)}\|h\|_{L^\infty(\mathbb{R}^n)} + t^{-\beta}\|\tau^\beta g(x, \tau)\|_{L^\infty((\mathbb{R}^n\setminus\Omega)\times (t_0,\infty))}\right)v(x)
\end{equation*}
is a supersolution to (\ref{e4:6}) if $\delta = \delta(\Omega) > 0$ is sufficiently small. Then $|\psi_0|\leq \bar{\psi}$ by the maximum principle (see, for example, \cite{CabreSire2014Nonlinear} and \cite{cabresire2015tamsnonlinear}). Thus, it suffices to prove the estimates (\ref{e4:4}) and (\ref{e4:5}) for the case $g = 0$, $h=0$.

Let  $p(|z|)$ be the radial positive solution of the equation
\begin{equation*}
-(-\Delta)^sp + 4q = 0\text{ in }\mathbb{R}^n
\end{equation*}
with $q(|z|) = \frac{1}{1+|z|^{2s+\alpha}}$. Then by Riesz kernel, we get $p(z)\sim \frac{1}{1+|z|^{\alpha}}$. For a given sufficiently small $\delta > 0$, we have
\begin{equation*}
-(-\Delta)^sp + \frac{\delta}{1+|z|^{2s}}p+2q \leq 0\text{ in }\mathbb{R}^n.
\end{equation*}
Thus $\bar{p}(x):=\sum_{j=1}^kp\left(\frac{x-\xi_j}{\mu_j}\right)$ satisfies
\begin{equation*}
-(-\Delta)^s\bar{p} + \left(\sum_{j=1}^k\mu_j^{-2s}\frac{\delta}{1+|\frac{x-\xi_j}{\mu_j}|^{2s}}\right)\bar{p}+\frac{3}{2}\bar{q} \leq 0\text{ in }\mathbb{R}^n
\end{equation*}
with $\bar{q}:=\sum_{j=1}^k{\mu_j^{-2s}}q\left(\frac{x-\xi_j}{\mu_j}\right)$.
From the definition of $V_{\mu, \xi}$, we have
\begin{equation}\label{e4:11}
|V_{\mu, \xi}|\lesssim \sum_{j=1}^k\mu_j^{-2s}\frac{R^{-2s}}{1+|y_j|^{2s}}.
\end{equation}
For a given number $\beta > 0$, it is easy to see that $\bar{\psi}(x, t) = 2t^{-\beta}\bar{p}$ is a positive supersolution to
\begin{equation*}
\partial_t\bar{\psi} = -(-\Delta)^s\bar{\psi} + V_{\mu,\xi}\bar{\psi} + t^{-\beta}\bar{q},
\end{equation*}
i.e.,
$$\partial_t\bar{\psi} \geq -(-\Delta)^s\bar{\psi} + V_{\mu,\xi}\bar{\psi} + t^{-\beta}\bar{q}$$
for $t > t_0$ and $t_0$ is sufficiently large. Therefore, one has
\begin{equation}\label{e4:13}
|\psi(x, t)|\lesssim t^{-\beta}\|f\|_{*, \beta, 2s+\alpha}\sum_{j=1}^k\frac{1}{1+|y_j|^\alpha}.
\end{equation}
and \eqref{e4:4} is proved.

To prove (\ref{e4:5}), let
\begin{equation*}
\psi(x, t):=\tilde{\psi}\left(\frac{x-\xi_j}{\mu_j},\tau(t)\right)
\end{equation*}
where $\dot{\tau}(t) = \mu_j^{-2s}(t)$, namely $\tau(t)\thicksim t^{\frac{n-2s}{n-4s}}$. Without loss of generality, we assume $\tau(t_0)\geq 2$ by fixing $t_0$. Then $\tilde{\psi}$ satisfies
\begin{equation}\label{e4:15}
\partial_\tau\tilde{\psi} = -(-\Delta)^s\tilde{\psi} + a(z, t)\cdot \nabla_z\tilde{\psi} + b(z, t)\tilde{\psi} + \tilde{f}(z, \tau)
\end{equation}
for $|z|\leq \delta\mu_0^{-1}$ and $\tilde{f}(z, \tau) = \mu_j^{2s}f(\xi_j + \mu_jz, t(\tau))$.
The uniformly small coefficients $a(z, t)$ and $b(z, t)$ in \eqref{e4:15} are given by
\begin{equation*}
a(z, t):=\mu_j^{2s-1}\dot{\mu}_jz+\mu_j^{2s-1}\dot{\xi}_j,\quad b(z, t) = \mu_j^{2s}V_{\mu, \xi}(\xi_j+\mu_jz).
\end{equation*}
Then from assumption \eqref{e4:2} and \eqref{e4:13} we have
\begin{equation*}\label{e4:18}
|\tilde{f}(z, \tau)|\lesssim t(\tau)^{-\beta}\frac{\|f\|_{*,\beta, 2s+\alpha}}{1+|z|^{2s+\alpha}}
\end{equation*}
and
\begin{equation*}\label{e4:19}
|\tilde{\psi}(z, \tau)|\lesssim t(\tau)^{-\beta}\frac{\|f\|_{*, \beta, 2s+\alpha}}{1+|z|^\alpha}.
\end{equation*}
Now fix $0 < \eta < 1$, from the regularity estimates for parabolic integro-differential equations (see \cite{silvestreium2012differentiability}), for $\tau_1 \geq \tau(t_0)+2$, we have
\begin{equation*}
\begin{aligned}
~[\tilde{\psi}(\cdot,\tau_1)]_{\eta, B_{10}(0)} &\lesssim \|\tilde{\psi}\|_{L^\infty}+\|\tilde{f}\|_{L^\infty}\\
&\lesssim t(\tau_1-1)^{-\beta}\|f\|_{*, \beta, 2s+\alpha}\\
&\lesssim t(\tau_1)^{-\beta}\|f\|_{*, \beta, 2s+\alpha}.
\end{aligned}
\end{equation*}
Therefore, choosing an appropriate constant $c_n$ such that for any $t\geq c_nt_0$ we have
\begin{equation}\label{e4:20}
(R\mu_j)^{\eta}[\psi(\cdot,t)]_{\eta, B_{10\mu_j}(\xi_j)}\lesssim t^{-\beta}\|f\|_{*, \beta, 2s+\alpha}.
\end{equation}
By the same token,  the estimate (\ref{e4:20}) also holds for $t_0\leq t\leq c_nt_0$. Thus, (\ref{e4:5}) holds for any $t\geq t_0$. The proof is completed.
\end{proof}

\subsection{Solvability of the outer problem.}
Now we fix $\sigma$ satisfying
\begin{equation}\label{e4:29}
0 < \sigma < \bar{\sigma}\text{ where }\bar{\sigma}\leq \frac{n-2s}{2s}\bar{\sigma}_jb_j^{2-2s},\quad j = 1,\cdots,k,
\end{equation}
and $\bar{\sigma}_j$ and $b_j$ are defined in (\ref{e2:46}) and (\ref{e2:42}) respectively. Given $h(t):(t_0, \infty)\to\mathbb{R}^k$ and $\delta > 0$, we define the weighted $L^\infty$ norm as
\begin{equation*}\label{e4:30}
\|h\|_\delta:=\|\mu_0(t)^{-\delta}h(t)\|_{L^\infty(t_0, \infty)}.
\end{equation*}
In the rest of this paper, we always assume that $a$ is a positive constant satisfying $a > 2s$ and $a-2s$ is sufficiently small.
We also assume the parameters $\lambda$, $\xi$, $\dot{\lambda}$, $\dot{\xi}$ satisfy the following two constraints,
\begin{equation}\label{e4:21}
\|\dot{\lambda}(t)\|_{n-4s+1+\sigma} + \|\dot{\xi}(t)\|_{n-4s+1+\sigma}\leq \frac{c}{R^{a-2s}},
\end{equation}
\begin{equation}\label{e4:22}
\|\lambda(t)\|_{1+\sigma} + \|\xi(t)-q\|_{1+\sigma}\leq \frac{c}{R^{a-2s}},
\end{equation}
where $c$ is a positive constant independent of $t$, $t_0$ and $R$.

Denote
\begin{equation*}\label{e4:23}
\|\phi\|_{n-2s+\sigma,a}=\max_{j=1,\cdots, k}\|\phi_j\|_{n-2s+\sigma,a},
\end{equation*}
where $\|\phi_j\|_{n-2s+\sigma,a}$ is defined as the least number $M$ such that
\begin{equation}\label{e4:24}
(1+|y|)|\nabla_y \phi_j(y, t)|\chi_{B_{2R}(0)}(y) + |\phi_j(y, t)|\leq M\frac{\mu_0^{n-2s+\sigma}}{1+|y|^a},\quad j= 1,\cdots,k
\end{equation}
holds. Suppose $\phi = (\phi_1,\cdots, \phi_k)$ satisfies
\begin{equation}\label{e4:25}
\|\phi\|_{n-2s+\sigma,a}\leq ct_0^{-\varepsilon}
\end{equation}
for some small $\varepsilon > 0$.
Then we have the following proposition.
\begin{prop}\label{p4:4.1}
Assume $\lambda$, $\xi$, $\dot{\lambda}$, $\dot{\xi}$ satisfy (\ref{e4:21}) and (\ref{e4:22}), $\phi = (\phi_1,\cdots, \phi_k)$ satisfies (\ref{e4:25}), $\psi_0\in C^2(\mathbb{R}^n)$ and we have
\begin{equation*}\label{e4:26}
\|\psi_0\|_{L^\infty(\mathbb{R}^n)} + \|\nabla\psi_0\|_{L^\infty(\mathbb{R}^n)}\leq \frac{t_0^{-\varepsilon}}{R^{a-2s}}.
\end{equation*}
Then there exists $t_0$ sufficiently large such that the outer problem (\ref{e4:main}) has a unique solution $\psi = \Psi[\lambda, \xi, \dot{\lambda}, \dot{\xi}, \phi]$. Moreover, there exists $\sigma$ satisfying \eqref{e4:29} and $\varepsilon > 0$ small such that, for $y_j=\frac{x-\xi_j}{\mu_{0j}}$,
\begin{equation}\label{e4:27}
|\psi(x, t)|\lesssim \frac{t_0^{-\varepsilon}}{R^{a-2s}}\sum_{j=1}^k\frac{\mu_0^{\frac{n-2s}{2}+\sigma}(t)}{1+|y_j|^{a-2s}} + e^{-\delta(t-t_0)}\|\psi_0\|_{L^{\infty}(\mathbb{R}^n)},
\end{equation}
and
\begin{equation}\label{e4:28}
[\psi(x, t)]_{\eta, B_{\mu_jR}(\xi_j)}\lesssim \frac{t_0^{-\varepsilon}}{R^{a-2s}}\sum_{j=1}^k\frac{\mu_j^{-\eta}\mu_0^{\frac{n-2s}{2}+\sigma}(t)}{1+|y_j|^{a-2s+\eta}}\text{ for } |y_j|\leq 2 R,
\end{equation}
where $R$, $\rho$ are defined in \eqref{e3:20} and $\eta\in (0, 1)$.
\end{prop}

Proposition \ref{p4:4.1} states that for any small initial conditions $\psi_0$, a solution $\psi$ to (\ref{e4:main}) exists.  Moreover, it clarifies the dependence of $\Psi[\lambda, \xi, \dot{\lambda}, \dot{\xi}, \phi]$ in the parameters $\lambda, \xi, \dot{\lambda}, \dot{\xi}, \phi$, which is proved by estimating, for example,
\begin{equation*}\label{e4:31}
\partial_\phi\Psi[\lambda, \xi, \dot{\lambda}, \dot{\xi}, \phi][\bar{\phi}] = \partial_s\Psi[\lambda, \xi, \dot{\lambda}, \dot{\xi}, \phi+s\bar{\phi}]|_{s=0}
\end{equation*}
as a linear operator between parameter Banach spaces. For simplicity, we denote the above operator by $\partial_\phi\Psi[\bar{\phi}]$. Similarly, we have $\partial_\lambda\Psi[\bar{\lambda}]$, $\partial_\xi\Psi[\bar{\xi}]$, $\partial_{\dot{\lambda}}\Psi[\dot{\bar{\lambda}}]$, $\partial_{\dot{\xi}}\Psi[\dot{\bar{\xi}}]$.

\begin{proof}
Lemma \ref{l4:lemma4.1} defines a linear operator $T$ which associates the solution $\psi = T(f,g,h)$ of problem (\ref{e4:3}) to any given functions $f(x, t)$, $g(x, t)$ and $h(x)$. Denote $\psi_1(x, t) := T(0, -u^*_{\mu, \xi}, \psi_0)$. From (\ref{e2:52}), (\ref{e2:3}) and (\ref{e2:51}), for any $x\in\mathbb{R}^n\setminus\Omega$ we have
\begin{equation}\label{e4:32}
|u^*_{\mu,\xi}(x,t)|\lesssim \mu_0^{\frac{n+2s}{2}}(t).
\end{equation}
By Lemma \ref{l4:lemma4.1}, we have
\begin{equation*}\label{e4:33}
|\psi_1|\lesssim e^{-\delta(t-t_0)}\|\psi_0\|_{L^\infty(\mathbb{R}^n)} + t^{-\beta}\mu_0(t_0)^{2s-\sigma}\text{ where }\beta = \frac{n-2s}{2(n-4s)}+\frac{\sigma}{n-4s}.
\end{equation*}
Therefore, the function $\psi+\psi_1$ is a solution to (\ref{e4:main}) if $\psi$ is a fixed point for the operator
\begin{equation*}\label{e4:34}
\mathcal{A}(\psi):=T(f(\psi),0,0),
\end{equation*}
where
\begin{equation}\label{e4:35}
f(\psi)= \sum_{j=1}^k\left\{\big[-(-\Delta)^{\frac{s}{2}}\eta_{j, R},-(-\Delta)^{\frac{s}{2}}\tilde{\phi}_j\big] + \tilde{\phi}_j\big(-(-\Delta)^{s} -\partial_t\big)\eta_{j, R}\right\} + \tilde{N}_{\mu, \xi}(\tilde{\phi}) + S_{out}.
\end{equation}
By the Contraction Mapping Theorem, we will prove the existence of a fixed point $\psi$ for $\mathcal{A}$ in the following function space
\begin{equation*}\label{e4:36}
\|\psi\|_{**,\beta,a}\text{ is bounded with }\beta = \frac{n-2s}{2(n-4s)}+\frac{\sigma}{n-4s}.
\end{equation*}
Here $\|\psi\|_{**,\beta,a}$ is the least $M > 0$ such that the following inequality holds
$$
|\psi(x, t)|\leq M\sum_{j=1}^k\frac{t^{-\beta}}{1+|y_j|^{a-2s}},\quad y_j = \frac{|x-\xi_j|}{\mu_j}.
$$

As a first step, we establish the following estimates.
\begin{itemize}
\item[(1)] Estimate for $S_{out}(x, t)$:
\begin{equation}\label{e4:37}
|S_{out}(x, t)|\lesssim \frac{t_0^{-\varepsilon}}{R^{a-2s}}\sum_{j=1}^k\frac{\mu_j^{-2s}\mu_0^{\frac{n-2s}{2}+\sigma}(t)}{1+|y_j|^{a}}.
\end{equation}
\item[(2)] Estimate for $\sum_{j=1}^k\left\{\big[-(-\Delta)^{\frac{s}{2}}\eta_{j, R},-(-\Delta)^{\frac{s}{2}}\tilde{\phi}_j\big] + \tilde{\phi}_j\big(-(-\Delta)^{s} -\partial_t\big)\eta_{j, R}\right\}$:
\begin{equation}\label{e4:38}
\begin{aligned}
&\left|\sum_{j=1}^k\left\{\big[-(-\Delta)^{\frac{s}{2}}\eta_{j, R},-(-\Delta)^{\frac{s}{2}}\tilde{\phi}_j\big]+ \tilde{\phi}_j\big(-(-\Delta)^{s} -\partial_t\big)\eta_{j, R}\right\}\right|\\
&\quad\quad\lesssim \frac{1}{R^{a-2s}}\|\phi\|_{n-2s+\sigma,a}\sum_{j=1}^k\frac{\mu_j^{-2s}\mu_0^{\frac{n-2s}{2}+\sigma}(t)}{1+|y_j|^{a}}.
\end{aligned}
\end{equation}
\item[(3)] Estimate for $\tilde{N}_{\mu, \xi}(\tilde{\phi})$:
\begin{equation}\label{e4:39}
\begin{aligned}
&\tilde{N}_{\mu, \xi}(\tilde{\phi}) \lesssim\\
&\left\{
\begin{aligned}
    t_0^{-\varepsilon}(\|\phi\|^2_{n-2s+\sigma,a}+\|\psi\|^2_{**,\beta,a})\frac{1}{R^{a-2s}}\sum_{j=1}^k\frac{\mu_j^{-2s}\mu_0^{\frac{n-2s}{2}+\sigma}(t)}{1+|y_j|^{a}},  & \quad \text{when } 6s \geq n,\\
    t_0^{-\varepsilon}(\|\phi\|^p_{n-2s+\sigma,a}+\|\psi\|^p_{**,\beta,a})\frac{1}{R^{a-2s}}\sum_{j=1}^k\frac{\mu_j^{-2s}\mu_0^{\frac{n-2s}{2}+\sigma}(t)}{1+|y_j|^{a}},       & \quad \text{when } 6s < n.\\
  \end{aligned}
\right.
\end{aligned}
\end{equation}
\end{itemize}

{\it Proof of (\ref{e4:37})}. Recall from \eqref{e3:8} that
\begin{equation*}\label{e4:41}
S_{out} = S^{(2)}_{\mu, \xi} + \sum_{j = 1}^k(1 - \eta_{j, R})S_{\mu, \xi, j}.
\end{equation*}
By (\ref{awayq}) and Lemma \ref{l2.2}, in the region $|x-q_j |>\delta$ with $\delta > 0$ small, $S_{out}$ can be estimated for all $j$ as
\begin{equation}\label{e4:42}
|S_{out}(x, t)|\lesssim \mu_0^{\frac{n-2s}{2}}(\mu_0^{2s}+\mu_0^{n-4s})\lesssim \mu_0^{\min(n-4s,2s)-(a-2s)-\sigma}(t_0)\sum_{j=1}^k\frac{\mu_j^{-2s}\mu_0^{\frac{n-2s}{2}+\sigma}}{1+|y_j|^{a}}.
\end{equation}
Now we consider the region $|x-q_j| \leq \delta$ with $\delta > 0$ small, where $j\in \{1,\cdots, k\}$ is fixed. Lemma \ref{l2.2} implies that
\begin{equation}\label{e4:43}
\left|S^{(2)}_{\mu,\xi}(x, t)\right|\lesssim \mu_0^{-\frac{n+2s}{2}}\frac{\mu_0^{n-2s+2}}{1+|y_j|^{4s-2}}\lesssim  \mu_0^{2s-(a-2s)-\sigma}(t_0)\sum_{j=1}^k\frac{\mu_j^{-2s}\mu_0^{\frac{n-2s}{2}+\sigma}}{1+|y_j|^{a}}.
\end{equation}
From the definition of $\eta_{j, R}$, $(1-\eta_{j, R})\neq 0$ if $|x-\xi_j|>\mu_0R$. Therefore, in the region $|x-q_j| < \delta$,
\begin{equation}\label{e4:44}
\left|(1 - \eta_{j, R})S_{\mu, \xi, j}\right|\lesssim \left(\frac{1}{R^{n-2s-a}}+\frac{1}{R^{4s-a}}\right)\frac{1}{R^{a-2s}}\sum_{j=1}^k\frac{\mu_j^{-2s}\mu_0^{\frac{n-2s}{2}+\sigma}}{1+|y_j|^{a}}.
\end{equation}
Here we have used the decaying assumptions \eqref{e4:21} and \eqref{e4:22} for $\lambda$ and $\xi$, respectively. Thus, (\ref{e4:37}) is valid.

{\it Proof of (\ref{e4:38})}. First, we consider $\big[-(-\Delta)^{\frac{s}{2}}\eta_{j, R},-(-\Delta)^{\frac{s}{2}}\tilde{\phi}_j\big]$ for $j$ fixed. Recall that
$$\tilde{\phi}_j(x, t): = \mu_{0j}^{-\frac{n-2s}{2}}\phi_j\left(\frac{x-\xi_j}{\mu_{0j}}, t\right).$$
From the assumptions \eqref{e4:24} and \eqref{e4:25}, we obtain
\begin{equation}\label{4.46}
\begin{aligned}
&\bigg|\left(\big[-(-\Delta)^{\frac{s}{2}}\eta_{j, R},-(-\Delta)^{\frac{s}{2}}\tilde{\phi}_j\big]\right)(x, t)\bigg|\\
&\lesssim \left[\int_{\mathbb{R}^n} \left(\frac{\eta_{j,R}(x)-\eta_{j,R}(y)}{|x-y|^{\frac{n}{2}+s}}\right)^2dy\right]^{\frac{1}{2}}\left[\int_{\mathbb{R}^n} \left(\frac{\tilde{\phi}_{j}(x)-\tilde{\phi}_{j}(y)}{|x-y|^{\frac{n}{2}+s}}\right)^2dy\right]^{\frac{1}{2}}\\
&\lesssim \frac{1}{R^s\mu_{0j}^s}\left[\int_{\mathbb{R}^n} \left(\frac{\eta(|\frac{x-\xi_j}{R\mu_{0j}}|)-\eta(|\frac{y-\xi_j}{R\mu_{0j}}|)}
{|\frac{x-y}{R\mu_{0j}}|^{\frac{n}{2}+s}}\right)^2d\left(\frac{y-\xi_j}{R\mu_{0j}}\right)\right]^{\frac{1}{2}}\\
&\quad\times\frac{\mu_0^{-\frac{n-2s}{2}}}{\mu_{0j}^s}\left[\int_{\mathbb{R}^n} \left(\frac{{\phi}_{j}(\frac{x-\xi_j}{\mu_{0j}},t)-{\phi}_{j}(\frac{y-\xi_j}{\mu_{0j}},t)}{|\frac{x-y}
{\mu_{0j}}|^{\frac{n}{2}+s}}\right)^2d\left(\frac{y-\xi_j}{\mu_{0j}}\right)\right]^{\frac{1}{2}}\\
&\lesssim \frac{1}{R^s\mu_{0j}^{2s}}\left[\int_{\mathbb{R}^n} \left(\frac{\eta(|\frac{x-\xi_j}{R\mu_{0j}}|)-\eta(|\frac{y-\xi_j}{R\mu_{0j}}|)}
{|\frac{x-y}{R\mu_{0j}}|^{\frac{n}{2}+s}}\right)^2d\left(\frac{y-\xi_j}{R\mu_{0j}}\right)\right]^{\frac{1}{2}}\frac{\mu_0^{\frac{n-2s}{2}+\sigma}}{(1+|y_j|^{s+a})}
\|\phi\|_{n-2s+\sigma,a}\\
&\lesssim \frac{1}{R^{a-2s}}\|\phi\|_{n-2s+\sigma,a}\sum_{j=1}^k\frac{\mu_j^{-2s}\mu_0^{\frac{n-2s}{2}+\sigma}(t)}{1+|y_j|^{a}}.
\end{aligned}
\end{equation}
Now let us consider the second term $\tilde{\phi}_j\big(-(-\Delta)^{s} -\partial_t\big)\eta_{j, R}$. From direct computations, we have
\begin{equation}\label{e4:46}
\begin{aligned}
\left|\tilde{\phi}_j\big(-(-\Delta)^{s} -\partial_t\big)\eta_{j, R}\right|\lesssim & \frac{\left|-(-\Delta)^{s}\eta\left(|\frac{x-\xi_j}{R\mu_{0j}}|\right)\right|}{R^{2s}\mu_{0j}^{2s}}\mu_0^{-\frac{n-2s}{2}}|\phi_j|\\
&+\left|\eta'\left(|\frac{x-\xi_j}{R\mu_{0j}}|\right)\left(\frac{|x-\xi_j|}{R\mu_0^2}\dot{\mu_0}+\frac{1}{R\mu_0}\dot{\xi}\right)\right|\mu_0^{-\frac{n-2s}{2}}|\phi_j|.
\end{aligned}
\end{equation}
For the first term in the right hand side of (\ref{e4:46}), by the definition of $\tilde{\phi}_j$, we obtain
\begin{equation}\label{e4:47}
\begin{aligned}
\frac{\left|-(-\Delta)^{s}\eta\left(|\frac{x-\xi_j}{R\mu_{0j}}|\right)\right|}{R^{2s}\mu_{0j}^{2s}}\mu_0^{-\frac{n-2s}{2}}|\phi_j|& \lesssim  \frac{\left|-(-\Delta)^{s}\eta\left(|\frac{x-\xi_j}{R\mu_{0j}}|\right)\right|}{R^{2s}\mu_{0j}^{2s}}\frac{\mu_0^{\frac{n-2s}{2}+\sigma}}{(1+|y_j|^a)}\|\phi\|_{n-2s+\sigma,a}\\
&\lesssim \frac{1}{R^{a-2s}}\|\phi\|_{n-2s+\sigma,a}\sum_{j=1}^k\frac{\mu_j^{-2s}\mu_0^{\frac{n-2s}{2}+\sigma}(t)}{1+|y_j|^{a}},
\end{aligned}
\end{equation}
where we have used the fact that $\left|-(-\Delta)^{s}\eta\left(|\frac{x-\xi_j}{R\mu_{0j}}|\right)\right|\sim \frac{1}{1+|\frac{y_j}{R}|^{2s}}$. From \eqref{e2:40} and \eqref{e4:21}, the second term in the right hand side of (\ref{e4:46}) can be estimated as
\begin{equation*}
\begin{aligned}
&\left|\eta'\left(|\frac{x-\xi_j}{R\mu_{0j}}|\right)\left(\frac{|x-\xi_j|\dot{\mu_0}+\mu_0\dot{\xi}}{R\mu_0^2}\right)\right|\mu_0^{-\frac{n-2s}{2}}|\phi_j|\\
\end{aligned}
\end{equation*}
\begin{equation}\label{e4:48}
\begin{aligned}
&\quad\quad\quad\quad\quad\lesssim \frac{\left|\eta'\left(|\frac{x-\xi_j}{R\mu_{0j}}|\right)\right|}{R^{2s}\mu_{0j}^{2s}}(\mu_0^{n-2s}R^{2s} +\mu_0^{n-2s+\sigma}R^{2s-1})\mu_0^{-\frac{n-2s}{2}}|\phi_j|\\
&\quad\quad\quad\quad\quad\lesssim \frac{1}{R^{a-2s}}\|\phi\|_{n-2s+\sigma,a}\sum_{j=1}^k\frac{\mu_j^{-2s}\mu_0^{\frac{n-2s}{2}+\sigma}(t)}{1+|y_j|^{a}}.
\end{aligned}
\end{equation}
From \eqref{4.46}-\eqref{e4:48}, we get (\ref{e4:38}).

{\it Proof of (\ref{e4:39})}. Since $p-2\geq 0$ gives $6s\geq n$, we have
\begin{equation}\label{e4:49}
\begin{aligned}
&\tilde{N}_{\mu, \xi}(\psi + \psi_1 + \sum_{j=1}^k\eta_{j, R}\tilde{\phi}_j)\lesssim\\
&\quad\quad\quad\quad\quad\left\{
\begin{aligned}
&(u^*_{\mu, \xi})^{p-2}\left[|\psi|^2 + |\psi_1|^2 + \sum_{j=1}^k|\eta_{j, R}\tilde{\phi}_j|^2\right], & \quad \mbox{when}~ 6s\geq n,\\
&|\psi|^p + |\psi_1|^p + \sum_{j=1}^k|\eta_{j, R}\tilde{\phi}_j|^p, & \quad \mbox{when}~ 6s< n.
\end{aligned}
\right.
\end{aligned}
\end{equation}
When $6s\geq n$, we have
\begin{equation*}\label{e4:51}
\begin{aligned}
\left|(u^*_{\mu, \xi})^{p-2}(\eta_{j, R}\tilde{\phi}_j)^2\right|&\lesssim \frac{\mu_0^{\frac{3n}{2}-5s+2\sigma}}{1+|y_j|^{2a}}\|\phi\|^2_{n-2s+\sigma, a}\\
&\lesssim \mu_0^{n-2s+\sigma}R^{a-2s}\|\phi\|^2_{n-2s+\sigma, a}\frac{1}{R^{a-2s}}\sum_{j=1}^k\frac{\mu_j^{-2s}\mu_0^{\frac{n-2s}{2}+\sigma}}{1+|y_j|^{a}}
\end{aligned}
\end{equation*}
and
\begin{equation*}\label{e4:52}
\begin{aligned}
\left|(u^*_{\mu, \xi})^{p-2}\psi^2\right|&\lesssim\mu_0^{-\frac{6s-n}{2}}\frac{t^{-2\beta}}{1+|y_j|^{2(a-2s)}}\|\psi\|^2_{**,\beta, a}\\
&\lesssim R^{a-2s}\mu_0^{n-4s+\sigma + a -2s}\|\psi\|^2_{**,\beta,a}\frac{1}{R^{a-2s}}\sum_{j=1}^k\frac{\mu_j^{-2s}t^{-\beta}}{1+|y_j|^{a}}.
\end{aligned}
\end{equation*}
When $6s < n$, we have
\begin{equation*}\label{e4:51}
\begin{aligned}
\left|\eta_{j, R}\tilde{\phi}_j\right|^p&\lesssim \frac{\mu_0^{(\frac{n-2s}{2}+\sigma)p}}{1+|y_j|^{ap}}\|\phi\|^p_{n-2s+\sigma, a}\\
&\lesssim \mu_0^{2s+(p-1)\sigma}R^{a-2s}\mu_0^{2s}\|\phi\|^p_{n-2s+\sigma, a}\frac{1}{R^{a-2s}}\sum_{j=1}^k\frac{\mu_j^{-2s}\mu_0^{\frac{n-2s}{2}+\sigma}}{1+|y_j|^{a}},
\end{aligned}
\end{equation*}
and
\begin{equation*}\label{e4:52}
\begin{aligned}
\left|\psi\right|^p&\lesssim\frac{t^{-p\beta}}{1+|y_j|^{p(a-2s)}}\|\psi\|^p_{**,\beta, a}\\
&\lesssim \mu^{4s(1+\frac{\sigma}{n-2s})+p(a-2s)-a} R^{a-2s}\|\psi\|^p_{**,\beta,a}\frac{1}{R^{a-2s}}\sum_{j=1}^k\frac{\mu_j^{-2s}\mu_0^{\frac{n-2s}{2}+\sigma}}{1+|y_j|^{a}}.
\end{aligned}
\end{equation*}
The estimates for $\psi_1$ are similar. Hence we have (\ref{e4:39}).

Now we apply the Contraction Mapping Theorem to prove the existence of a fixed point $\psi$ for $\mathcal{A}$.
First, set
\begin{equation*}\label{e4:54}
\mathcal{B}=\left\{\psi:\|\psi\|_{**, \beta, a}\leq M\frac{t_0^{-\varepsilon}}{R^{a-2s}}\right\}
\end{equation*}
with $\beta = \frac{n-2s}{2(n-4s)}+\frac{\sigma}{n-4s}$ and $a$, $\varepsilon$ are fixed as above. Here the positive large constant $M$ is independent of $t$ and $t_0$.
For any $\psi\in \mathcal{B}$, $\mathcal{A}(\psi)\in \mathcal{B}$ as a consequence of \eqref{e4:35} and the estimates \eqref{e4:37}-\eqref{e4:39}. We claim that for any $\psi_1$, $\psi_2\in\mathcal{B}$,
\begin{equation*}\label{e4:55}
\|\mathcal{A}(\psi^{(1)}) - \mathcal{A}(\psi^{(2)})\|_{**,\beta, a}\leq C\|\psi^{(1)}-\psi^{(2)}\|_{**,\beta, a},
\end{equation*}
where $C<1$ is a constant depending on $t_0$ which is chosen sufficiently large. Indeed,
\begin{equation*}\label{e4:56}
\mathcal{A}(\psi^{(1)}) - \mathcal{A}(\psi^{(2)})=T\left(\tilde{N}_{\mu, \xi}(\psi^{(1)}+\psi_1+\phi^{in}) - \tilde{N}_{\mu, \xi}(\psi^{(2)}+\psi_1+\phi^{in}),0,0\right),
\end{equation*}
where
\begin{equation*}\label{e4:57}
\begin{aligned}
&\tilde{N}_{\mu, \xi}(\psi^{(1)}+\psi_1+\phi^{in}) - \tilde{N}_{\mu, \xi}(\psi^{(2)}+\psi_1+\phi^{in}) = \\
&\left(u^*_{\mu,\xi}+\psi^{(1)}+\psi_1+\phi^{in}\right)^p-\left(u^*_{\mu,\xi}+\psi^{(1)}+\psi_1+\phi^{in}\right)^p-p(u^*_{\mu,\xi})^{p-1}
\left[\psi^{(1)}-\psi^{(2)}\right].
\end{aligned}
\end{equation*}
Similar to \eqref{e4:49}, we have
\begin{equation*}\label{e4:58}
\begin{aligned}
&\left|\tilde{N}_{\mu, \xi}(\psi^{(1)}+\psi_1+\phi^{in}) - \tilde{N}_{\mu, \xi}(\psi^{(2)}+\psi_1+\phi^{in})\right|\lesssim\\
&\quad\quad\quad\quad\quad\quad\quad\quad\quad\quad\quad\quad\quad\quad\quad\left\{
\begin{aligned}
& (u^*_{\mu,\xi})^{p-2}|\phi^{in}||\psi^{(1)}- \psi^{(2)}|, & \quad \mbox{when}~ 6s\geq n,\\
&|\phi^{in}|^{p-1}|\psi^{(1)}- \psi^{(2)}|, & \quad \mbox{when}~ 6s< n.
\end{aligned}
\right.
\end{aligned}
\end{equation*}
When $6s \geq n$,
\begin{equation*}\label{e4:60}
\begin{aligned}
&\left|\tilde{N}_{\mu, \xi}(\psi^{(1)}+\psi_1+\phi^{in}) - \tilde{N}_{\mu, \xi}(\psi^{(2)}+\psi_1+\phi^{in})\right|\\
&\lesssim \|\phi\|_{n-2s+\sigma, a}\|\psi^{(1)}-\psi^{(2)}\|_{**,\beta,a}R^{a-2s}\mu_0^{\frac{n}{2}+s+\sigma}(t_0)\frac{1}{R^{a-2s}}\sum_{j=1}^k\frac{\mu_j^{-2s}t^{-\beta}}{1+|y_j|^{a}},
\end{aligned}
\end{equation*}
while in the case of $6s < n$,
\begin{equation*}\label{e4:61}\begin{aligned}
&\left|\tilde{N}_{\mu, \xi}(\psi^{(1)}+\psi_1+\phi^{in}) - \tilde{N}_{\mu, \xi}(\psi^{(2)}+\psi_1+\phi^{in})\right|\\
&\lesssim \|\phi\|^{p-1}_{n-2s+\sigma, a}\|\psi^{(1)}-\psi^{(2)}\|_{**,\beta,a}R^{a-2s}\mu_0^{2s+\frac{4s(\sigma+a)}{n-2s}}(t_0)\frac{1}{R^{a-2s}}\sum_{j=1}^k\frac{\mu_j^{-2s}t^{-\beta}}{1+|y_j|^{a}}.
\end{aligned}
\end{equation*}
Hence there exists a choice of $R$ in the form (\ref{e3:20}) such that
\begin{equation*}\label{e4:62}
\|\mathcal{A}(\psi^{(1)}) - \mathcal{A}(\psi^{(2)})\|_{**,\beta, a}\leq C\|\psi^{(1)}-\psi^{(2)}\|_{**,\beta, a}
\end{equation*}
holds with $C < 1$, provided $t_0$ is sufficiently large. Therefore, if $t_0$ is fixed sufficiently large, $\mathcal{A}$ is a contraction map in $\mathcal{B}$. The validity of (\ref{e4:28}) follows directly from (\ref{e4:5}). The proof is completed.
\end{proof}

\subsection{Properties of the solution $\psi$.}
\begin{prop}\label{p4:4.2}
Under the assumptions in Proposition \ref{p4:4.1}, $\Psi$ depends smoothly on the parameters $\lambda$, $\xi$, $\dot{\lambda}$, $\dot{\xi}$, $\phi$, for $y_j=\frac{x-\xi_j}{\mu_{0j}}$, we have
\begin{equation}\label{e4:64}
\big|\partial_\lambda\Psi[\lambda,\xi,\dot{\lambda},\dot{\xi},\phi][\bar{\lambda}](x, t)\big|\lesssim \frac{t_0^{-\varepsilon}}{R^{a-2s}}\|\bar{\lambda}(t)\|_{1+\sigma}\left(\sum_{j=1}^k\frac{\mu_0^{\frac{n-2s}{2}-1}(t)}{1+|y_j|^{a-2s}}\right),
\end{equation}
\begin{equation}\label{e4:74}
\big|\partial_\xi\Psi[\lambda,\xi,\dot{\lambda},\dot{\xi},\phi][\bar{\xi}](x, t)\big|\lesssim \frac{t_0^{-\varepsilon}}{R^{a-2s}}\|\bar{\xi}(t)\|_{1+\sigma}\left(\sum_{j=1}^k\frac{\mu_0^{\frac{n-2s}{2}-1}(t)}{1+|y_j|^{a-2s}}\right),
\end{equation}
\begin{equation}\label{e4:75}
\big|\partial_{\dot{\xi}}\Psi[\lambda,\xi,\dot{\lambda},\dot{\xi},\phi][\dot{\bar{\xi}}](x, t)\big|\lesssim \frac{t_0^{-\varepsilon}}{R^{a-2s}}\|\dot{\bar{\xi}}(t)\|_{n-4s+1+\sigma}\left(\sum_{j=1}^k\frac{\mu_0^{-\frac{n-6s}{2}-1+\sigma}(t)}{1+|y_j|^{a-2s}}\right),
\end{equation}
\begin{equation}\label{e4:76}
\big|\partial_{\dot{\lambda}}\Psi[\lambda,\xi,\dot{\lambda},\dot{\xi},\phi][\dot{\bar{\lambda}}](x, t)\big|\lesssim \frac{t_0^{-\varepsilon}}{R^{a-2s}}\|\dot{\bar{\lambda}}(t)\|_{n-4s+1+\sigma}\left(\sum_{j=1}^k\frac{\mu_0^{-\frac{n-6s}{2}-1+\sigma}(t)}{1+|y_j|^{a-2s}}\right),
\end{equation}
\begin{equation}\label{e4:84}
\big|\partial_{\phi}\Psi[\lambda,\xi,\dot{\lambda},\dot{\xi},\phi][\bar{\phi}](x, t)\big|\lesssim \frac{1}{R^{a-2s}}\|\bar{\phi}(t)\|_{n-2s+\sigma, a}\left(\sum_{j=1}^k\frac{\mu_0^{\frac{n-2s}{2}+\sigma}(t)}{1+|y_j|^{a-2s}}\right).
\end{equation}
\end{prop}
\begin{proof}

{\bf Step 1}. Proof of (\ref{e4:64}) and (\ref{e4:74}).

We fix $j=1$. $\Psi[\lambda_1]$ is a solution to problem (\ref{e4:main})  for all $\lambda_1$ satisfying (\ref{e4:22}). Differentiating problem (\ref{e4:main}) with respect to $\lambda_1$ gives us a nonlinear equation. From the Implicit Function Theorem, the solutions are given by $\partial_{\lambda_1}\Psi[\bar{\lambda}_1](x, t)$. Decompose $\partial_{\lambda_1}\Psi[\bar{\lambda}_1](x, t) = Z_1 + Z$ with $Z_1 = T(0, -(\partial_{\lambda_1}u^*_{\mu,\xi})[\bar{\lambda}_1],0)$, where $T$ is defined in Lemma \ref{l4:lemma4.1}. Then $Z$ is a solution of the following nonlinear problem
\begin{equation}\label{e4:72}
\left\{
\begin{aligned}
&\partial_tZ =
-(-\Delta)^{s}Z + V_{\mu, \xi}Z + \left(\partial_{\lambda_1}V_{\mu,\xi}\right)[\bar{\lambda}_1]\psi + \partial_{\lambda_1}\left[\tilde{N}_{\mu,\xi}\left(\psi + \phi^{in}\right)\right][\bar{\lambda}_1] \\
&\quad\quad\quad + \partial_{\lambda_1}S_{out}[\bar{\lambda}_1]~~\text{ in }\Omega\times (t_0, \infty),\\
&Z= 0~~\text{ in }(\mathbb{R}^n\setminus\Omega)\times (t_0,\infty),\\
&Z(\cdot,t_0) = 0~~\text{ in }\mathbb{R}^n.
\end{aligned}
\right.
\end{equation}
By definition, $\sum_{j=1}^k\left\{\big[-(-\Delta)^{\frac{s}{2}}\eta_{j, R},-(-\Delta)^{\frac{s}{2}}\tilde{\phi}_j\big] + \tilde{\phi}_j(-(-\Delta)^{s} -\partial_t)\eta_{j, R}\right\}$ is independent of $\lambda_1$. Then for any $x\in \mathbb{R}^n\setminus\Omega$,
\begin{equation}\label{bdr}
\left|\partial_{\lambda_1}u^*_{\mu, \xi}(x, t)\right|\lesssim \mu_0^{\frac{n-2s}{2}-1}(t)|\bar{\lambda}_1(t)|.
\end{equation}
From \eqref{bdr} and Lemma \ref{l4:lemma4.1}, we obtain
\begin{equation*}
|Z_1(x, t)|\lesssim \frac{t_0^{-\varepsilon}}{R^{a-2s}}\|\bar{\lambda}_1\|_{1+\sigma}\left(\sum_{j=1}^k\frac{\mu_0^{\frac{n-2s}{2}-1}(t)}{1+|y_j|^{a-2s}}\right).
\end{equation*}
For problem (\ref{e4:72}), we compute
\begin{equation*}
\begin{aligned}
\partial_{\lambda_1}\left[\tilde{N}_{\mu,\xi}\left(\psi + \phi^{in}\right)\right][\bar{\lambda}_1]=&p\left[(u^*_{\mu,\xi}+\psi+\phi^{in})^{p-1} -(u^*_{\mu,\xi})^{p-1}\right](Z+Z_1)\\
&+p(p-1)(u^*_{\mu,\xi})^{p-2}(\psi+\phi^{in})\partial_{\lambda_1}u^*_{\mu,\xi}[\bar{\lambda}_1].
\end{aligned}
\end{equation*}
Therefore, $Z$ is a fixed point of the operator
\begin{equation}\label{e4:73}
\mathcal{A}_1(Z) = T(f + p\left[(u^*_{\mu,\xi}+\psi+\phi^{in})^{p-1} -(u^*_{\mu,\xi})^{p-1}\right]Z, 0, 0),
\end{equation}
where
\begin{equation}\label{fff}
\begin{aligned}
f =&\partial_{\lambda_1}S_{out}[\bar{\lambda}_1]+\left(\partial_{\lambda_1}V_{\mu,\xi}\right)[\bar{\lambda}_1]\psi + p\left[(u^*_{\mu,\xi}+\psi+\phi^{in})^{p-1} -(u^*_{\mu,\xi})^{p-1}\right]Z_1\\
&+ p(p-1)(u^*_{\mu,\xi})^{p-2}(\psi+\phi^{in})\partial_{\lambda_1}u^*_{\mu,\xi}[\bar{\lambda}_1].
\end{aligned}
\end{equation}
We claim that
\begin{equation}\label{e4:71}
\left|f(x, t)\right|\lesssim \frac{t_0^{-\varepsilon}}{R^{a-2s}}\|\bar{\lambda}_1\|_{1+\sigma}\sum_{j=1}^k\frac{\mu_j^{-2s}\mu_0^{\frac{n-2s}{2}-1+\sigma}}{1+|y_j|^{a}}.
\end{equation}
To prove \eqref{e4:71}, we first estimate $\partial_{\lambda_1}S_{out}[\bar{\lambda}_1]$. In the region $|x-q_i| > \delta$ ($i=1,\cdots,k$), we have the following estimate for $\partial_{\lambda_1}S(u^*_{\mu,\xi})$ by \eqref{awayq}, (\ref{e4:21}) and (\ref{e4:22})
\begin{equation*}
\partial_{\lambda_1}S(u^*_{\mu,\xi})[\bar{\lambda}_1](x, t) = \mu_0^{\frac{n-2s}{2}-1}f(x,\mu_0^{-1}\mu, \xi)\bar{\lambda}_1(t),
\end{equation*}
where the smooth and bounded function $f$ depends on $(x,\mu_0^{-1}\mu,\xi)$. Now we fix $j$ and consider the region $|x-q_j|\leq \delta$. From (\ref{e2:55}), we have
\begin{equation*}
\partial_{\lambda_1}S(u^*_{\mu,\xi})[\bar{\lambda}_1](x, t) = \partial_{\lambda_1}S(u_{\mu,\xi})[\bar{\lambda}_1](x, t)(1 + \mu_0f(x, \mu_0^{-1}\mu, \xi,t)),
\end{equation*}
where the smooth and bounded function $f$ depends on $(x,\mu_0^{-1}\mu,\xi,t)$. Differentiating  (\ref{e2:3333}) with respect to $\lambda_1$, we obtain
\begin{equation*}\label{e4:65}
\begin{aligned}
\partial_{\lambda_1}S(u_{\mu,\xi})[\bar{\lambda}_1](x, t) =& -(\frac{n-2s}{2}+1)\mu_1^{-\frac{n-2s}{2}-2}\bigg[\dot{\mu}_1Z_{n+1}(y_1)+\dot{\xi}_1\cdot \nabla U(y_1)\\
& -\frac{1}{\frac{n-2s}{2}+1} \frac{n-2s}{2}\left(\frac{n-2s}{2}-1\right)\mu_1^{n-2s}\dot{\mu}_1H(x, q_i)\bigg]\bar{\lambda}_1(t)\\
&-\mu_i^{-\frac{n-2s}{2}-1}\left[\dot{\xi}_1D^2 U(y_1)+\dot{\mu}_1\nabla Z_{n+1}(y_1)\right]\cdot \frac{x-\xi_1}{\mu_1^2}\bar{\lambda}_1(t)\\
&+ p\left(\sum_{i=1}^k\mu_i^{-\frac{n-2s}{2}}U(y_i) - \mu_i^{\frac{n-2s}{2}}H(x, q_i)\right)^{p-1}\\
&\quad\times\partial_{\lambda_1}\left[\mu_1^{-\frac{n-2s}{2}}U(y_1) - \mu_1^{\frac{n-2s}{2}}H(x, q_1)\right]\bar{\lambda}_1(t)\\
& - p\left(\mu_1^{-\frac{n-2s}{2}}U(y_1)\right)^{p-1}\partial_{\lambda_1}[\mu_1^{-\frac{n-2s}{2}}U(y_1)]\bar{\lambda}_1(t).
\end{aligned}
\end{equation*}
From (\ref{e4:21}) and (\ref{e4:22}), we have
\begin{eqnarray}\label{e4:66}
\left|\partial_{\lambda_1}S(u_{\mu,\xi})[\bar{\lambda}_1](x, t)\right| \lesssim\frac{t_0^{-\varepsilon}}{R^{a-2s}}\|\bar{\lambda}_1\|_{1+\sigma}\sum_{j=1}^k\frac{\mu_j^{-2s}\mu_0^{\frac{n-2s}{2}-1}}{1+|y_j|^{a}}.
\end{eqnarray}
Therefore, by the definition of $S_{out}$ together with \eqref{e4:66}, we obtain
\begin{eqnarray*}\label{e4:67}
\left|\partial_{\lambda_1}S_{out}[\bar{\lambda}_1](x, t)\right| \lesssim\frac{t_0^{-\varepsilon}}{R^{a-2s}}\|\bar{\lambda}_1\|_{1+\sigma}\sum_{j=1}^k\frac{\mu_j^{-2s}\mu_0^{\frac{n-2s}{2}-1}}{1+|y_j|^{a}}.
\end{eqnarray*}
Next, we estimate the remainders in $f$. Direct computations imply that
\begin{equation*}\label{e4:68}
\begin{aligned}
(\partial_{\lambda_1}V_{\mu,\xi})[\bar{\lambda}_1](x, t) =& p(p-1)\bigg[(u^*_{\mu, \xi})^{p-2}\partial_{\lambda_1}u^*_{\mu,\xi}[\bar{\lambda}_1]\\
&-\eta_{1, R}(\mu_1^{-\frac{n-2s}{2}}U(y_1))^{p-2}\partial_{\lambda_1}\big(\mu_1^{-\frac{n-2s}{2}}U(y_1)\big)[\bar{\lambda}_1]\bigg].
\end{aligned}
\end{equation*}
Since $\left|\partial_{\lambda_1}\left(\mu_1^{-\frac{n-2s}{2}}U\left(y_1\right)\right)\right|\lesssim\mu_0^{-1}\left|\mu_1^{-\frac{n-2s}{2}}U\left(y_1\right)\right|$ and $\beta = \frac{n-2s}{2(n-4s)} + \frac{\sigma}{n-4s}$, we have
\begin{eqnarray*}\label{e4:69}
\left|\left(\partial_{\lambda_1}V_{\mu,\xi}\right)[\bar{\lambda}_1]\psi(x, t)\right| \lesssim \|\psi\|_{**,\beta,a} \frac{t_0^{-\varepsilon}}{R^{a-2s}}\|\bar{\lambda}_1\|_{1+\sigma}\sum_{j=1}^k\frac{\mu_j^{-2s}\mu_0^{\frac{n-2s}{2}-1+\sigma}}{1+|y_j|^{a}}.
\end{eqnarray*}
By the same token, we can deal with $p(p-1)(u^*_{\mu,\xi})^{p-2}(\psi+\phi^{in})\partial_{\lambda_1}u^*_{\mu,\xi}[\bar{\lambda}_1]$ in \eqref{fff} and obtain
\begin{equation*}\label{e4:70}
 \left|p(p-1)(u^*_{\mu,\xi})^{p-2}(\psi+\phi^{in})\partial_{\lambda_1}u^*_{\mu,\xi}[\bar{\lambda}_1]\right| \lesssim \frac{t_0^{-\varepsilon}}{R^{a-2s}}\|\bar{\lambda}_1\|_{1+\sigma}\sum_{j=1}^k\frac{\mu_j^{-2s}\mu_0^{\frac{n-2s}{2}-1+\sigma}}{1+|y_j|^{a}}.
\end{equation*}
Analogously, we can estimate the last term $p\left[(u^*_{\mu,\xi}+\psi+\phi^{in})^{p-1} -(u^*_{\mu,\xi})^{p-1}\right]Z_1$. Therefore, we conclude the validity of (\ref{e4:71}).

Now we consider the fixed point problem (\ref{e4:73}). Then the operator $\mathcal{A}_1$ has a fixed point in the set of functions satisfying
$$|Z(x, t)|\leq M \frac{t_0^{-\varepsilon}}{R^{a-2s}}\|\bar{\lambda}_1\|_{1+\sigma}\sum_{j=1}^k\frac{\mu_0^{\frac{n-2s}{2}-1}}{1+|y_j|^{a-2s}}$$
with the large constant $M$ fixed. In fact, $\mathcal{A}_1$ is a contraction map when $R$ is chosen properly large in terms of $t_0$. Therefore, the estimate (\ref{e4:64}) for $\partial_{\lambda_1}\Psi[\bar{\lambda}_1]$ holds. The estimate (\ref{e4:74}) for $\partial_\xi\Psi[\bar{\xi}]$ can be verified in a similar way. Here we omit the details.

{\bf Step 2}. Proof of (\ref{e4:75}) and (\ref{e4:76}).

We fix $j=1$. From the discussions above, the function $\Psi[\dot{\lambda}_1]$ is a solution to  (\ref{e4:main})  for all $\dot{\lambda}_1$ satisfying (\ref{e4:22}). Then we differentiate problem (\ref{e4:main}) with respect to $\dot{\lambda}_1$ and obtain a nonlinear equation. From the Implicit Function Theorem, the solutions are given by $\partial_{\dot{\lambda}_1}\Psi[\dot{\bar{\lambda}}_1](x, t)$. Denote $Z(x, t) = \partial_{\dot{\lambda}_1}\Psi[\dot{\bar{\lambda}}_1](x, t)$. Then $Z$ is a solution to the following nonlinear problem
\begin{equation*}\label{e4:77}
\left\{
\begin{aligned}
&\partial_tZ = \\
&\quad -(-\Delta)^{s}Z + V_{\mu, \xi}Z + \partial_{\dot{\lambda}_1}\left[\tilde{N}_{\mu,\xi}\left(\psi + \phi^{in}\right)\right][\dot{\bar{\lambda}}_1] + \partial_{\dot{\lambda}_1}S_{out}[\dot{\bar{\lambda}}_1]~~\text{ in }\Omega\times (t_0, \infty),\\
&Z(x,t) = 0~~\text{ in }(\mathbb{R}^n\setminus\Omega)\times (t_0,\infty),\\
&Z(\cdot,t_0) = 0~~\text{ in }\mathbb{R}^n.
\end{aligned}
\right.
\end{equation*}
From the definition of $\tilde{N}_{\mu,\xi}\left(\psi + \phi^{in}\right)$, we have
\begin{equation*}
\partial_{\dot{\lambda}_1}\left[\tilde{N}_{\mu,\xi}\left(\psi + \phi^{in}\right)\right][\dot{\bar{\lambda}}_1]=p\left[(u^*_{\mu,\xi}+\psi+\phi^{in})^{p-1} -(u^*_{\mu,\xi})^{p-1}\right]Z(x, t).
\end{equation*}
Therefore, $Z$ is a fixed point for the operator
\begin{equation}\label{e4:78}
\mathcal{A}_1(Z) = T\left(\partial_{\dot{\lambda}_1}S_{out}[\dot{\bar{\lambda}}_1] + p\left[(u^*_{\mu,\xi}+\psi+\phi^{in})^{p-1} -(u^*_{\mu,\xi})^{p-1}\right]Z, 0, 0\right).
\end{equation}
Now we differentiate $S(u^*_{\mu,\xi})$ with respect to $\dot{\bar{\lambda}}_1$ in \eqref{sstar} directly and obtain
\begin{equation*}\label{e4:79}
\begin{aligned}
\partial_{\dot{\lambda}_1}S(u^*_{\mu,\xi})[\dot{\bar{\lambda}}_1](x, t)=&\mu_1^{-\frac{n-2s}{2}-1}\left[Z_{n+1}(y_1) + \frac{n-2s}{2}\mu_1^{n-2s}H(x, q_1)\right]\dot{\bar{\lambda}}_1(t)\\
&+\mu_j^{-\frac{n}{2}+s-1}\left[\frac{n-2s}{2}\Phi_{1}(y_1, t)+y_1\cdot\nabla_y\Phi_1\right]\dot{\bar{\lambda}}_1(t).
\end{aligned}
\end{equation*}
Hence
\begin{eqnarray*}\label{e4:80}
\left|\partial_{\dot{\lambda}_1}S(u^*_{\mu,\xi})[\dot{\bar{\lambda}}_1](x, t)\right|\lesssim \frac{t_0^{-\varepsilon}}{R^{a-2s}}\|\dot{\bar{\lambda}}_1(t)\|_{n-4s+1+\sigma}
\left(\sum_{j=1}^k\frac{\mu_j^{-2s}(t)\mu_{0}^{-\frac{n-6s}{2}-1}}{1+|y_j|^{a}}\right).
\end{eqnarray*}

Now we consider the fixed point problem (\ref{e4:78}). Similar to Step 1, $\mathcal{A}_1$ has a fixed point in the set of functions satisfying
\begin{eqnarray*}\label{e4:81}
\left|Z(x, t)\right|\lesssim \frac{t_0^{-\varepsilon}}{R^{a-2s}}\|\dot{\bar{\lambda}}_1(t)\|_{n-4s+1+\sigma}\sum_{j=1}^k\frac{\mu_{0}^{-\frac{n-6s}{2}-1}}{1+|y_j|^{a-2s}}.
\end{eqnarray*}
Thus estimate (\ref{e4:75}) holds.

On the other hand, observe that
\begin{eqnarray}\label{e4:82}
&& \partial_{\dot{\xi}_1}S(u^*_{\mu,\xi})[\dot{\bar{\xi}}_1](x, t)= \mu_1^{-\frac{n-2s}{2}-1}\left[\nabla U(y_1) + \nabla\Phi_1(y_1, t)\right]\dot{\bar{\xi}}_1(t).
\end{eqnarray}
From \eqref{e4:82} we have
\begin{eqnarray*}\label{e4:83}
\left|\partial_{\dot{\xi}_1}S(u^*_{\mu,\xi})[\dot{\bar{\lambda}}_1](x, t)\right|\lesssim \frac{t_0^{-\varepsilon}}{R^{a-2s}}\|\dot{\bar{\xi}}_1(t)\|_{n-4s+1+\sigma}
\left(\sum_{j=1}^k\frac{\mu_j^{-2s}(t)\mu_{0}^{-\frac{n-6s}{2}-1}}{1+|y_j|^{a}}\right).
\end{eqnarray*}
Therefore, we have (\ref{e4:76}).

{\bf Step 3}. Proof of (\ref{e4:84}).

Define $Z(x, t) = \partial_\phi\psi[\bar{\phi}](x, t)$ with $\bar{\phi}$ satisfying (\ref{e4:25}). Therefore, $Z$ is a solution to
\begin{equation*}\label{e4:85}
\left\{
\begin{aligned}
&\partial_tZ =
-(-\Delta)^{s}Z + V_{\mu, \xi}Z\\
&\quad\quad\quad + \sum_{j=1}^k\left\{\big[-(-\Delta)^{\frac{s}{2}}\eta_{j, R},-(-\Delta)^{\frac{s}{2}}\hat{\phi}_j\big] + \hat{\phi}_j\big(-(-\Delta)^{s} -\partial_t\big)\eta_{j, R}\right\}\\
&\quad\quad\quad + p\left[(u^*_{\mu,\xi}+\psi+\phi^{in})^{p-1} -(u^*_{\mu,\xi})^{p-1}\right]\bar{\phi}~~\text{ in }\Omega\times (t_0, \infty),\\
&Z = 0~~\text{ in }(\mathbb{R}^n\setminus\Omega)\times (t_0,\infty),\\
&Z(\cdot,t_0) = 0~~\text{ in }\mathbb{R}^n,
\end{aligned}
\right.
\end{equation*}
where $\hat{\phi} = \mu_0^{-\frac{n-2s}{2}}\bar{\phi}_j\left(\frac{x-\xi_j}{\mu_{0j}}, t\right)$.

As in Step 1 and Step 2, we have
\begin{equation*}\label{e4:86}
\begin{aligned}
&\left|\sum_{j=1}^k\left\{\big[-(-\Delta)^{\frac{s}{2}}\eta_{j, R},-(-\Delta)^{\frac{s}{2}}\hat{\phi}_j\big]+ \hat{\phi}_j\big(-(-\Delta)^{s} -\partial_t\big)\eta_{j, R}\right\}\right|\\
&\quad\lesssim \frac{1}{R^{a-2s}}\|\bar{\phi}\|_{n-2s+\sigma, a}\sum_{j=1}^k\frac{\mu_j^{-2s}\mu_0^{\frac{n-2s}{2}+\sigma}}{1+|y_j|^{a}}
\end{aligned}
\end{equation*}
and
\begin{equation*}\label{e4:87}
\begin{aligned}
&\bigg|p\left[(u^*_{\mu,\xi}+\psi+\phi^{in})^{p-1} -(u^*_{\mu,\xi})^{p-1}\right]\bar{\phi}\bigg|\\
&\quad\lesssim \frac{1}{R^{a-2s}}\|\bar{\phi}\|_{n-2s+\sigma, a}\left[\|\psi\|^{p-1}_{**,\beta,a} + \|\phi^{in}\|^{p-1}_{n-2s+\sigma,a}\right]\sum_{j=1}^k\frac{\mu_j^{-2s}\mu_0^{\frac{n-2s}{2}+\sigma}}{1+|y_j|^{a}}.
\end{aligned}
\end{equation*}
From Lemma \ref{l4:lemma4.1}, we conclude the validity of (\ref{e4:84}).
\end{proof}

\section{The inner problem}
Substituting the solution $\psi = \Psi[\lambda,\xi,\dot{\lambda},\dot{\xi},\phi]$ of the outer problem given by proposition \ref{p4:4.1} into the inner problem (\ref{e3:10}), the full problem is reduced to the following system
\begin{eqnarray}\label{e5:1}
&& \mu_{0j}^{2s}\partial_t\phi_j = -(-\Delta)^s_y\phi_j + pU^{p-1}(y)\phi_j + H_j[\lambda,\xi,\dot{\lambda},\dot{\xi},\phi](y,t),~ y\in \mathbb{R}^n,~ t\geq t_0
\end{eqnarray}
for $j = 1,\cdots,k$, where
\begin{equation}\label{e5:2}
\begin{aligned}
H_j[\lambda,\xi,\dot{\lambda},\dot{\xi},\phi]:=&\Bigg\{\mu_{0j}^{\frac{n+2s}{2}}S_{\mu, \xi, j}(\xi_j + \mu_{0j}y, t)+ B_j[\phi_j] + B_j^0[\phi_j]\\
&+ p\mu_{0j}^{\frac{n-2s}{2}}\frac{\mu_{0j}^{2s}}{\mu_j^{2s}}U^{p-1}\left(\frac{\mu_{0j}}{\mu_j}y\right)\psi(\xi_j + \mu_{0j}y, t)\Bigg\}\chi_{B_{2R}(0)}(y)
\end{aligned}
\end{equation}
and $B_j[\phi_j]$ and $B_j^0[\phi_j]$ are defined in (\ref{e3:11}), (\ref{e3:12}) respectively.

After the change of variables
\begin{equation*}\label{e5:3}
t = t(\tau),\quad \frac{dt}{d\tau} = \mu_{0j}^{2s}(t),
\end{equation*}
(\ref{e5:1}) is reduced to
\begin{eqnarray}\label{e5:4}
\partial_\tau\phi_j = -(-\Delta)^s_y\phi_j + pU^{p-1}(y)\phi_j + H_j[\lambda,\xi,\dot{\lambda},\dot{\xi},\phi](y,t(\tau)),~ y\in \mathbb{R}^n,~ \tau\geq \tau_0
\end{eqnarray}
with $\tau_0$ the unique positive number satisfying $t(\tau_0) = t_0$.

We will find a solution $\phi = (\phi_1,\cdots,\phi_k)$ to the system
\begin{equation}\label{e5:6}
\left\{
\begin{aligned}
&\partial_\tau\phi_j = -(-\Delta)^s_y\phi_j + pU^{p-1}(y)\phi_j + H_j[\lambda,\xi,\dot{\lambda},\dot{\xi},\phi](y,t(\tau)),~ y\in \mathbb{R}^n,~ \tau\geq \tau_0,\\
&\phi_j(y,\tau_0) = e_{0j}Z_0(y),~ y\in \mathbb{R}^n,
\end{aligned}
\right.
\end{equation}
for a constant $e_{0j}$ and all $j = 1,\cdots, k$.
Here $Z_0$ is a radially symmetric eigenfunction associated to the unique negative eigenvalue $\lambda_0$ of the eigenvalue problem
\begin{equation*}\label{e3:13}
L_0(\phi) + \lambda\phi = 0,\quad \phi\in L^\infty(\mathbb{R}^n).
\end{equation*}
Note that $\lambda_0$ is simple and $Z_0$ satisfies
\begin{equation*}\label{e3:14}
Z_0(y) \sim |y|^{-n-2s}\text{ as } |y|\to \infty,
\end{equation*}
see, for example, \cite{FrankLenzmannSilvestre2016}.
We will prove that (\ref{e5:6}) is solvable in the function space of those $\phi_j$'s satisfying (\ref{e4:25}), provided $\xi$ and $\lambda$ are chosen so that  $H_j[\lambda,\xi,\dot{\lambda},\dot{\xi},\phi](y,t(\tau))$ satisfy the orthogonality conditions
\begin{equation}\label{e5:7}
\int_{B_{2R}}H_j[\lambda,\xi,\dot{\lambda},\dot{\xi},\phi](y,t(\tau))Z_l(y)dy = 0,
\end{equation}
for all $\tau\geq \tau_0$, $j = 1,\cdots,k ~\mbox{and}~l = 1,2,\cdots,n+1.$
We first develop a linear theory which is the context of the subsection 5.1.
\subsection{The linear theory.}
In this subsection, for $R > 0$ fixed large, we find a solution to the nonlocal initial value problem
\begin{equation}\label{e5:31}
\left\{
\begin{aligned}
&\partial_\tau\phi = -(-\Delta)^s\phi + pU^{p-1}(y)\phi + h(y,\tau),~ y\in \mathbb{R}^n,~ \tau\geq \tau_0,\\
&\phi(y,\tau_0) = e_{0}Z_0(y),~ y\in \mathbb{R}^n.
\end{aligned}
\right.
\end{equation}
Let
\begin{equation*}
\nu = 1 + \frac{\sigma}{n-2s},
\end{equation*}
then $\mu_0^{n-2s+\sigma}\sim \tau^{-\nu}$.
Define
\begin{equation*}
\|h\|_{a, \nu, \eta}: = \sup_{\tau > \tau_0}\sup_{y\in B_{2R}}\tau^\nu(1 + |y|^a)(|h(y,\tau)|+(1+|y|^\eta)\chi_{B_{2R}(0)}(y)[h(\cdot, \tau)]_{\eta, B_{1}(0)}).
\end{equation*}
In the following, we always assume that $h = h(y, \tau)$ is a function defined in the whole space $\mathbb{R}^n$ which is zero outside $B_{2R}(0)$ for all $\tau > \tau_0$.
The main result in this subsection is the following.
\begin{prop}\label{proposition5.1}
Suppose $a\in (2s, n-2s)$, $\nu > 0$, $\|h\|_{2s+a,\nu, \eta} < +\infty$ and
\begin{equation*}
\int_{B_{2R}}h(y,\tau)Z_j(y)dy = 0~\text{ for all }~\tau\in (\tau_0,\infty), ~j = 1,\cdots, n+1.
\end{equation*}
For sufficiently large $R$, there exist $\phi = \phi[h](y, \tau)$ and $e_0 = e_0[h](\tau)$ ($\tau\in (\tau_0,+\infty),y\in \mathbb{R}^n$) satisfying (\ref{e5:31}) and
\begin{equation}\label{e5:100}
\begin{aligned}
&(1+|y|)|\nabla_y \phi(y, \tau)|\chi_{B_{2R}(0)}(y)+ |\phi(y,\tau)|\\
&\quad\quad\quad\quad\quad\quad\quad\quad\quad\quad\lesssim \tau^{-\nu}(1+|y|)^{-a}\|h\|_{2s+a,\nu, \eta}, \tau\in (\tau_0,+\infty),y\in \mathbb{R}^n,
\end{aligned}
\end{equation}
\begin{equation}\label{e5:101}
|e_0[h]|\lesssim \|h\|_{2s+a,\nu, \eta}.
\end{equation}
\end{prop}
\begin{lemma}\label{l5:3}
Suppose $a\in (2s, n-2s)$, $\nu > 0$, $\|h\|_{2s+a,\nu, \eta} < +\infty$ and
\begin{equation*}
\int_{\mathbb{R}^n}h(y,\tau)Z_j(y)dy = 0~\text{ for all }~\tau\in (\tau_0,\infty), ~j = 1,\cdots, n+1.
\end{equation*}
For any sufficiently large $\tau_1 > 0$, the solution $(\phi(y,\tau), c(\tau))$ of the problem
\begin{eqnarray}\label{e5:32}
\left\{
\begin{aligned}
&\partial_\tau\phi = -(-\Delta)^s\phi + pU^{p-1}(y)\phi + h(y,\tau)-c(\tau)Z_0(y),~ y\in \mathbb{R}^n,~\tau\geq \tau_0,\\
&\int_{\mathbb{R}^n}\phi(y,\tau)Z_0(y)dy = 0~\mbox{ for all }~ \tau\in (\tau_0,+\infty),\\
&\phi(y,\tau_0) = 0,~y\in \mathbb{R}^n,
\end{aligned}
\right.
\end{eqnarray}
satisfies the estimates
\begin{equation}\label{e5:35}
\|\phi(y,\tau)\|_{a,\tau_1}\lesssim \|h\|_{2s+a,\tau_1}
\end{equation}
and
\begin{equation*}
|c(\tau)|\lesssim \tau^{-\nu}R^a\|h\|_{2s+a,\tau_1}~ \text{ for }~\tau\in (\tau_0,\tau_1).
\end{equation*}
Here $\|h\|_{b,\tau_1}:=\sup_{\tau\in (\tau_0,\tau_1)}\tau^\nu\|(1+|y|^b)h\|_{L^\infty(\mathbb{R}^n)}$.
\end{lemma}
\begin{proof}
Note that (\ref{e5:32}) is equivalent to
\begin{eqnarray}\label{e5:99}
\left\{
\begin{aligned}
&\partial_\tau\phi = -(-\Delta)^s\phi + pU^{p-1}(y)\phi + h(y,\tau)-c(\tau)Z_0(y),~ y\in \mathbb{R}^n,~ \tau\geq \tau_0,\\
&\phi(y,\tau_0) = 0,~ y\in \mathbb{R}^n
\end{aligned}
\right.
\end{eqnarray}
for $c(\tau)$ given by the relation
\begin{equation*}
c(\tau) \int_{\mathbb{R}^n}|Z_0(y)|^2dy =  \int_{\mathbb{R}^n}h(y,\tau)Z_0(y)dy.
\end{equation*}
It is easy to see that
\begin{equation}\label{e5:102}
|c(\tau)|\lesssim \tau^{-\nu}R^a\|h\|_{2s+a,\tau_1}
\end{equation}
holds for $\tau\in (\tau_0,\tau_1)$. So we only need to prove (\ref{e5:35}) for the solution $\phi$ of (\ref{e5:99}). Inspired by Lemma 4.5 of \cite{davila2017singularity} and the linear theory of \cite{sireweizhenghalf}, we will use the blow-up argument.

First, we claim that, given $\tau_1 > \tau_0$, we have $\|\phi\|_{a, \tau_1} < +\infty$. Indeed, by the fractional parabolic theory (see \cite{jinxiongjram2014}), given $R_0 > 0$ there is a $K = K(R_0,\tau_1)$ such that
\begin{equation*}
|\phi(y,\tau)|\leq K \quad\text{in }B_{R_0}(0)\times (\tau_0, \tau_1].
\end{equation*}
Fix $R_0$ large and take $K_1$ sufficiently large, $K_1\rho^{-a}$ is a supersolution for (\ref{e5:99}) when $\rho > R_0$. Hence $|\phi|\leq 2K_1\rho^{-a}$ and $\|\phi\|_{a,\tau_1} < +\infty$ for any $\tau_1 > 0$. Next, we claim that the following identities hold,
\begin{equation}\label{e5:33}
\int_{\mathbb{R}^n}\phi(y, \tau)\cdot Z_j(y)dy = 0\text{ for all }\tau\in (\tau_0,\tau_1), j= 0,1,\cdots, n+1.
\end{equation}
Indeed, From the definition of $c(\tau)$, we have
\begin{equation*}
\int_{\mathbb{R}^n}\phi(y, \tau)\cdot Z_0(y)dy = 0.
\end{equation*}
Testing (\ref{e5:99}) with $Z_j\eta,$ where $\eta(y) = \eta_0(|y|/R_1),\,\, j = 1,\cdots, n+1,$  $R_1$ is an arbitrary positive constant and the smooth cut-off function $\eta_0$ is defined as
\begin{equation*}
\eta_0(r) =
\left\{
\begin{aligned}
1,~ \mbox{for}~ r < 1,\\
0,~ \mbox{for}~ r > 2,
\end{aligned}
\right.
\end{equation*}
we get
\begin{equation*}
\int_{\mathbb{R}^n}\phi(\cdot, \tau)\cdot Z_j\eta = \int_{0}^\tau ds\int_{\mathbb{R}^n}(\phi(\cdot, s)\cdot L_0[\eta Z_j] + h Z_j\eta - c(s)Z_0Z_j\eta).
\end{equation*}
Furthermore, it holds that
\begin{equation*}
\begin{aligned}
&\int_{\mathbb{R}^n}\bigg(\phi\cdot L_0[\eta Z_j] + h Z_j\eta - c(s)Z_0Z_j\eta\bigg)\\
&= \int_{\mathbb{R}^n}\phi\cdot \bigg(Z_j(-(-\Delta)^s)\eta + \big[-(-\Delta)^{\frac{s}{2}}\eta, -(-\Delta)^{\frac{s}{2}}Z_j\big]\bigg)\\
&\quad - h\cdot Z_j(1-\eta)+ c(s)Z_0Z_j(1-\eta)\\
&= O(R_1^{-\varepsilon})
\end{aligned}
\end{equation*}
for some small positive number $\varepsilon$ uniformly on $\tau\in (\tau_0,\tau_1)$. Then (\ref{e5:33}) hold by letting $R_1\to +\infty$.
Finally, we claim that for all $\tau_1 > 0$ large enough, any $\phi$ with $\|\phi\|_{a,\tau_1} < +\infty$ solving (\ref{e5:99}) and satisfying (\ref{e5:33}), we have
\begin{equation}\label{e5:34}
\|\phi\|_{a,\tau_1}\lesssim \|h\|_{2s+a,\tau_1}.
\end{equation}
Hence (\ref{e5:35}) holds.

To prove (\ref{e5:34}), we use the contradiction argument. Suppose that there exist sequences $\tau_1^k\to +\infty$ and $\phi_k$, $h_k$, $c_k$ satisfying
\begin{equation*}\label{e5:36}
\left\{
\begin{aligned}
&\partial_\tau\phi_k = -(-\Delta)^s\phi_k + pU^{p-1}(y)\phi_k + h_k - c_k(\tau)Z_0(y),~ y\in \mathbb{R}^n,~ \tau\geq \tau_0,\\
&\int_{\mathbb{R}^n}\phi_k(y, \tau)\cdot Z_j(y)dy = 0\text{ for all }\tau\in (\tau_0,\tau_1^k), j= 0, 1,\cdots, n+1,\\
&\phi_k(y,\tau_0) = 0, y\in \mathbb{R}^n
\end{aligned}
\right.
\end{equation*}
and
\begin{equation}\label{e5:38}
\|\phi_k\|_{a,\tau_1^k}=1,\quad \|h_k\|_{2s+a,\tau_1^k}\to 0.
\end{equation}
By (\ref{e5:102}), we have $\sup_{\tau\in (\tau_0, \tau_1^k)}\tau^\nu c_k(\tau)\to 0$.
First, we claim that
\begin{equation}\label{e5:37}
\sup_{\tau_0 < \tau < \tau_1^k}\tau^\nu|\phi_k(y,\tau)|\to 0
\end{equation}
holds uniformly on compact subsets of $\mathbb{R}^n$. Indeed, if for some $|y_k|\leq M$ and $\tau_0 < \tau_2^k < \tau_1^k$,
\begin{equation*}
(\tau_2^k)^\nu|\phi_k(y_k,\tau_2^k)|\geq \frac{1}{2},
\end{equation*}
then it is easy to see that $\tau_2^k\to +\infty$. Now, we define
\begin{equation*}
\tilde{\phi}_n(y,\tau) = (\tau_2^k)^\nu\phi_n(y,\tau_2^k + \tau).
\end{equation*}
Then
\begin{equation*}
\partial_\tau\tilde{\phi}_k = L_0[\tilde{\phi}_k] + \tilde{h}_k - \tilde{c}_k(\tau)Z_0(y)\text{ in }\mathbb{R}^n\times (\tau_0-\tau_2^k,0],
\end{equation*}
with $\tilde{h}_k\to 0$, $\tilde{c}_k\to 0$ uniformly on compact subsets of $\mathbb{R}^n\times (-\infty, 0]$ and
\begin{equation*}
|\tilde{\phi}_k(y,\tau)|\leq \frac{1}{1+|y|^a}\text{ in }\mathbb{R}^n\times (\tau_0-\tau_2^k,0].
\end{equation*}
Using the fact that $a\in (2s, n-2s)$ and the dominant convergence theorem, we have $\tilde{\phi}_k\to\tilde{\phi}$ uniformly on compact subsets of $\mathbb{R}^n\times (-\infty, 0]$ with $\tilde{\phi}\neq 0$ and
\begin{equation}\label{e5:103}
\left\{
\begin{aligned}
&\partial_\tau\tilde{\phi} = -(-\Delta)^s\tilde{\phi} + pU^{p-1}(y)\tilde{\phi}~~\text{ in }~~\mathbb{R}^n\times (-\infty, 0],\\
&\int_{\mathbb{R}^n}\tilde{\phi}(y, \tau)\cdot Z_j(y)dy = 0\text{ for all }\tau\in (-\infty, 0], ~ j= 0, 1,\cdots, n+1,\\
&|\tilde{\phi}(y,\tau)|\leq \frac{1}{1+|y|^a}\text{ in }\mathbb{R}^n\times (-\infty, 0],\\
&\tilde{\phi}(y,\tau_0) = 0, y\in \mathbb{R}^n.
\end{aligned}
\right.
\end{equation}
We claim that $\tilde{\phi} = 0$, which is a contradiction. By fractional parabolic regularity (see \cite{jinxiongjram2014}), $\tilde{\phi}(y,\tau)$ is smooth. A scaling argument shows
\begin{equation*}
(1+|y|^s)|(-\Delta)^\frac{s}{2}\tilde{\phi}| + |\tilde{\phi}_\tau| + |(-\Delta)^s\tilde{\phi}|\lesssim (1+|y|)^{-2s-a}.
\end{equation*}
Differentiating (\ref{e5:103}), we get $\partial_\tau\tilde{\phi}_\tau = -(-\Delta)^s\tilde{\phi}_\tau + pU^{p-1}(y)\tilde{\phi}_\tau$ and
\begin{equation*}
(1+|y|^s)|(-\Delta)^\frac{s}{2}\tilde{\phi}_\tau| + |\tilde{\phi}_{\tau\tau}| + |(-\Delta)^s\tilde{\phi}_\tau|\lesssim (1+|y|)^{-4s-a}.
\end{equation*}
Moreover, it holds that
\begin{equation*}
\frac{1}{2}\partial_\tau\int_{\mathbb{R}^n}|\tilde{\phi}_\tau|^2 + B(\tilde{\phi}_\tau, \tilde{\phi}_\tau) = 0,
\end{equation*}
where
\begin{equation*}
B(\tilde{\phi}, \tilde{\phi}) = \int_{\mathbb{R}^n}\left[|(-\Delta)^{\frac{s}{2}}\tilde{\phi}|^2 - pU^{p-1}(y)|\tilde{\phi}|^2\right]dy.
\end{equation*}
Since $\int_{\mathbb{R}^n}\tilde{\phi}(y, \tau)\cdot Z_j(y)dy = 0$ for all $\tau\in (-\infty, 0]$, $j= 0, 1,\cdots, n+1 $, $B(\tilde{\phi}, \tilde{\phi})\geq 0$. Also, we have
\begin{equation*}
\int_{\mathbb{R}^n}|\tilde{\phi}_\tau|^2 = -\frac{1}{2}\partial_\tau B(\tilde{\phi}, \tilde{\phi}).
\end{equation*}
From these relations,
\begin{equation*}
\partial_\tau\int_{\mathbb{R}^n}|\tilde{\phi}_\tau|^2 \leq 0,\quad \int_{-\infty}^0d\tau\int_{\mathbb{R}^n}|\tilde{\phi}_\tau|^2 < +\infty.
\end{equation*}
Hence $\tilde{\phi}_\tau = 0$. So $\tilde{\phi}$ is independent of $\tau$ and $L_0[\tilde{\phi}] = 0$. Since $\tilde{\phi}$ is bounded, by the nondegeneracy of $L_0$ (see, \cite{DelPinoSirePAMS}),  $\tilde{\phi}$ is a linear combination of $Z_j$, $j = 1,\cdots, n+1$. But $\int_{\mathbb{R}^n}\tilde{\phi}\cdot Z_j = 0$, $j = 1,\cdots, n$, $\tilde{\phi} = 0$, a contradiction. Thus (\ref{e5:37}) holds.

From (\ref{e5:38}), there exists a certain $y_k$ with $|y_k|\to +\infty$ such that
\begin{equation*}
(\tau_2^k)^\nu(1+|y_k|^a)|\phi_k(y_k, \tau_2^k)|\geq \frac{1}{2}.
\end{equation*}
Let
\begin{equation*}
\tilde{\phi}_k(z, \tau):=(\tau_2^k)^\nu|y_k|^a\phi_k(y_k+|y_k|z,|y_k|^{2s}\tau + \tau_2^k),
\end{equation*}
then
\begin{equation*}
\partial_\tau \tilde{\phi}_k = -(-\Delta)^s\tilde{\phi}_k + a_k\tilde{\phi}_k + \tilde{h}_k(z,\tau),
\end{equation*}
where
\begin{equation*}
\tilde{h}_k(z,\tau) = (\tau_2^k)^\nu|y_k|^{2s+a}h_k(y_k+|y_k|z,|y_k|^{2s}\tau + \tau_2^k).
\end{equation*}
By the assumption on $h_k$, one has
\begin{equation*}
|\tilde{h}_k(z,\tau)| \lesssim o(1)|\hat{y}_k+z|^{-2s-a}((\tau_2^k)^{-1}|y_k|^{2s}\tau + 1)^{-\nu}
\end{equation*}
with
\begin{equation*}
\hat{y}_k = \frac{y_k}{|y_k|}\to -\hat{e}
\end{equation*}
and $|\hat{e}|= 1$. Thus $\tilde{h}_k(z,\tau)\to 0$ uniformly on compact subsets of $\mathbb{R}^n\setminus\{\hat{e}\}\times (-\infty, 0]$ and $a_k$ has the same property. Moreover, $|\tilde{\phi}_k(0, \tau_0)|\geq \frac{1}{2}$ and
\begin{equation*}
|\tilde{\phi}_k(z,\tau)| \lesssim |\hat{y}_k+z|^{-a}\left((\tau_2^k)^{-1}|y_k|^{2s}\tau + 1\right)^{-\nu}.
\end{equation*}
Hence we may assume $\tilde{\phi}_k\to \tilde{\phi}\neq 0$ uniformly on compact subsets of $\mathbb{R}^n\setminus\{\hat{e}\}\times (-\infty,0]$ with $\tilde{\phi}$ satisfying
\begin{equation}\label{e5:39}
\tilde{\phi}_\tau = -(-\Delta)^s\tilde{\phi}\quad\text{in }\mathbb{R}^n\setminus\{\hat{e}\}\times (-\infty,0]
\end{equation}
and
\begin{equation}\label{e5:40}
|\tilde{\phi}(z,\tau)|\leq |z-\hat{e}|^{-a}\quad\text{in }\mathbb{R}^n\setminus\{\hat{e}\}\times (-\infty,0].
\end{equation}
Similar to Lemma 5.2 of \cite{sireweizhenghalf}, functions $\tilde{\phi}$ satisfying (\ref{e5:39}) and (\ref{e5:40}) must be equal to zero (A proof can be found in \cite{chenwei2018}), which is a contradiction and we conclude the validity of (\ref{e5:34}). The proof is complete.
\end{proof}
{\it Proof of Proposition \ref{proposition5.1}}.
First, we consider the problem
\begin{equation*}
\left\{
\begin{aligned}
&\partial_\tau\phi = -(-\Delta)^s\phi + pU^{p-1}(y)\phi + h(y,\tau) - c(\tau)Z_0,~ y\in \mathbb{R}^n,~ \tau\geq \tau_0,\\
&\phi(y,\tau_0) = 0,~ y\in \mathbb{R}^n.
\end{aligned}
\right.
\end{equation*}
Let $(\phi(y,\tau),c(\tau))$ be the unique solution of the nonlocal initial value problem (\ref{e5:32}). From Lemma \ref{l5:3}, for any $\tau_1 > \tau_0$, we have
\begin{equation*}
|\phi(y,\tau)|\lesssim\tau^{-\nu}(1+|y|)^{-a}\|h\|_{2s+a, \tau_1}\text{ for all }\tau\in (\tau_0, \tau_1), \,\,y\in \mathbb{R}^n
\end{equation*}
and
\begin{equation*}
|c(\tau)|\leq \tau^{-\nu}R^{a}\|h\|_{2s+a,\tau_1}\text{ for all }\tau\in (\tau_0,\tau_1).
\end{equation*}
By assumption, $\|h\|_{2s+a,\nu, \eta} < +\infty$ and $\|h\|_{2s+a, \tau_1}\leq \|h\|_{2s+a,\nu, \eta}$ for an arbitrary $\tau_1$. It follows that
\begin{equation*}
|\phi(y,\tau)|\lesssim\tau^{-\nu}(1+|y|)^{-a}\|h\|_{2s+a,\nu, \eta}\text{ for all }\tau\in (\tau_0, \tau_1),\,\, y\in \mathbb{R}^n
\end{equation*}
and
\begin{equation*}
|c(\tau)|\leq \tau^{-\nu}R^{a}\|h\|_{2s+a,\nu, \eta}\text{ for all }\tau\in (\tau_0, \tau_1).
\end{equation*}
By the arbitrariness of $\tau_1$,
\begin{equation*}
|\phi(y,\tau)|\lesssim\tau^{-\nu}(1+|y|)^{-a}\|h\|_{2s+a,\nu, \eta}\text{ for all }\tau\in (\tau_0, +\infty),\,\, y\in \mathbb{R}^n
\end{equation*}
and
\begin{equation*}
|c(\tau)|\leq \tau^{-\nu}R^{a}\|h\|_{2s+a,\nu, \eta}\text{ for all }\tau\in (\tau_0, +\infty).
\end{equation*}
From the regularity result of \cite{silvestreium2012} and a scaling argument, we get the validity of (\ref{e5:100}) and (\ref{e5:101}).\qed
\subsection{The solvability conditions: choice of the parameters $\lambda$ and $\xi$.}
Denote
\begin{equation*}
\lambda(t) = \left(
\begin{matrix}
\lambda_1(t)\\
\lambda_2(t)\\
\vdots\\
\lambda_k(t)
\end{matrix}
\right),
\dot{\lambda}(t) = \left(
\begin{matrix}
\dot{\lambda}_1(t)\\
\dot{\lambda}_2(t)\\
\vdots\\
\dot{\lambda}_k(t)
\end{matrix}
\right),
\xi(t) = \left(
\begin{matrix}
\xi_1(t)\\
\xi_2(t)\\
\vdots\\
\xi_k(t)
\end{matrix}
\right),
\dot{\xi}(t) = \left(
\begin{matrix}
\dot{\xi}_1(t)\\
\dot{\xi}_2(t)\\
\vdots\\
\dot{\xi}_k(t)
\end{matrix}
\right),
q = \left(
\begin{matrix}
q_1\\
q_2\\
\vdots\\
q_k
\end{matrix}
\right).
\end{equation*}
First we consider (\ref{e5:7}) in the case $l = n+1$.
\begin{lemma}\label{l5:1}
When $l = n+1$, (\ref{e5:7}) is equivalent to
\begin{equation}\label{e5:11}
\dot{\lambda}_j + \frac{1}{t}\left(P^Tdiag\left(\frac{(2s-1)\bar{\sigma}_rb_r^{2-2s} + 1}{n-4s}\right)P\lambda\right)_j = \Pi_1[\lambda, \xi, \dot{\lambda}, \dot{\xi}, \phi](t)
\end{equation}
where the matrix $P$, the numbers $\bar{\sigma}_r > 0$ and $b_r > 0$ are defined in Section 2. The right hand side term can be expressed as
\begin{equation}\label{expressionpi1}
\begin{aligned}
\Pi_1[\lambda, \xi, \dot{\lambda}, \dot{\xi}, \phi](t) =& \frac{t_0^{-\varepsilon}}{R^{a-2s}}\mu_0^{n+1-4s + \sigma}(t)f(t)\\
& + \frac{t_0^{-\varepsilon}}{R^{a-2s}}\Theta\left[\dot{\lambda},\dot{\xi}, \mu_0^{n-4s}(t)\lambda, \mu_0^{n-4s}(\xi-q), \mu_0^{n+1-4s+\sigma}\phi\right](t)
\end{aligned}
\end{equation}
where $f(t)$ and $\Theta\left[\dot{\lambda},\dot{\xi}, \mu_0^{n-4s}(t)\lambda, \mu_0^{n-4s}(\xi-q), \mu_0^{n+1-4s+\sigma}\phi\right](t)$ are smooth and bounded functions for $t\in [t_0,\infty)$. Further, the following estimates hold,
\begin{equation*}
\left|\Theta[\dot{\lambda}_1](t) - \Theta[\dot{\lambda}_2](t)\right|\lesssim \frac{t_0^{-\varepsilon}}{R^{a-2s}}|\dot{\lambda}_1(t) - \dot{\lambda}_2(t)|
\end{equation*}
\begin{equation*}
\left|\Theta[\dot{\xi}_1](t) - \Theta[\dot{\xi}_2](t)\right|\lesssim \frac{t_0^{-\varepsilon}}{R^{a-2s}}|\dot{\xi}_1(t) - \dot{\xi}_2(t)|,
\end{equation*}
\begin{equation*}
\left|\Theta[\mu_0^{n-4s}\lambda_1](t) - \Theta[\mu_0^{n-4s}\lambda_2](t)\right|\lesssim \frac{t_0^{-\varepsilon}}{R^{a-2s}}|\dot{\lambda}_1(t) - \dot{\lambda}_2(t)|
\end{equation*}
\begin{equation*}
\left|\Theta[\mu_0^{n-4s}(\xi_1-q)](t) - \Theta[\mu_0^{n-4s}(\xi_2-q)](t)\right|\lesssim \frac{t_0^{-\varepsilon}}{R^{a-2s}}|\xi_1(t) - \xi_2(t)|,
\end{equation*}
\begin{equation}\label{e5:104}
\left|\Theta[\mu_0^{n+1-4s+\sigma}\phi_1](t) - \Theta[\mu_0^{n+1-4s+\sigma}\phi_2](t)\right|\lesssim \frac{t_0^{-\varepsilon}}{R^{a-2s}}\|\phi_1(t) - \phi_2(t)\|_{n-2s+\sigma, a}.
\end{equation}
\end{lemma}
\begin{proof}
Suppose $\phi$ satisfies (\ref{e4:25}). For a fixed $j \in \{1,\cdots,k\}$, we compute
\begin{eqnarray*}
\int_{B_{2R}}H_j[\phi,\lambda,\xi,\dot{\lambda},\dot{\xi}](y,t(\tau))Z_{n+1}(y)dy,
\end{eqnarray*}
where $H_j$ is given by (\ref{e5:2}). Decompose
\begin{equation*}
\begin{aligned}
&\mu_{0j}^{\frac{n+2s}{2}}S_{\mu, \xi, j}(\xi_j + \mu_{0j}y, t)\\
&=\left(\frac{\mu_{0j}}{\mu_j}\right)^{\frac{n+2s}{2}}\left[\mu_{0j}S_1(z, t) + \lambda_jb_j^{2s-1}S_2(z, t) + \mu_jS_3(z, t)\right]_{z = \xi_j + \mu_jy}\\
&\quad+\left(\frac{\mu_{0j}}{\mu_j}\right)^{\frac{n+2s}{2}}\mu_{0j}\left[S_1(\xi_j+\mu_{0j}y, t)-S_1(\xi_j+\mu_{j}y, t)\right]\\
&\quad+\left(\frac{\mu_{0j}}{\mu_j}\right)^{\frac{n+2s}{2}}\lambda_jb_j^{2s-1}\left[S_2(\xi_j+\mu_{0j}y, t)-S_2(\xi_j+\mu_{j}y, t)\right]\\
&\quad+\left(\frac{\mu_{0j}}{\mu_j}\right)^{\frac{n+2s}{2}}\mu_j\left[S_3(\xi_j+\mu_{0j}y, t)-S_3(\xi_j+\mu_{j}y, t)\right],
\end{aligned}
\end{equation*}
where
\begin{equation*}
\begin{aligned}
&S_1(z)\\
&= (b_j\mu_0)^{2s-2}\dot{\lambda}_j\\
&\times\left(Z_{n+1}\left(\frac{z-\xi_j}{\mu_j}\right) + \frac{n-2s}{2}\alpha_{n, s}\frac{1}{\left(1+\left|\frac{z-\xi_j}{\mu_j}\right|^2\right)^{\frac{n-2s}{2}}}-2sApU\left(\frac{z-\xi_j}{\mu_j}\right)^{p-1}\right)\\
&\quad -\mu_0^{n-2s-2}pU(y_j)^{p-1}\sum_{i=1}^kM_{ij}\lambda_i,
\end{aligned}
\end{equation*}
\begin{equation*}
\begin{aligned}
S_2(z) =& (2s-1)\mu_0^{2s-2}\dot{\mu}_0\left(Z_{n+1}\left(\frac{z-\xi_j}{\mu_j}\right) + \frac{n-2s}{2}\alpha_{n, s}\frac{1}{\left(1+\left|\frac{z-\xi_j}{\mu_j}\right|^2\right)^{\frac{n-2s}{2}}}\right)\\
&+pU\left(\frac{z-\xi_j}{\mu_j}\right)^{p-1}\mu_0^{n-2s-1}\\
&\times\bigg(-b_j^{n-4s}H(q_j, q_j) + \sum_{i\neq j}b_j^{\frac{n-6s}{2}}b_i^{\frac{n-2s}{2}}G(q_j, q_i)+(2s-1)B\bigg)
\end{aligned}
\end{equation*}
and
\begin{equation*}
\begin{aligned}
S_3(z) =& \mu_j^{2s-2}\alpha_{n, s}(n-2s)\frac{\dot{\xi}_j\cdot \frac{z-\xi_j}{\mu_j}}{\left(1+\left|\frac{z-\xi_j}{\mu_j}\right|^2\right)^{\frac{n-2s}{2}+1}}+ pU\left(\frac{z-\xi_j}{\mu_j}\right)^{p-1}\\
&\times\left(-\mu_j^{n-2s}\nabla H(q_j, q_j) + \sum_{i\neq j}\mu_j^{\frac{n-2s}{2}}\mu_i^{\frac{n-2s}{2}}\nabla G(q_j, q_i)\right)\cdot \left(\frac{z-\xi_j}{\mu_j}\right).
\end{aligned}
\end{equation*}
By direct computations, we have
\begin{equation*}
\begin{aligned}
\int_{B_{2R}}S_1(\xi_j+\mu_jy)Z_{n+1}(y)dy &= (2sAc_1+c_2)(1+O(R^{4s-n}))\dot{\lambda}_j(b_j\mu_0)^{2s-2} \\
&\quad  + c_1(1+O(R^{-2s}))\mu_0^{n-2s-2}\sum_{i=1}^kM_{ij}\lambda_i,
\end{aligned}
\end{equation*}
\begin{equation*}
\begin{aligned}
&\int_{B_{2R}}S_2(\xi_j+\mu_jy)Z_{n+1}(y)dy\\
&\quad =  -(2s-2)\mu_0^{n-2s-1}\frac{2sAc_1 + c_2}{(n-4s)c_{n, s}^{n-4s}}+ O(R^{4s-n}+R^{-2s})\mu_{0}^{n-2s-1}\\
&\quad = -(2s-2)\mu_0^{n-2s-1}\frac{2sc_1}{(n-2s)}+ O(R^{4s-n}+R^{-2s})\mu_{0}^{n-2s-1}
\end{aligned}
\end{equation*}
and
\begin{equation*}
\int_{B_{2R}}S_3(\xi_j+\mu_jy)Z_{n+1}(y)dy = 0 \,\,(\text{by symmetry}).
\end{equation*}
Since $\frac{\mu_{0j}}{\mu_j} = (1 + \frac{\lambda_j}{\mu_{0j}})^{-1}$, for any $l = 1, 2, 3$, we have
\begin{equation*}
\begin{aligned}
&\int_{B_{2R}}[S_l(\xi_j+\mu_{0j}y, t)-S_l(\xi_j+\mu_{j}y, t)]Z_{n+1}(y)dy \\
&= g(t,\frac{\lambda}{\mu_0})\mu_0^{2s-2}\dot{\lambda}_j + g(t, \frac{\lambda}{\mu_0})\mu_0^{2s-2}\dot{\xi} +g(t,\frac{\lambda}{\mu_0})\sum_i\mu_0^{n-2s-2}\lambda_i + \mu_0^{n-2s-1+\sigma}f(t),
\end{aligned}
\end{equation*}
where $f$, $g$ are smooth and bounded functions such that $g(\cdot, s)\sim s$ as $s\to 0$. Thus
\begin{equation*}
\begin{aligned}
& c\left(\frac{\mu_{j}}{\mu_{0j}}\right)^{\frac{n+2s}{2}}\mu_{0j}^{1-2s}\int_{B_{2R}}\mu_{0j}^{\frac{n+2s}{2}}S_{\mu, \xi, j}(\xi_j + \mu_{0j}y, t)Z_{n+1}(y)dy\\
& = \left[\dot{\lambda}_j + \frac{1}{t}\left(P^Tdiag\left(\frac{\frac{n-2s}{2s}\bar{\sigma}_rb_r^{2-2s} + 1}{n-4s}\right)P\lambda\right)_j\right]\\
&\quad + \frac{t_0^{-\varepsilon}}{R^{a-2s}}g(t,\frac{\lambda}{\mu_0})(\dot{\lambda} + \dot{\xi}) + \frac{t_0^{-\varepsilon}}{R^{a-2s}}\mu_0^{n-4s}g(t, \frac{\lambda}{\mu_0}),
\end{aligned}
\end{equation*}
where $c$ is a positive number, the function $g$ is smooth, bounded and $g(\cdot, s)\sim s$ as $s\to 0$.

Next we compute $p\mu_{0j}^{\frac{n-2s}{2}}(1 + \frac{\lambda_j}{\mu_{0j}})^{-2s}\int_{B_{2R}}U^{p-1}(\frac{\mu_{0j}}{\mu_j}y)\psi(\xi_j + \mu_{0j}y, t)Z_{n+1}(y)dy$. The principal part is $I: = \int_{B_{2R}}U^{p-1}(y)\psi(\xi_j + \mu_{0j}y, t)Z_{n+1}(y)dy$. Recall $\psi = \psi[\lambda,\xi,\dot{\lambda},\dot{\xi}, \phi](y, t)$, we have
\begin{equation*}
\begin{aligned}
I &=\psi[0,q,0,0,0](q_j, t)\int_{B_{2R}}U^{p-1}(y)Z_{n+1}(y)dy\\
&\quad + \int_{B_{2R}}U^{p-1}(y)Z_{n+1}(y)(\psi[0,q,0,0,0](\xi_j+\mu_{0j}y, t)-\psi[0,q,0,0,0](q_j, t))dy\\
&\quad + \int_{B_{2R}}U^{p-1}(y)Z_{n+1}(y)(\psi[\lambda,\xi,\dot{\lambda},\dot{\xi}, \phi] - \psi[0,q,0,0,0])(\xi_j+\mu_{0j}y, t)dy\\
& = I_1 + I_2 + I_3.
\end{aligned}
\end{equation*}
By (\ref{e4:27}), $I_1 = \frac{t_0^{-\varepsilon}}{R^{a-2s}}\mu_0^{\frac{n-2s}{2}+\sigma}f(t)$ with $f$ smooth and bounded. By (\ref{e4:28}), $I_2 = \frac{t_0^{-\varepsilon}}{R^{a-2s}}\mu_0^{\frac{n-2s}{2}+\sigma}g(t, \frac{\lambda}{\mu_0}, \xi -q)$ for a smooth and bounded function $g$ satisfying $g(\cdot, s, \cdot)\sim s$ and $g(\cdot,\cdot, s)\sim s$ as $s \to 0$. From the mean value theorem again, we have
\begin{equation*}
\begin{aligned}
I_3 =& \int_{B_{2R}}U^{p-1}(y)Z_{n+1}(y)\bigg[\partial_\lambda\psi[0,q,0,0,0][s\lambda](\xi_j + \mu_{0j}y, t)\\
&+ \partial_\xi\psi[0,q,0,0,0][s(\xi_j-q_j)](\xi_j + \mu_{0j}y, t)+ \partial_{\dot{\lambda}}\psi[0,q,0,0,0][s\dot{\lambda}](\xi_j + \mu_{0j}y, t)\\
&+\partial_{\dot{\xi}}\psi[0,q,0,0,0][s\dot{\xi}](\xi_j + \mu_{0j}y, t) + \partial_{\phi}\psi[0,q,0,0,0][s\phi](\xi_j + \mu_{0j}y, t)\bigg]dy
\end{aligned}
\end{equation*}
for some $s\in (0, 1)$.
Using Proposition \ref{p4:4.2}, $I_3$ is the sum of terms like
\begin{equation*}
\mu_0^{-\frac{n-6s}{2}-1+\sigma}\frac{t_0^{-\varepsilon}}{R^{a-2s}}f(t)(\dot{\lambda} + \dot{\xi})F[\lambda, \xi, \dot{\lambda}, \dot{\xi}, \phi](t)
\end{equation*}
and
$$
\mu_0^{\frac{n-2s}{2}-1}\frac{t_0^{-\varepsilon}}{R^{a-2s}}f(t)(\lambda + \xi)F[\lambda, \xi, \dot{\lambda}, \dot{\xi}, \phi](t),
$$
where $f$ is a smooth, bounded function and $F$ is a nonlocal operator satisfying $F[0,q,0,0,0](t)$ bounded.

Now, we consider the terms $B_j[\phi_j]$, $B_j^0[\phi_j]$ and obtain that
\begin{equation*}
\int_{B_{2R}}B_j[\phi_j](y, t)Z_{n+1}(y)dy = \frac{t_0^{-\varepsilon}}{R^{a-2s}}[\mu_0^{n+1-4s+\sigma}(t)\ell[\phi](t) + \dot{\xi}_j\ell[\phi](t)]
\end{equation*}
and
\begin{equation*}
\int_{B_{2R}}B^0_j[\phi_j](y, t)Z_{n+1}(y)dy = \frac{t_0^{-\varepsilon}}{R^{a-2s}}\mu_0^{n-2s-1}g\left(\frac{\lambda}{\mu_0}\right)[\phi](t)
\end{equation*}
for a smooth function $g(s)$ satisfying $g(s)\sim s$ as $s\to 0$, $\ell[\phi](t)$ is smooth and bounded in $t$.
Combining the above estimates, we conclude the result.
\end{proof}

Similarly, we compute
\begin{equation}
\int_{B_{2R}}H_j[\lambda,\xi, \dot{\lambda}, \dot{\xi},\phi](y, t(\tau))Z_l(y)dy,
\end{equation}
for any $j = 1, \cdots, k$, $l = 1,\cdots, n$. We have
\begin{lemma}\label{l5:2}
For $j = 1, \cdots, k$, $l = 1,\cdots, n$, (\ref{e5:7}) is equivalent to
\begin{equation}\label{e5:18}
\dot{\xi}_j = \Pi_{2, j}[\lambda, \xi, \dot{\lambda}, \dot{\xi}, \phi](t),
\end{equation}
\begin{equation*}
\begin{aligned}
&\Pi_{2, j}[\lambda, \xi, \dot{\lambda}, \dot{\xi}, \phi](t)\\
&= \mu_0^{n-4s+2}c\left[b_j^{n-2s}\nabla H(q_j, q_j) - \sum_{i\neq j}b_j^{\frac{n-2s}{2}}b_i^{\frac{n-2s}{2}}\nabla G(q_j, q_i)\right]+ \mu_0^{n-4s+2+\sigma}(t)f_j(t)\\
&\quad + \frac{t_0^{-\varepsilon}}{R^{a-2s}}\Theta[\dot{\lambda},\dot{\xi}, \mu_0^{n-2s-2}(t)\lambda, \mu_0^{n-2s-1}(\xi-q), \mu_0^{n+1-4s+\sigma}\phi](t),
\end{aligned}
\end{equation*}
where $c = \frac{p\int_{\mathbb{R}^n}U^{p-1}\frac{\partial U}{\partial y_1}y_1dy}{\int_{\mathbb{R}^n}\left(\frac{\partial U}{\partial y_1}\right)^2dy}$, $f_j(t)$ is an $n$ dimensional vector function which is smooth and bounded for $t\in [t_0, \infty)$. The function $\Theta$ has the same properties as in Lemma \ref{l5:1}.
\end{lemma}
The proof of Lemma \ref{l5:2} is similar to that of Lemma \ref{l5:1} so we omit it.

From Lemma \ref{l5:1} and Lemma \ref{l5:2}, we know that the orthogonality conditions
\begin{equation*}
\int_{B_{2R}}H_j[\lambda,\xi, \dot{\lambda}, \dot{\xi},\phi](y, t(\tau))Z_l(y)dy, \mbox{ for } j = 1,\cdots, k \mbox{ and } l = 1, \cdots, n+1,
\end{equation*}
are equivalent to the system of ODEs for $\lambda$ and $\xi$
\begin{equation}\label{e5:9}
\left\{
\begin{aligned}
&\dot{\lambda}_j + \frac{1}{t}\left(P^Tdiag\left(\frac{\frac{n-2s}{2s}\bar{\sigma}_rb_r^{2-2s} + 1}{n-4s}\right)P\lambda\right)_j = \Pi_1[\lambda, \xi, \dot{\lambda}, \dot{\xi}, \phi](t),\\
&\dot{\xi}_j = \Pi_{2, j}[\lambda, \xi, \dot{\lambda}, \dot{\xi}, \phi](t), j = 1,\cdots, k.
\end{aligned}
\right.
\end{equation}
System (\ref{e5:9}) is solvable for parameters $\lambda$ and $\xi$ satisfying (\ref{e4:21}) and (\ref{e4:22}). Indeed, we have
\begin{prop}\label{p5:5.1}
There exists a solution $\lambda = \lambda[\phi](t)$, $\xi = \xi[\phi](t)$ to (\ref{e5:9}) satisfying (\ref{e4:21}) and (\ref{e4:22}). For $t\in (t_0, \infty)$, it holds that
\begin{equation}\label{e5:12}
\mu_0^{-(1+\sigma)}(t)\big|\lambda[\phi_1](t) - \lambda[\phi_2](t)\big|\lesssim \frac{t_0^{-\varepsilon}}{R^{a-2s}}\|\phi_1 - \phi_2\|_{n-2s+\sigma, a}
\end{equation}
and
\begin{equation}\label{e5:13}
\mu_0^{-(1+\sigma)}(t)\big|\xi[\phi_1](t) - \xi[\phi_2](t)\big|\lesssim \frac{t_0^{-\varepsilon}}{R^{a-2s}}\|\phi_1 - \phi_2\|_{n-2s+\sigma, a}.
\end{equation}
\end{prop}
\begin{proof}
Let $h$ be a vector function with $\|h\|_{n+1-4s+\sigma} \lesssim \frac{1}{R^{a-2s}}$. The solution to
\begin{equation}\label{e5:14}
\dot{\lambda}_j + \frac{1}{t}\left(P^Tdiag\left(\frac{\frac{n-2s}{2s}\bar{\sigma}_rb_r^{2-2s} + 1}{n-4s}\right)P\lambda\right)_j = h(t)_j
\end{equation}
can be expressed as
\begin{equation*}
\lambda(t) = P^T\nu(t),\nu(t) = \left(
\begin{matrix}
\nu_1(t)\\
\nu_2(t)\\
\vdots\\
\nu_k(t)
\end{matrix}
\right),
\end{equation*}
\begin{equation}\label{e5:15}
\nu_j(t) = t^{-\frac{1+\frac{n-2s}{2s}\bar{\sigma}_jb_j^{2-2s}}{n-4s}}\left[d_j + \int_{t_0}^t\tau^{\frac{1+\frac{n-2s}{2s}\bar{\sigma}_jb_j^{2-2s}}{n-4s}}(Ph)_j(\tau)d\tau\right],
\end{equation}
where $d_j$, $j = 1, \cdots, k$ are arbitrary constants. Then, for $0\leq d := \max_{i = 1,\cdots, k}|d_i|$, we have
\begin{equation*}
\|t^{\frac{1+\sigma}{n-4s}}\lambda(t)\|_{L^\infty(t_0,\infty)}\lesssim t_0^{-\frac{\bar{\sigma} - \sigma}{n-4s}}d + \|h\|_{n+1-4s+\sigma}
\end{equation*}
and
\begin{equation*}
\|\dot{\lambda}(t)\|_{n+1-4s+\sigma}\lesssim t_0^{-\frac{\bar{\sigma} - \sigma}{n-4s}}d + \|h\|_{n+1-4s+\sigma}.
\end{equation*}

Let $\Lambda(t) = \dot{\lambda}(t)$, then
\begin{equation}\label{e5:17}
\Lambda + \frac{1}{t}\left(P^Tdiag\left(\frac{\frac{n-2s}{2s}\bar{\sigma}_rb_r^{2-2s} + 1}{n-4s}\right)\right)P\int_{t}^\infty\Lambda(s)ds = h(t),
\end{equation}
which defines a linear operator $\mathcal{L}_1: h\to \Lambda$ associating to any $h$ with $\|h\|_{n+1-4s+\sigma}$ bounded the solution $\Lambda$.
$\mathcal{L}_1$ is continuous between the spaces $L^\infty(t_0, \infty)^k$ with the $\|\cdot\|_{n+1-4s+\sigma}$-topology.

For any $h: [t_0,\infty)\to \mathbb{R}^k$ with $\|h\|_{n+1-4s+\sigma}$ bounded, the solution to
\begin{equation}\label{e5:19}
\dot{\xi}_j = \mu_0^{n-4s+2}c\left[b_j^{n-2s}\nabla H(q_j, q_j) - \sum_{i\neq j}b_j^{\frac{n-2s}{2}}b_i^{\frac{n-2s}{2}}\nabla G(q_j, q_i)\right] + h(t)
\end{equation}
is given by
\begin{equation}\label{e5:20}
\xi_j(t) = \xi_j^0(t) + \int_{t}^\infty h(s)ds,
\end{equation}
where
\begin{equation*}
\xi_j^0(t) = q_j + c\left[-b_j^{n-2s}\nabla H(q_j, q_j) + \sum_{i\neq j}b_j^{\frac{n-2s}{2}}b_i^{\frac{n-2s}{2}}\nabla G(q_j, q_i)\right]\int_{t}^\infty\mu_0^{n-4s+2}(s)ds.
\end{equation*}
Then we have
\begin{equation*}
|\xi_j(t) - q_j|\lesssim t^{-\frac{2}{n-4s}} + t^{-\frac{1+\sigma}{n-4s}}\|h\|_{n+1-4s+\sigma}
\end{equation*}
and
\begin{equation*}
\|\dot{\xi}_j - \dot{\xi}_j^0\|_{n+1-4s+\sigma}\lesssim \|h\|_{n+1-4s+\sigma}.
\end{equation*}
Let $\Xi(t) = \dot{\xi}(t) - \dot{\xi}^0$ which is a vector function, then (\ref{e5:20}) defines a linear operator $\mathcal{L}_2: h\to \Xi$ which is continuous in the $\|\cdot\|_{n+1-4s+\sigma}$-topology.

Observe that ($\lambda$, $\xi$) is a solution of (\ref{e5:9}) if ($\Lambda = \dot{\lambda}$, $\Xi = \dot{\xi} - \dot{\xi}^0$) is a fixed point for the problem
\begin{equation}\label{e5:23}
(\Lambda, \Xi) = \mathcal{A}(\Lambda, \Xi)
\end{equation}
where
\begin{equation*}
\mathcal{A}: = \left(\mathcal{L}_1(\hat{\Pi}_1[\Lambda, \Xi, \phi], \mathcal{L}_2(\hat{\Pi}_2[\Lambda, \Xi, \phi])\right) = \left(\bar{A}_1(\Lambda, \Xi), \bar{A}_2(\Lambda, \Xi)\right)
\end{equation*}
with
\begin{equation*}
\hat{\Pi}_1[\Lambda, \Xi, \phi] := \Pi_1\left[\int_{t}^\infty \Lambda, q + \int_{t}^\infty\Xi, \Lambda, \Xi, \phi\right],
\end{equation*}
and
\begin{equation*}
\hat{\Pi}_2[\Lambda\Xi, \phi] := \Pi_2\left[\int_{t}^\infty \Lambda, q + \int_{t}^\infty\Xi, \Lambda, \Xi, \phi\right].
\end{equation*}
Let
\begin{equation*}
K := R^{a-2s}\max\{\|f\|_{n+1-4s+\sigma}, \|f_1\|_{n+1-4s+\sigma},\cdots, \|f_k\|_{n+1-4s+\sigma}\}
\end{equation*}
where $f$, $f_1$, $\cdots$, $f_k$ are defined in Lemma \ref{l5:1} and Lemma \ref{l5:2}. Now, we show that problem (\ref{e5:23}) has a fixed point $(\Lambda, \Xi)$ in the following space
\begin{equation*}
\begin{aligned}
\mathcal{B} = \Big\{(\Lambda, \Xi)\in L^\infty(t_0, \infty)&\times L^\infty(t_0, \infty): \\
&\|\Lambda\|_{n-2s-1+(2s-1)\sigma} + \|\Xi\|_{n-2s-1+(2s-1)\sigma}\leq \frac{cK}{R^{a-2s}}\Big\}
\end{aligned}
\end{equation*}
for suitable $c > 0$.
Indeed, from (\ref{expressionpi1}) we have
\begin{equation*}
\begin{aligned}
&\left|t^{\frac{n+1-4s+\sigma}{n-4s}}\bar{A}_1(\Lambda, \Xi)\right|\\
&\lesssim t_0^{-\frac{\bar{\sigma} - \sigma}{n-4s}}d + \frac{1}{R^{a-2s}}\|\phi\|_{n-2s+\sigma, a} + \frac{K}{R^{a-2s}}\\
&\quad + \frac{t_0^{-\varepsilon}}{R^{a-2s}}\|\Lambda\|_{n+1-4s+\sigma} + \frac{t_0^{-\varepsilon}}{R^{a-2s}}\|\Xi\|_{n+1-4s+\sigma}
\end{aligned}
\end{equation*}
and
\begin{equation*}
\begin{aligned}
\left|t^{\frac{n+1-4s+\sigma}{n-4s}}\bar{A}_2(\Lambda, \Xi)\right|&\lesssim\frac{1}{R^{a-2s}}\|\phi\|_{n-2s+\sigma, a} + \frac{K}{R^{a-2s}}\\
&\quad + \frac{t_0^{-\varepsilon}}{R^{a-2s}}\|\Lambda\|_{n+1-4s+\sigma} + \frac{t_0^{-\varepsilon}}{R^{a-2s}}\|\Xi\|_{n+1-4s+\sigma}.
\end{aligned}
\end{equation*}
Thus, for $d$ satisfying $t_0^{-\frac{\bar{\sigma}-\sigma}{n-4s}}d < \frac{K}{R^{a-2s}}$ and the constant $c$ chosen sufficiently large, $\mathcal{A}(\mathcal{B})\subset \mathcal{B}$. As for the Lipschitz property of $\mathcal{A}$, we have
\begin{equation*}
\begin{aligned}
& t^{\frac{n+1-4s+\sigma}{n-4s}}\left|\bar{A}_1(\Lambda_1, \Xi)-\bar{A}_1(\Lambda_2, \Xi)\right|\\
& = t^{\frac{n+1-4s+\sigma}{n-4s}}\left|\mathcal{L}_1(\hat{\Pi}_1[\Lambda_1, \Xi, \phi]-\hat{\Pi}_1[\Lambda_2, \Xi, \phi])\right|\\
& \leq t^{\frac{n+1-4s+\sigma}{n-4s}}t_0^{-\varepsilon}\left|\mathcal{L}_1(\Theta_2(\Lambda_1, \Xi)-\Theta_2(\Lambda_2, \Xi))\right|\\
&\quad\quad ~~~~+ t^{\frac{n+1-4s+\sigma}{n-4s}}t_0^{-\varepsilon}\left|\mathcal{L}_1(\mu_0^{n-2s-2}\Theta_3(\Lambda_1, \Xi)-\mu_0^{n-2s-2}\Theta_3(\Lambda_2, \Xi))\right|\\
& \leq t_0^{-\varepsilon}\|\Lambda_1 - \Lambda_2\|_{n+1-4s+\sigma}.
\end{aligned}
\end{equation*}
The same estimate holds for $\left|\bar{A}_1(\Lambda, \Xi_1)-\bar{A}_1(\Lambda, \Xi_2)\right|.$ Thus, we have
\begin{equation*}
\|\mathcal{A}(\Lambda_1, \Xi_1) - \mathcal{A}(\Lambda_2, \Xi_2)\|_{n+1-4s+\sigma}\leq t_0^{-\varepsilon}\|\Lambda_1 - \Lambda_2\|_{n+1-4s+\sigma}.
\end{equation*}
Since $t_0^{-\varepsilon} < 1$ when $t_0$ is large enough, $\mathcal{A}$ is a contraction map. Hence, from the Contraction Mapping Theorem, there exists a solution to system (\ref{e5:9}) with $\lambda$, $\xi$ satisfying (\ref{e4:21}) and (\ref{e4:22}) .

To prove (\ref{e5:12}) and (\ref{e5:13}), we observe that $\bar{\lambda} = \lambda[\phi_1] - \lambda[\phi_2]$ and $\bar{\xi} = \xi[\phi_1] - \xi[\phi_2]$ satisfy
\begin{equation*}
\dot{\lambda} + \frac{1}{t}\left(P^Tdiag\left(\frac{\frac{n-2s}{2s}\bar{\sigma}_rb_r^{2-2s} + 1}{n-4s}\right)P\lambda\right) = \bar{\Pi}_1(t),~~\dot{\xi}_j = \bar{\Pi}_{2, j}(t),\quad j = 1,\cdots, k
\end{equation*}
where
\begin{equation*}
\begin{aligned}
&(\bar{\Pi}_1(t))_j\\
&= cp\mu_j^{\frac{n-2s}{2}}\mu_{0j}^{1-2s}\int_{B_{2R}}U^{p-1}\left(\frac{\mu_{0j}}{\mu_j}y\right)\left[\psi[\phi_1] - \psi[\phi_2]\right](\xi_j+\mu_{0j}y, t)Z_{n+1}(y)dy\\
&\quad + c\left(\frac{\mu_{j}}{\mu_{0j}}\right)^{\frac{n+2s}{2}}\mu_{0j}^{1-2s}\int_{B_{2R}}\Big[B_j[(\phi_1)_j] - B_j[(\phi_2)_j]\Big]Z_{n+1}(y)dy\\
&\quad + c\left(\frac{\mu_{j}}{\mu_{0j}}\right)^{\frac{n+2s}{2}}\mu_{0j}^{1-2s}\int_{B_{2R}}\Big[B_j^0[(\phi_1)_j] - B_j^0[(\phi_2)_j]\Big]Z_{n+1}(y)dy
\end{aligned}
\end{equation*}
and
\begin{equation*}
\begin{aligned}
&(\bar{\Pi}_1(t))_j\\
&= cp\mu_j^{\frac{n-2s}{2}}\mu_{0j}^{1-2s}\int_{B_{2R}}U^{p-1}\left(\frac{\mu_{0j}}{\mu_j}y\right)\left[\psi[\phi_1] - \psi[\phi_2]\right](\xi_j+\mu_{0j}y, t)\frac{\partial U}{\partial y_j}(y)dy\\
&\quad + c\left(\frac{\mu_{j}}{\mu_{0j}}\right)^{\frac{n+2s}{2}}\mu_{0j}^{1-2s}\int_{B_{2R}}\Big[B_j[(\phi_1)_j] - B_j[(\phi_2)_j]\Big]\frac{\partial U}{\partial y_j}(y)dy\\
&\quad + c\left(\frac{\mu_{j}}{\mu_{0j}}\right)^{\frac{n+2s}{2}}\mu_{0j}^{1-2s}\int_{B_{2R}}\Big[B_j^0[(\phi_1)_j] - B_j^0[(\phi_2)_j]\Big]\frac{\partial U}{\partial y_j}(y)dy.
\end{aligned}
\end{equation*}
Then (\ref{e5:12}) and (\ref{e5:13}) follow from (\ref{e5:104}). This completes the proof.
\end{proof}

\section{Gluing: Proof of Theorem \ref{t:main}}
After we have chosen parameters $\lambda = \lambda[\phi]$ and $\xi = \xi[\phi]$ such that the orthogonality conditions (\ref{e5:7}) hold, we only need to solve problem (\ref{e5:4}) in the class of functions with $\|\phi\|_{a,\nu}$ (or equivalently $\|\phi\|_{n-2s+\sigma,a}$) bounded. With the chosen parameters, we can apply Proposition 5.1 which states that there exists a linear operator $\mathcal{T}$ associating any function $h(y,\tau)$ with $\|h\|_{2s+a, \nu}$-bounded the solution to (\ref{e5:31}). Thus problem (\ref{e5:4}) is reduced to a fixed point problem
\begin{equation}
\phi = (\phi_1,\cdots, \phi_k) = \mathcal{A}(\phi): = (\mathcal{T}(H_1[\lambda,\xi,\dot{\lambda},\dot{\xi},\phi]), \cdots, \mathcal{T}(H_k[\lambda,\xi,\dot{\lambda},\dot{\xi},\phi])).
\end{equation}
We claim that, for each $j = 1,\cdots, k$, there hold
\begin{equation}\label{e6:1}
\begin{aligned}
&(1+|y|^\eta)\left[H[\lambda,\xi,\dot{\lambda},\dot{\xi},\phi](\cdot, t)\right]_{\eta, B_{1}(0)}\chi_{B_{2R}(0)}(y) + \left|H[\lambda,\xi,\dot{\lambda},\dot{\xi},\phi](y, t)\right|\\
&\quad\quad\quad\quad\quad\quad\quad\quad\quad\quad\quad\quad\quad\quad\quad\quad\quad\quad\quad\quad\quad\quad\quad\quad\quad\lesssim t_0^{-\varepsilon}\frac{\mu_0^{n-2s+\sigma}}{1+|y|^{2s+a}}
\end{aligned}
\end{equation}
and
\begin{equation}\label{e6:2}
\begin{aligned}
&(1+|y|^\eta)\left[H[\phi^{(1)}](\cdot, t)-H[\phi^{(2)}](\cdot, t)\right]_{\eta, B_{1}(0)}\chi_{B_{2R}(0)}(y) + \left|H[\phi^{(1)}]-H[\phi^{(2)}]\right|(y, t)\\
&\quad\quad\quad\quad\quad\quad\quad\quad\quad\quad\quad\quad\quad\quad\quad\quad\quad\quad\quad\quad\quad\quad\quad\lesssim t_0^{-\varepsilon}\|\phi^{(1)} - \phi^{(2)}\|_{n-2s+\sigma, a}.
\end{aligned}
\end{equation}
From (\ref{e6:1}) and (\ref{e6:2}), $\mathcal{A}$ has a fixed point $\phi$ within the set of functions $\|\phi\|_{n-2s+\sigma, a}\leq ct_0^{-\varepsilon}$ for some large positive constant $c$. This proves the existence part of Theorem \ref{t:main}.

Estimate (\ref{e6:1}) is obtained from the definition of $H_j$, Lemma \ref{l2.2} and (\ref{e4:27}).
As for (\ref{e6:2}), from (\ref{e5:12}) and (\ref{e5:13}), we have
\begin{equation*}
\begin{aligned}
&\mu_{0j}^{\frac{n+2s}{2}}\left|S_{\mu_1, \xi_1, j}(\xi_{j,1} + \mu_{0j}y, t) - S_{\mu_2, \xi_2, j}(\xi_{j,2} + \mu_{0j}y, t)\right|\\
&\quad\quad\quad\quad\quad\quad\quad\quad\quad\quad\quad\quad\quad\quad\quad\quad\quad\quad\lesssim t_0^{-\varepsilon}\frac{\mu_0^{n-2s+\sigma}(t)}{1+|y|^{2s+a}}\|\phi^{(1)} - \phi^{(2)}\|_{n-2s+\sigma,a}
\end{aligned}
\end{equation*}
where
\begin{equation*}
\mu_i = \mu[\phi^{(i)}],\quad \xi_i = \xi[\phi^{(i)}],\quad \xi_{j, i} = \xi_j[\phi^{(i)}],\quad i = 1, 2.
\end{equation*}
By Proposition \ref{p4:4.2}, it holds that
\begin{eqnarray*}
\begin{aligned}
& p\mu_{0j}^{\frac{n-2s}{2}}\Bigg|\frac{\mu_{0j}^{2s}}{\mu_{j,1}^{2s}}U^{p-1}\left(\frac{\mu_{0j}}{\mu_{j,1}}y\right)\psi[\phi^{(1)}](\xi_{j,1} + \mu_{0j}y, t)\\
&\quad\quad\quad\quad\quad\quad\quad\quad\quad\quad\quad\quad\quad\quad\quad\quad -\frac{\mu_{0j}^{2s}}{\mu_{j,2}^{2s}}U^{p-1}\left(\frac{\mu_{0j}}{\mu_{j,2}}y\right)\psi[\phi^{(2)}](\xi_{j,2} + \mu_{0j}y, t)\Bigg|\\
&\quad\quad\quad\quad\quad\quad\quad\quad\quad\quad\quad\quad\quad\quad\quad\quad\quad\quad\lesssim t_0^{-\varepsilon}\frac{\mu_0^{n-2s+\sigma}(t)}{1+|y|^{2s+a}}\|\phi^{(1)} - \phi^{(2)}\|_{n-2s+\sigma,a}
\end{aligned}
\end{eqnarray*}
where
\begin{equation*}
\mu_{j,i} = \mu_j[\phi^{(i)}],\quad \psi[\phi^{(i)}] = \Psi[\lambda_i, \xi_i, \dot{\lambda}_i, \dot{\xi}_i, \phi^{(i)}],\quad i = 1, 2.
\end{equation*}
Finally, from the definitions \eqref{e3:11} and  \eqref{e3:12} in Section 3,
\begin{equation*}
\left|B_j[\phi^{(1)}_j]-B_j[\phi^{(2)}_j]\right|\lesssim t_0^{-\varepsilon}\frac{\mu_0^{n-2s+\sigma}(t)}{1+|y|^{2s+a}}\|\phi^{(1)} - \phi^{(2)}\|_{n-2s+\sigma,a}
\end{equation*}
and
\begin{equation*}
\left|B_j^0[\phi^{(1)}_j]-B_j^0[\phi^{(2)}_j]\right|\lesssim t_0^{-\varepsilon}\frac{\mu_0^{n-2s+\sigma}(t)}{1+|y|^{2s+a}}\|\phi^{(1)} - \phi^{(2)}\|_{n-2s+\sigma,a}
\end{equation*}
hold. This proves the estimate (\ref{e6:2}).

The stability part of Theorem \ref{t:main} is the same as \cite{cortazar2016green}, so we omit it.\qed

\section*{Acknowledgements}
J. Wei is partially supported by NSERC of Canada, Y. Zheng is partially supported by NSF of China
(11301374) and China Scholarship Council (CSC).


\begin{thebibliography}{10}

\bibitem{BG}
Agnid Banerjee and Nicola Garofalo,
\newblock Monotonicity of generalized frequencies and the strong unique continuation property for fractional parabolic equations, 
\newblock {\it Preprint}, 2017.



\bibitem{Barriosvaldinoci2014armawidder}
Begona Barrios, Ireneo Peral, Fernando Soria and Enrico Valdinoci,
\newblock A {W}idder's type theorem for the heat equation with nonlocal
  diffusion,
\newblock {\it Arch. Ration. Mech. Anal.}, 213(2):629--650, 2014.

\bibitem{BogdanTomaszRyznar2010}
Krzysztof Bogdan, Tomasz Grzywny and Michal Ryznar,
\newblock Heat kernel estimates for the fractional {L}aplacian with {D}irichlet
  conditions,
\newblock {\it Ann. Probab.}, 38(5):1901--1923, 2010.

\bibitem{BSV}
Matteo Bonforte, Yannick Sire and  Juan-Luis Vazquez, 
\newblock 
Optimal existence and uniqueness theory for the fractional heat equation,
\newblock {\it  Nonlinear Anal. 153, 142?168}, 2017.  

\bibitem{CabreSire2014Nonlinear}
Xavier Cabr\'e and Yannick Sire,
\newblock Nonlinear equations for fractional {L}aplacians, {I}: {R}egularity,
  maximum principles, and {H}amiltonian estimates,
\newblock {\it Ann. Inst. H. Poincar\'e Anal. Non Lin\'eaire}, 31(1):23--53,
  2014.

\bibitem{cabresire2015tamsnonlinear}
Xavier Cabr\'e and Yannick Sire,
\newblock Nonlinear equations for fractional {L}aplacians {II}: {E}xistence,
  uniqueness, and qualitative properties of solutions,
\newblock {\it Trans. Amer. Math. Soc.}, 367(2):911--941, 2015.

\bibitem{CabreRoquejoffreCMP}
Xavier Cabr\'e and Jean-Michel Roquejoffre,
\newblock The influence of fractional diffusion in {F}isher-{KPP} equations,
\newblock {\it Comm. Math. Phys.}, 320(3):679--722, 2013.

\bibitem{caffarellichanvasseurJams2011}
Luis Caffarelli, Chi~Hin Chan and Alexis Vasseur,
\newblock Regularity theory for parabolic nonlinear integral operators,
\newblock {\it J. Amer. Math. Soc.}, 24(3):849--869, 2011.

\bibitem{caffrarellifigallijram2013}
Luis Caffarelli and Alessio Figalli,
\newblock Regularity of solutions to the parabolic fractional obstacle problem,
\newblock {\it J. Reine Angew. Math.}, 680:191--233, 2013.

\bibitem{caffarellisoriavazquezjems2013}
Luis Caffarelli, Fernando Soria and Juan~Luis V\'azquez,
\newblock Regularity of solutions of the fractional porous medium flow,
\newblock {\it J. Eur. Math. Soc. (JEMS)}, 15(5):1701--1746, 2013.

\bibitem{caffarellivazquesarma2011}
Luis Caffarelli and Juan~Luis Vazquez,
\newblock Nonlinear porous medium flow with fractional potential pressure,
\newblock {\it Arch. Ration. Mech. Anal.}, 202(2):537--565, 2011.

\bibitem{caffarellivasseur2010annals}
Luis~A. Caffarelli and Alexis Vasseur,
\newblock Drift diffusion equations with fractional diffusion and the
  quasi-geostrophic equation,
\newblock {\it Ann. of Math. (2)}, 171(3):1903--1930, 2010.

\bibitem{chenwei2018}
Guoyuan Chen, Juncheng Wei and Yifu Zhou,
\newblock Finite time blow-up for fractional heat flow with critical exponent in $\mathbb{R}^4$,
\newblock {\it in preparation}.

\bibitem{ChenLiOucpam2006classification}
Wenxiong Chen, Congming Li and Biao Ou,
\newblock Classification of solutions for an integral equation,
\newblock {\it Comm. Pure Appl. Math.}, 59(3):330--343, 2006.

\bibitem{chenkimsongjems2010}
Zhen-Qing Chen, Panki Kim and Renming Song,
\newblock Heat kernel estimates for the {D}irichlet fractional {L}aplacian,
\newblock {\it J. Eur. Math. Soc. (JEMS)}, 12(5):1307--1329, 2010.

\bibitem{choikimjfa2014asymtotic}
Woocheol Choi, Seunghyeok Kim and Ki-Ahm Lee,
\newblock Asymptotic behavior of solutions for nonlinear elliptic problems with
  the fractional {L}aplacian,
\newblock {\it J. Funct. Anal.}, 266(11):6531--6598, 2014.

\bibitem{cortazar2016green}
Carmen Cortazar, Manuel del Pino and Monica Musso,
\newblock Green's function and infinite-time bubbling in the critical nonlinear
  heat equation,
\newblock {\it J. Eur. Math. Soc. (JEMS)}, to appear.


 \bibitem{DDDV} Juan D{\'a}vila, Manuel del Pino, S. Dipierro and E. Valdinoci,
 \newblock  Concentration phenomena for the nonlocal Schr\"odinger equation with Dirichlet datum,
 \newblock {\it  Anal. PDE}, 8 (2015), no. 5, 1165-1235.

\bibitem{DelPinoSirePAMS}
Juan D{\'a}vila, Manuel del Pino and Yannick Sire,
\newblock Nondegeneracy of the bubble in the critical case for nonlocal
  equations,
\newblock {\it Proc. Amer. Math. Soc.}, 141(11):3865--3870, 2013.

\bibitem{davila2017singularity}
Juan D{\'a}vila, Manuel del Pino and Juncheng Wei,
\newblock Singularity formation for the two-dimensional harmonic map flow into
  ${S}^2$,
\newblock {\it arXiv:1702.05801}, 2017.

\bibitem{del2012type}
Manuel del Pino, Panagiota Daskalopoulos and Natasa Sesum,
\newblock Type II ancient compact solutions to the Yamabe flow,
\newblock {\it J. Reine Angew. Math.}, to appear.

\bibitem{dkw2007concentration}
Manuel del Pino, Michal Kowalczyk and Juncheng Wei,
\newblock Concentration on curves for nonlinear {S}chr\"odinger equations,
\newblock {\it Comm. Pure Appl. Math.}, 60(1):113--146, 2007.

\bibitem{delwei2011Degiorgi}
Manuel del Pino, Michal Kowalczyk and Juncheng Wei,
\newblock On {D}e {G}iorgi's conjecture in dimension {$N\geq 9$},
\newblock {\it Ann. of Math. (2)}, 174(3):1485--1569, 2011.

\bibitem{delkowalczykweijdg2013entire}
Manuel del Pino, Michal Kowalczyk and Juncheng Wei,
\newblock Entire solutions of the {A}llen-{C}ahn equation and complete embedded
  minimal surfaces of finite total curvature in {$\Bbb R^3$},
\newblock {\it J. Differential Geom.}, 93(1):67--131, 2013.

\bibitem{delmussoweitype2}
Manuel del Pino, Monica Musso and Juncheng Wei,
\newblock Geometry driven Type II higher dimensional blow-up for the critical heat equation,
\newblock {\it arXiv:1710.11461}, 2017.

\bibitem{delmussowei3d}
Manuel del Pino, Monica Musso and Juncheng Wei,
\newblock Infinite time blow-up for the 3-dimensional energy critical heat equation,
\newblock {\it arXiv:1705.01672}, 2017.


\bibitem{DPV} Eleonora Di~Nezza, Giampiero Palatucci and Enrico Valdinoci,
\newblock Hitchhiker's guide to the fractional {S}obolev spaces,
\newblock {\it Bull. Sci. Math.}, 136(5):521--573, 2012.

\bibitem{felsingerkassmancpde2013}
Matthieu Felsinger and Moritz Kassmann,
\newblock Local regularity for parabolic nonlocal operators,
\newblock {\it Comm. Partial Differential Equations}, 38(9):1539--1573, 2013.

\bibitem{fernrosoton2016}
Xavier Fern\'andez-Real and Xavier Ros-Oton,
\newblock Boundary regularity for the fractional heat equation,
\newblock {\it Rev. R. Acad. Cienc. Exactas Fis. Nat. Ser. A Math. RACSAM},
  110(1):49--64, 2016.

\bibitem{FrankLenzmannSilvestre2016}
Rupert~L. Frank, Enno Lenzmann and Luis Silvestre,
\newblock Uniqueness of radial solutions for the fractional {L}aplacian,
\newblock {\it Comm. Pure Appl. Math.}, 69(9):1671--1726, 2016.

\bibitem{fujitajfsuts1966blowup}
Hiroshi Fujita,
\newblock On the blowing up of solutions of the {C}auchy problem for
  {$u_{t}=\Delta u+u^{1+\alpha }$},
\newblock {\it J. Fac. Sci. Univ. Tokyo Sect. I}, 13:109--124 (1966), 1966.

\bibitem{GK}
 Yoshikazu Giga and Robert V. Kohn,
 \newblock Asymptotically self-similar blow-up of semilinear heat equations,
\newblock {\it  Comm. Pure Appl. Math. 38, 297-319}, (1985). 

\bibitem{ishige}
Kotaro Hisa and Kazuhiro Ishige
\newblock Existence of solutions for a fractional semilinear
parabolic equation with singular initial data,
\newblock {\it Preprint}. 

\bibitem{jinxiongjram2014}
Tianling Jin and Jingang Xiong,
\newblock A fractional {Y}amabe flow and some applications,
\newblock {\it J. Reine Angew. Math.}, 696:187--223, 2014.

\bibitem{Lijems2004remarkconformally}
Yan~Yan Li,
\newblock Remark on some conformally invariant integral equations: the method
  of moving spheres,
\newblock {\it J. Eur. Math. Soc. (JEMS)}, 6(2):153--180, 2004.

\bibitem{matanomerlejfa2011threshold}
Hiroshi Matano and Frank Merle,
\newblock Threshold and generic type {I} behaviors for a supercritical
  nonlinear heat equation,
\newblock {\it J. Funct. Anal.}, 261(3):716--748, 2011.

\bibitem{merlezaagduke1997stability}
Frank Merle and Hatem Zaag,
\newblock Stability of the blow-up profile for equations of the type
  {$u_t=\Delta u+|u|^{p-1}u$},
\newblock {\it Duke Math. J.}, 86(1):143--195, 1997.

\bibitem{quittnersouplet2007superlinear}
Pavol Quittner and Philippe Souplet,
\newblock {\it Superlinear parabolic problems},
\newblock Birkh\"auser Verlag, Basel, 2007.

\bibitem{schweyerjfa2012typeii}
R\'emi Schweyer,
\newblock Type {II} blow-up for the four dimensional energy critical semi
  linear heat equation,
\newblock {\it J. Funct. Anal.}, 263(12):3922--3983, 2012.

\bibitem{servadeivaldinoci2015brezis}
Raffaella Servadei and Enrico Valdinoci,
\newblock The {B}rezis-{N}irenberg result for the fractional {L}aplacian,
\newblock {\it Trans. Amer. Math. Soc.}, 367(1):67--102, 2015.

\bibitem{silvestreium2012differentiability}
Luis Silvestre,
\newblock H\"older estimates for advection fractional-diffusion equations,
\newblock {\it Ann. Sc. Norm. Super. Pisa Cl. Sci. (5)}, 11(4):843--855, 2012.

\bibitem{silvestreium2012}
Luis Silvestre,
\newblock On the differentiability of the solution to an equation with drift
  and fractional diffusion,
\newblock {\it Indiana Univ. Math. J.}, 61(2):557--584, 2012.

\bibitem{sireweizhenghalf}
Yannick Sire, Juncheng Wei and Youquan Zheng,
\newblock Infinite time blow-up for half-harmonic map flow from $\mathbb{R}$ into $\mathbb{S}^1$,
\newblock {\it arXiv:1711.05387}, 2017.

\bibitem{sugitani}
S. Sugitani, 
\newblock On nonexistence of global solutions for some nonlinear integral equations,
\newblock {\it Osaka J. Math. 12, 45-51} (1975). 
\end{thebibliography}
\end{document}